\documentclass{amsart}
\usepackage[utf8]{inputenc}
\usepackage{graphicx}
\usepackage{a4wide}
\usepackage{color}
\usepackage{xcolor}
\usepackage{amsmath,amssymb}
\usepackage{float}
\usepackage{soul}
\usepackage{lipsum} 
\usepackage{hyperref}
\usepackage{cleveref}
\usepackage{bm}
\usepackage{enumitem}
\usepackage{float}
\usepackage{amsthm}
\usepackage{pgfplots}
\pgfplotsset{compat=1.18}
\usepackage{pgfplotstable}

\usepackage{comment}

\usepackage{xcolor} 
\definecolor{myblue}{HTML}{007BFF} \definecolor{myred}{HTML}{FF5733}  \newtheorem{problem}{Problem}
\newtheorem{theorem}{Theorem}[section]
\newtheorem{lemma}[theorem]{Lemma}
\newtheorem{example}{Example}

\newtheorem{condition}{Condition}

\newtheorem{prop}[theorem]{Proposition}
\newtheorem{remark}{Remark}
\newtheorem{corollary}{Corollary}
\usepackage{mathtools}
\usepackage{etoolbox}
\usepackage{bm}
\usepackage{amsmath}      

\usepackage{amssymb}      
\usepackage{algorithm}
\usepackage{algpseudocode}
\usepackage{fancybox}     
\usepackage{caption}
\usepackage{overpic}
\usepackage[foot]{amsaddr}

\numberwithin{example}{section} 

\DeclareFontFamily{U}{matha}{\hyphenchar\font45}
\DeclareFontShape{U}{matha}{m}{n}{
<-6> matha5 <6-7> matha6 <7-8> matha7
<8-9> matha8 <9-10> matha9
<10-12> matha10 <12-> matha12
}{}
\DeclareSymbolFont{matha}{U}{matha}{m}{n}

\DeclareFontFamily{U}{mathx}{\hyphenchar\font45}
\DeclareFontShape{U}{mathx}{m}{n}{
<-6> mathx5 <6-7> mathx6 <7-8> mathx7
<8-9> mathx8 <9-10> mathx9
<10-12> mathx10 <12-> mathx12
}{}
\DeclareSymbolFont{mathx}{U}{mathx}{m}{n}
\DeclareMathOperator*{\argmin}{arg\, min}
\DeclareMathDelimiter{\vvvert} {0}{matha}{"7E}{mathx}{"17}\newcommand{\domain}{\operatorname{dom}}
\DeclarePairedDelimiterX{\normiii}[1]
{\vvvert}
{\vvvert}
{\ifblank{#1}{\:\cdot\:}{#1}}
\newcommand{\rami}[1]{\textcolor{black}{#1}}

\newcommand{\R}{\mathbb{R}}

\newcommand{\dd}{\mathop{}\!\mathrm{d}}

\title[A locally-conservative proximal Galerkin method]{A locally-conservative proximal Galerkin \\ method for pointwise bound constraints}
\date{\today}
\author{Guosheng Fu$^1$}
\address{$^1$ Department of Applied and Computational Mathematics and Statistics (ACMS), University of Notre Dame, Notre Dame, IN 46556}
\email{gfu@nd.edu}
\author{Brendan Keith$^2$}
\address{$^2$ Division of Applied Mathematics, Brown University, Providence, RI 02912}
\email{brendan\_keith@brown.edu {\normalfont\it and}  
 rami\_masri@brown.edu}
\author{Rami Masri$^2$}
\thanks{
GF was supported in part by NSF DMS-2410740. 
BK and RM were supported in part by the U.S.\ Department of Energy Office of Science, Early Career Research Program under Award Number DE-SC0024335.}

\begin{document}

\begin{abstract}
We introduce the first-order system proximal Galerkin (FOSPG) method, a locally mass-conserving, hybridizable finite element method for solving heterogeneous anisotropic diffusion and obstacle problems.
Like other proximal Galerkin methods, FOSPG finds solutions by solving a recursive sequence of smooth, discretized, nonlinear subproblems.
We establish the well-posedness and convergence of these nonlinear subproblems along with stability and error estimates under low regularity assumptions for the linearized equations obtained by solving each subproblem using Newton's method.
The FOSPG method exhibits several advantages, including high-order accuracy, discrete maximum principle or bound-preserving discrete solutions, and local mass conservation.
It also achieves prescribed solution accuracy within asymptotically mesh-independent numbers of subproblems and linear solves per subproblem iteration.
Numerical experiments on benchmarks for anisotropic diffusion and obstacle problems confirm these attributes.
Furthermore, an open-source implementation of the method is provided to facilitate broader adoption and reproducibility.

\vspace{1em}
 \smallskip
  \noindent \textit{Key words}. 
  Anisotropic diffusion, obstacle problems, high order accuracy, local mass conservation, discrete maximum principle preserving 
  
  \smallskip 
   
  \noindent \textit{MSC codes.} 35J86, 49J40, 65N30. 
\end{abstract}

\maketitle 

\section{Introduction}
 We propose the first-order system proximal Galerkin (FOSPG) method for second-order, elliptic energy principles with pointwise bound constraints. The method, given below in \Cref{alg:mixed_PG}, is based on the proximal Galerkin method proposed in \cite{keith2023proximal} and features (i) high-order accuracy, (ii) bound preservation at every point in the computational domain (not only at nodal/quadrature points), and (iii) local mass conservation.
 To the best of our knowledge, this is the first finite element method with these three properties.
 
 We choose to focus our presentation on two categories of model problems: heterogeneous anisotropic diffusion
 and (unilateral/bilateral) obstacle problems.
 These two types of model problems have a multitude of applications, with the first appearing in reservoir simulations \cite{aavatsmark1998discretization} and plasma physics \cite{green2022efficient}.
 The second is vital for optimal control \cite{rodrigues1987obstacle}, topology optimization \cite{papadopoulos2021computing}, and glaciology \cite{greve2009dynamics}, among many other applications.
 Although much remains to be learned, we have found that FOSPG is highly accurate and efficient for solving these two categories of problems and so aim to promote its use in scientific applications involving them.

\subsection{Heterogeneous anisotropic diffusion}  The exact solutions of certain boundary value problems satisfy pointwise bound constraints almost everywhere in the domain.  This property is known as the weak maximum principle \cite{kinderlehrer2000introduction}.
A long-sought feature of numerical approximations is to satisfy this maximum principle while maintaining high-order accuracy.

Fundamental links have been established between discrete maximum principle (DMP) preserving first-order finite element methods and local mesh properties. This connection has been known since the seminal work of Ciarlet \cite{ciarlet1973maximum}, which identifies the non-obtuse angle condition as a key condition ensuring DMP preservation for \emph{isotropic} diffusion. Several relaxations of this condition have been established subsequently. For example, \cite{xu1999monotone} shows that Delaunay meshes with obtuse angles near diffusion discontinuities also preserve the DMP.

DMP-preserving local mesh conditions are more delicate to derive and difficult to enforce for anisotropic diffusion.
In particular, weakly acute angle conditions formulated with respect to the diffusion tensor need to be enforced on the internal dihedral angles of the mesh \cite{li2010anisotropic}. 
However, generating such anisotropy-aligned meshes can be very challenging or infeasible \cite{green2022efficient, green2024efficient}. 
Recently, {\it nonlinear} schemes satisfying the DMP without any mesh restrictions have been proposed \cite{BE05,BJKR18, ABP23,barrenechea2024nodally}. However, these schemes typically have low-order accuracy or enforce DMP only at nodal points \cite{BE05,barrenechea2024nodally}.
It is also unclear whether these approaches can preserve mass locally.
We refer readers to the recent review article \cite{barrenechea2024finite} for a more detailed account of DMP-preserving finite element methods and further references. 
\medskip

\paragraph{\bf Unique properties:}
FOSPG is DMP-preserving at every point in the computational domain, regardless of the polynomial degree employed.
Moreover, the method is locally mass-conservative and does not require anisotropy-aligned meshes.
We note that FOSPG requires a sequence of nonlinear solves to approximate the solution. However, in practice, inexact Newton solves (often involving fewer than three  linear solves per subproblems) are typically sufficient to obtain an accurate solution. Combined with the \textit{mesh--independent} property of FOSPG,  this demonstrates that the additional computational cost remains manageable. 
\subsection{Obstacle problems} 
Minimizing a quadratic energy functional over a closed, non-empty, convex set results in a variational \textit{inequality} \cite[Section 6.9]{ciarlet2013linear}. We refer to \cite{kinderlehrer2000introduction} for a comprehensive introduction.  The obstacle problem, modeling the equilibrium position of an elastic membrane lying above an obstacle, is the prototypical example in this problem class. 
Here, we present a brief overview of the standard finite element approaches to solving the obstacle problem; see \cite[Section 3]{keith2023proximal}, \cite{karkkainen2003augmented}, and \cite{gustafsson2017finite} for additional details.  

The most well-known and widely-used approach is the quadratic penalty method, which relaxes the bound-constrained optimization problem into an unconstrained one by adding terms to the energy functional that grow quadratically with the violation of the constraint \cite{scholz1984numerical,lions1969quelques}. However, to maintain the accuracy of high-order finite elements, the penalty parameter needs to scale suitably with the mesh size and polynomial degree, leading to mesh-dependent ill-conditioning \cite{gustafsson2017finite}. Alternatively, one can discretize the variational inequality directly using the primal formulation \cite{brezzi1977error} or the mixed formulation \cite{brezzi1978error}.
However, the low regularity of the Lagrange multiplier associated with the inequality constraint in these formulations affects the performance of numerical solution techniques \cite{hintermuller2006feasible}, often making penalty methods more desirable \cite{hintermuller2006path,adam2019semismooth}; see also \cite[Section~4.1]{farrell2020deflation}.
A different mixed formulation is also obtained by introducing a Lagrange multiplier, modeling the contact pressure \cite{gustafsson2017mixed,wohlmuth2011variationally}. For this formulation to be inf-sup stable, one enriches the finite element spaces with bubble functions or uses consistent stability terms \cite{gustafsson2017finite}. The latter option leads to another penalty method for piecewise linear elements or to a Nitsche-type penalty method when broken polynomial spaces are used for the multiplier \cite{gustafsson2017finite}.
We refer to \cite{gustafsson2017mixed} for analyzing the latter class of methods. 
 
Techniques from nonlinear programming are often employed to solve discretized variational inequalities.
Unfortunately, naive ``first-discretize-then-optimize" approaches typically lead to mesh-dependence; i.e., the number of nonlinear solves required for convergence grows indefinitely with mesh refinement \cite{hintermuller2002primal,hintermuller2006feasible,keith2023proximal}.
Of the most prominent techniques, we highlight the primal-dual active set (or semi-smooth Newton) method \cite{hintermuller2002primal} and the augmented Lagrangian method \cite{glowinski1989augmented}.
We also wish to highlight that multigrid methods often help accelerate convergence and diminish mesh-dependence \cite{bueler2024full}; we refer to \cite{graser2009multigrid} for a detailed overview.

It is well-known that one can enforce pointwise bound constraints, such as pointwise non-negativity, by constraining the nodal values of linear elements \cite{brezzi1977error}. However, enforcing pointwise constraints on higher order approximations is far more difficult \cite{kirby2024high}.
We refer to \cite[Section~3.2]{keith2023proximal} for an overview.
Proximal Galerkin methods overcome this challenge by constructing discrete solutions in the image of a bound-preserving monotone map.
The resulting algorithm consists of solving a sequence of smooth, semilinear PDE systems that couple standard finite element variables to a so-called latent variable in the domain of the nonlinear map \cite{keith2023proximal}.
Analogs of this algorithm have been extended to density-based topology optimization \cite{keith2024analysis,kim2024simple} and a multitude of other challenging problems such as elastic contact, variational fracture, gradient constraints, and obstacle-type quasi-variational inequalities \cite{dokken2023latent}.
Efficient preconditioners and $hp$-adaptive schemes for proximal Galerkin have also been proposed in \cite{papadopoulos2024hierarchical}, exhibiting $hp$-robustness in the number of Newton linear solves and up to 20x speed-ups over state-of-the-art methods.
In this work, we show that the essential features of the proximal Galerkin method are maintained for mass-conserving discretizations.~~~
\medskip

\paragraph{\bf Unique properties:}
FOSPG has a low iteration complexity and does not require mesh-dependent penalty parameters. It also provides a discrete solution that satisfies prescribed bound constraints at every point in the computational domain and achieves high-order accuracy.
Further, it delivers a locally-conservative flux approximation in every element not intersecting the obstacle contact zone. 

\subsection{Main contributions and outline}
The outline of this paper and its main contributions are summarized below. 
\begin{itemize}[leftmargin=*]
\item In \Cref{sec:model_pbs}, we state the anisotropic diffusion and obstacle problems that motivate this work.
Here, we demonstrate that the former problem can be formulated as a variational inequality, a property allowing us to unify our exposition. 
\item We review the latent variable proximal point (LVPP) algorithm in \Cref{sec:lvpp_infinite}. 
We then derive the first-order system proximal Galerkin (FOSPG) method from this algorithm; cf.\ \Cref{alg:main_alg_disc_mixed}.
We then prove the existence and uniqueness of solutions to the FOSPG subproblems and their stability properties in \Cref{thm:ext_mixed} and \Cref{lemma:stability}.  
\item
In \Cref{sec:hybridization}, we present a hybridized form of the FOSPG method, given in \Cref{alg:main_alg_disc}.
We advocate for this approach because, after static condensation, the linearized FOSPG subproblems in \Cref{alg:main_alg_disc} reduce to sparse, symmetric positive definite systems involving only the mesh facet unknowns.
This convex structure is particularly convenient for applying Newton's method to solve the subproblems.
\Cref{alg:main_alg_disc} is used in the practical implementation of our method.  
\item  
In the main theoretical result, \Cref{thm:conv_mixed_vi}, we prove convergence of \Cref{alg:main_alg_disc_mixed,alg:main_alg_disc} to the solution $u_h^*$ of a discrete variational inequality for simplicial elements with polynomial degree $p=0$ and quadrilateral elements for any $p \geq 0$. We then prove that FOSPG is locally mass-conserving on all elements where $u_h^*$ does not come in contact with the obstacle or bound constraint. For $p=0$, error estimates between the exact solution and the iterates of \Cref{alg:main_alg_disc} are also established.
\item \Cref{sec:linear_analysis} contains stability and convergence analysis of the linearized subproblems that result from applying Newton's method to every (nonlinear) proximal subproblem in FOSPG. This analysis is carried out for the hybridized method in~\Cref{alg:main_alg_disc} and is valid for any polynomial degree $p\geq 0$ on simplicial and quadrilateral meshes.
Notably, this section focuses on low-regular solutions in order to remain valid for heterogeneous anisotropic diffusion. 
\item  Finally, the performance of our method and its key features (high order accuracy, local mass conservation, and mesh independence) are demonstrated on a series of numerical examples in \Cref{sec:numerics}. We share our implementation, \href{https://github.com/ramimasri/FOSPG-first-order-system-proximal-Galerkin.git}{https://github.com/ramimasri/FOSPG-first-order-system-proximal-Galerkin.git}, in NGSolve \cite{schoberl2014c++} to facilitate broader adoption and reproducibility. 
\end{itemize}
\subsection{Notation}  In what follows, $\Omega$ is an open bounded polygonal or polyhedral domain in $\R^d$, $d \in \{2,3\}$. We use standard notation for the Lebesgue spaces $L^p(\Omega)$ for $p \in [1, \infty]$ and for the Sobolev--Hilbert spaces $H^m(\Omega)$ for $ m \in \mathbb{N}$. The space $H^{1/2}(\partial \Omega)$ is the standard trace space of $H^1(\Omega)$. 
For $g \in H^{1/2}(\partial \Omega)$, the subset $H^1_g(\Omega)$ is the closed subset of $H^1(\Omega)$ consisting of functions $u$ with trace $\mathrm{tr}(u) = g $ on $\partial \Omega$. Further, we use the notation $(\cdot, \cdot)$ to denote the $L^2(\Omega)$-inner product. The spaces $W^{s,p}(\Omega)$ $(s \geq 0, p \in [1,\infty]$) denote the standard Sobolev spaces.
We follow the convention that $0 \ln 0 = 0$ and define the essential domain of a proper function $f \colon \mathbb{R}^d \to \mathbb{R}\cup\{+\infty\}$ as
\[
    \operatorname{dom} f \coloneqq \{ x \in \mathbb{R}^d \mid f(x) < \infty \}
    \,.
\]

We consider a conforming shape regular partition  $\mathcal{T}_h$  of $\Omega$ into elements $T$. Denote by  $\mathcal{E}_h$ the set of facets $E$ (edges in 2D/ faces in 3D) of the partition $\mathcal{T}_h$, and 
denote by $\partial\mathcal{T}_h$ the set of all element boundaries $\partial T$ with outward normal $\bm{n}$. 
We assume that each element $T$ and facet $E$ are generated by an affine map 
$\Phi_T$ or $\Phi_E$ from a reference element $\hat{T}$ or $\hat{E}$, respectively. For quadrilateral elements, this assumption means that we only consider parallelograms in 2D and parallelotopes in 3D. Although this assumption can be relaxed in practical implementations. The diameters of an element $T$ and of a facet $E$ are denoted by $h_T$ and $h_E$ respectively. The mesh size is given by $h = \max_{T \in \mathcal{T}_h} h_T$. For constants $W$ and $Q$, we use the notation $W \lesssim Q$ whenever there is a constant $C$ independent of the mesh size $h$ such that $W \leq C Q $.  For a symmetric positive definite matrix $A$, we define $A^{1/2}$ as the matrix satisfying $A^{1/2}A^{1/2} = A$.  

The space $H^1(\mathcal{T}_h)$ denotes the broken $H^1$ space corresponding to the mesh $\mathcal{T}_h$: 
\[H^1(\mathcal{T}_h) = \{ u \in L^2(\Omega): \;\; u_{\vert_{T}}  \in H^1(T), \;\; \forall T \in \mathcal{T}_h \}. \] The broken gradient is denoted by $\nabla_h$, meaning that $(\nabla_h v)_{\vert_{T}} = \nabla (v_{\vert_T})$ for $v \in H^1(\mathcal{T}_h)$. We also use the broken divergence operator, $(\nabla_h \cdot v )_{\vert_T} =\nabla  \cdot (v_{\vert_T}) $.  
Further, for all $q, \varphi \in  L^2(\partial\mathcal{T}_h)$ and $v,w \in L^2(\Omega)$, we use the notation 
\begin{equation*}
 (q,\varphi)_{\partial \mathcal T_h} = \sum_{T \in \mathcal{T}_h} \int_{\partial T} q\varphi \dd s, \quad (v, w)_{\mathcal T_h} = \sum_{T \in \mathcal{T}_h} \int_{T} v w \dd x.    
\end{equation*}

\section{Model problems}\label{sec:model_pbs}
We have introduced FOSPG to target the two model problems given below. As we demonstrate, \Cref{problem:1} can be seen as a specific version of \Cref{problem:2}.
However, these two problems usually have different mathematical assumptions affecting their analyses.
\begin{problem}[Heterogeneous anisotropic diffusion]\label{problem:1}
Consider the following model problem:
\begin{subequations}\label{eq:model_problem}
\begin{alignat}{2}
-\nabla \cdot (A \nabla u) & = f  &&  \quad \mathrm { in }   \;  \Omega,  \\ 
u & = g && \quad  \mathrm{ on }  \;  \partial \Omega. 
\end{alignat}
\end{subequations}
Here, $A \in L^{\infty}(\Omega; \R^{d,d})$ is a symmetric diffusion tensor with eigenvalues that are uniformly lower and upper bounded by positive constants,  $f \in L^2(\Omega)$ and  $g \in H^{1/2}(\partial \Omega)$. The Lax--Milgram Theorem \cite[Section 25.2]{ern2021finite} provides the existence and uniqueness of the weak solution $u \in H^1_g(\Omega)$ to  \eqref{eq:model_problem} satisfying 
\begin{equation}
(A \nabla u, \nabla v) = (f,v) \quad \forall v \in H^1_0(\Omega).  \label{eq:weak_solution}
\end{equation}
\end{problem}
\begin{problem}[Obstacle problem] \label{problem:2}Solve for $u \in K \cap H^1_g(\Omega) $ such that
\begin{subequations}
\label{eq:MinimizationProblem}
\begin{equation}
 J(u) \leq J(v) \;\;~ \forall  v \in K \cap H^1_g(\Omega),  \quad J(v) = \frac{1}{2} (A \nabla v,  \nabla v) - (f,v),  \label{eq:min_pb}
\end{equation}  
where the closed and convex set $K$ is given by  
\begin{equation} \label{eq:max_principle}
    K
    =
    \{
        u \in H^1(\Omega)
        \mid
        \underline{u} \leq u \leq \overline{u}
        ~\mathrm{ a.e.\ in } \
        \Omega
    \}.
    \end{equation}
\end{subequations}
Here, $\underline u, \overline{u} \in  H^1(\Omega) \cap  C(\overline{\Omega})$ with  $ \underline{u} \leq \overline u$ a.e. in $\Omega$ and $\underline u \leq g \leq \overline{u}$ on $\partial \Omega$ and the above problem is called the bilateral obstacle problem. 
We also allow $\overline{u} = \infty$ to include the unilateral obstacle problem.  
To ensure that $K$ is nonempty, we assume that $\underline u \leq g \leq \overline u$ a.e on $\partial \Omega$. 
The existence and uniqueness of solutions follow from \cite[Theorem 2.1]{kinderlehrer2000introduction}. Further,
the above problem is equivalent to a variational inequality over $ K \cap H^1_g(\Omega)$, see \cite[Theorem 1.1.2]{ciarlet2002finite}: Solve for $u \in K \cap H^1_g(\Omega)$ such that 
\begin{equation}\label{eq:VI}
    (A \nabla u, \nabla (v-u))
    \geq
    (f,v-u)
    \quad \forall v \in K \cap H^1_g(\Omega). 
\end{equation}
\end{problem}
We now demonstrate that \Cref{problem:1} can be equivalently formulated in the setting of \Cref{problem:2} with constant $\underline u $ and $\overline{u}$. Indeed, if $f \in L^q(\Omega)$ for $q > d$, the weak solution  $u \in H_g^1(\Omega)$ of \eqref{eq:model_problem} satisfies the weak maximum principle \cite[Theorem 5.5 and Theorem B.2]{kinderlehrer2000introduction} 
\begin{equation}
   \| u\|_{L^{\infty}(\Omega)} \le \|f\|_{L^q(\Omega)} + \sup_{ \bm x \in \partial \Omega} | g |. 
\end{equation}
Further, if $f \geq 0$ a.e. in $\Omega$, then  we can also ensure that $u \geq \inf_{\bm x \in \partial \Omega} g$ a.e in $\Omega$ \cite[Theorem 5.7]{kinderlehrer2000introduction}. 
That is,  there exist constants $\underline{u}$ and $\overline{u}$ such that the weak solution $u$ belongs to the set $K\cap H^1_g(\Omega)$. In addition, $u$ solves \eqref{eq:VI} since $v - u \in H^1_0(\Omega)$ for any $v \in K \cap H^1_g(\Omega)$ and since $u$ satisfies \eqref{eq:weak_solution}. By uniqueness of solutions, $u$ is both the solution to \eqref{eq:VI} and to \eqref{eq:min_pb}.

In practice, the diffusion tensor in \Cref{problem:1} can have very low regularity.
Yet, the fact that the exact solution satisfies a variational \emph{equation} can simplify error estimates, cf.\ \Cref{cor:complete_story_p0} and \cite[Theorem~1]{kirby2024high}.
On the other hand, the diffusion tensor in~\Cref{problem:2} usually has higher regularity \rami{and} the obstacles $\underline{u},\; \overline{u} \colon \Omega \to \mathbb{R}\cup\{-\infty\}\cup\{+\infty\}$ are non-trivial functions.
However, in this case, we cannot depart from the variational \emph{inequality}~\eqref{eq:VI}.

\section{The latent variable proximal point algorithm} \label{sec:lvpp_infinite}
The Latent Variable Proximal Point (LVPP) algorithm (see \Cref{alg:main_alg}) was introduced in \cite{keith2023proximal}
to derive numerical methods for solving variational inequalities like~\eqref{eq:VI}.
A first-order system finite element discretization of this algorithm leads to the FOSPG method.
Here, we present a derivation of LVPP based in part on the exposition in \cite{dokken2023latent}.

\begin{algorithm}[htb]
\caption{The Latent Variable Proximal Point Algorithm}
\begin{algorithmic}[1]\label{alg:main_alg}
    \State \textbf{input:} Initial latent solution guess $\psi^0  \in L^\infty(\Omega)$, a sequence of positive step sizes $\{\alpha^k\}$, and a Carath\'eodory function $\Upsilon \colon \Omega \times \mathbb{R} \to \mathbb{R}$, where $\Upsilon(x,\cdot)$ is an increasing bijection from $\R \to (\underline u(x), \overline u(x))$ for a.e.\ $x\in \Omega$.
        \State Initialize \(k = 1\). 
    \State \textbf{repeat}
    \State \quad Solve the following (nonlinear) saddle-point problem:
    Find $ u  \in H^1_g(\Omega)   \text{ and } \psi \in L^{\infty}(\Omega)  $ such that 
    \begin{subequations}
     \label{eq:saddle_point_problem}
    \begin{alignat}{2}
        \alpha^k (A \nabla u , \nabla v) +  (\psi , v) &= \alpha^k (f, v) + (\psi^{k-1}, v)  && \quad \forall v \in H^1_0(\Omega), \\
    (u , \varphi) - (\mathcal{U} (\psi), \varphi) & = 0  &&  \quad \forall \varphi \in L^\infty(\Omega),
    \end{alignat}
    where $\mathcal{U}(\psi)(x) := \Upsilon(x,\psi(x))$ for a.e.\ $x\in \Omega$.
\end{subequations}
\State \quad Assign  \(\psi^{k} \gets \psi\) and \(k \gets k + 1\).
    \State \textbf{until} a convergence test is satisfied.
\end{algorithmic}
\end{algorithm}

\subsection{Superposition operators}
\Cref{alg:main_alg} relies on a Carath\'eodory function $\Upsilon \colon \Omega \times \mathbb{R} \to \mathbb{R}$, whose restrictions $\Upsilon(x,\cdot)$ are monotonically-increasing, invertible functions taking $\R$ to $(\underline u(x), \overline u(x))$.
For any sufficiently regular map $\psi \colon \Omega \to \mathbb{R}$, one can then generate a new function $\mathcal{U}(\psi) \colon \Omega \to \mathbb{R}$ via the expression $\mathcal{U}(\psi)(x) = \Upsilon(x,\psi(x))$.
The resulting operator $\mathcal{U}$ is called a superposition (Nemytskii) operator \cite{Appell_Zabrejko_1990,ambrosetti1995primer}.
In \Cref{sec:Discretization}, we will see that these operators allow us to construct approximate solutions to the minimization problem~\eqref{eq:MinimizationProblem} that are guaranteed to respect pointwise bound constraints.

One can select among different choices for the generating function $\Upsilon$.
The following two examples are appropriate for double obstacle and anisotropic diffusion problems:
\begin{subequations}
\begin{alignat}{4}
\label{eq:Z_1}
    \text{Example 1:}&&
    \qquad\qquad
    \Upsilon(x,z) &= \frac{\underline u(x) + \overline u(x) \exp z}{1+\exp z}
    \,;
    \\[3pt]
\label{eq:Z_2}
    \text{Example 2:}&&
    \qquad\qquad
    \Upsilon(x,z) &= \frac12 (\underline u(x) + \overline u(x)) + \frac12 (\overline u(x) - \underline u(x)) \frac{z}{ \sqrt{1+z^2}}
    \,.
\end{alignat}
\end{subequations}
Meanwhile, for the unilateral obstacle problem ($\overline{u } = \infty$), one may consider:
\begin{subequations}
\begin{alignat}{4}
\label{eq:U3}
    \text{Example 3:}&&
    \qquad\qquad
    \Upsilon(x,z) &= \underline u(x) + \exp z
    \,;
    \\[7pt]
\label{eq:U4}
    \text{Example 4:}&&
    \qquad\qquad
    \Upsilon(x,z) &= \underline u(x) + \ln(1+\exp z)
    \,.
    \qquad\qquad
    \qquad\qquad
    \quad\;
\end{alignat}
\end{subequations}
The behavior of the resulting superposition operators is illustrated in \Cref{fig:mathcalZ}.

\begin{figure}[t]
\begin{minipage}{0.47\textwidth}
    \centering    
\begin{overpic}[width=0.95\linewidth]{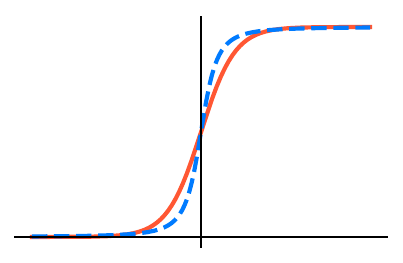} 
\put(58,16){\textcolor{myblue}{  \footnotesize{$\mathcal U(\psi) = \frac12 + \frac 12 \frac{\psi}{\sqrt{1+ \psi^2}}$}}} 
\put(5,35){\textcolor{myred}{ \footnotesize{$\mathcal U(\psi) = \frac{\exp(\psi)}{1+\exp(\psi) }$}}}
\put(96,1.1){$\psi$}
\put(48,64){$\mathcal{U}$}
\end{overpic}
\end{minipage}
\begin{minipage}{0.47\textwidth}
\begin{overpic} [width=0.95\linewidth]{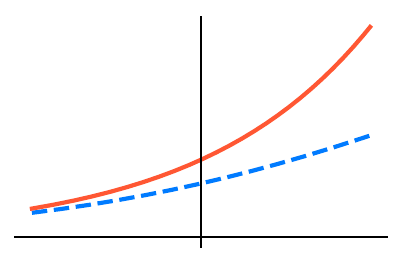} 
\put(58,16){\textcolor{myblue}{  \footnotesize{$\mathcal U(\psi) = \ln(1+ \exp(\psi)) $}}} 
\put(5,35){\textcolor{myred} {\footnotesize{$\mathcal{U}(\psi) = \exp(\psi) $}}}
\put(96,1.1){$\psi$}
\put(48,64){$\mathcal{U}$}
\end{overpic}
\end{minipage}
\caption{ Illustration of the different superposition operators $\mathcal{U}$. Left: $\underline{u} = 0$ and $\overline{u} = 1$.  Right: $\underline u = 0$ and $\overline{u} = \infty$. }
\label{fig:mathcalZ}
\end{figure}
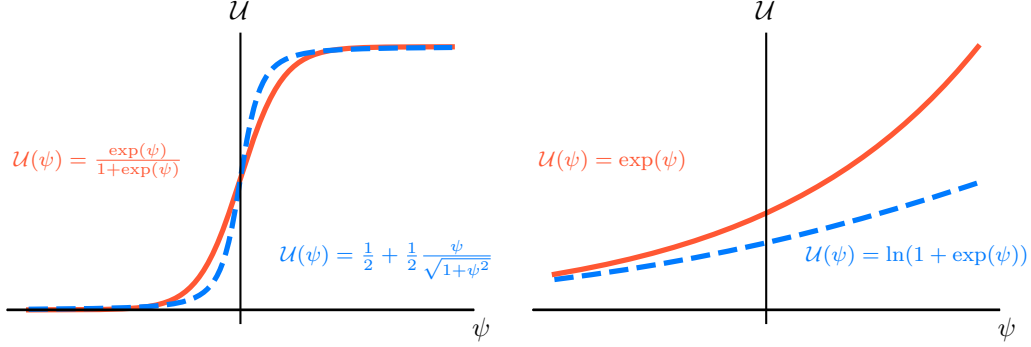

\subsection{The Bregman proximal point algorithm}
The main purpose of this section is to formally derive \Cref{alg:main_alg}.
As motivation, we recall the proximal operator \cite{doi:10.1137/0314056}, $P: H^1(\Omega) \rightarrow K \cap H_g^1(\Omega)$ and corresponding proximal point algorithm: Given $u^0 \in H^1(\Omega)$,  define
 \begin{equation} \label{eq:plain_proximal}
     u^{k} = Pu^{k-1} = \argmin_{v \in K \cap H^1_g(\Omega) } \left \{ J(v) 
     + \frac{1}{2\alpha^k} 
     \|v - u^{k-1}\|_{H^1(\Omega)}^2 
    \right \}
    \,,
     \end{equation} 
for each $k = 1,2,\ldots$

It is well known \cite{doi:10.1137/0314056} that the iterates $u^{k} = Pu^{k-1}$ converge to the unique fixed point of $P$, which coincides with the minimizer of~\eqref{eq:MinimizationProblem}. 
However, computing each iterate $Pu^{k-1}$ requires solving another variational inequality (Euler's inequality, see, e.g., \cite[Theorem 7.1-6]{ciarlet2013linear}): Find $u^k \in K \cap H^1_g(\Omega)$ such that 
\begin{multline}
   \frac{1}{\alpha^k} (\nabla (u^{k} -  u^{k-1}) , \nabla (v - u^k) ) +  \frac{1}{\alpha^k}(u^{k} -u^{k-1},v - u^k)  \\  + (A\nabla u^{k}, \nabla (v - u^k) )  \geq (f,v - u^k),    \label{eq:plain_proximal_opt} 
\end{multline}
for all $v \in K \cap H^1_g(\Omega)$. 
This renders the standard proximal point method~\eqref{eq:plain_proximal} computationally infeasible on this class of problems.
Indeed, one would prefer to directly approximate the solution of the original variational inequality \eqref{eq:VI} instead of solving a sequence of equally challenging problems like~\eqref{eq:plain_proximal_opt}.

As a remedy, one may consider a particular form of the Bregman proximal point algorithm \cite{keith2023proximal,teboulle2018simplified}. The key idea is to replace the squared $H^1(\Omega)$-norm in \eqref{eq:plain_proximal} by an alternative \textit{dissimilarity} function: 
\begin{align}
u^{k} = P u^{k-1} = \argmin_{v \in K \cap H^1_g(\Omega)} \left \{  J(v) + \frac{1}{\alpha^k} \int_{\Omega} \mathcal{D}(v,u^{k-1}) \dd x \right \}.  \label{eq:Bregman_proximal_map}  
\end{align}
Here, the so-called Bregman divergence $\mathcal{D}(\cdot,\cdot)$ is derived from the error in the first-order Taylor's expansion of a superposition operator denoted $\mathcal{R}(u)$.
This operator is generated by a Carath\'eodory function $R \colon \Omega \times \mathbb{R} \to \mathbb{R}\cup\{+\infty\}$ that is strictly convex in its second argument.
More specifically, for a.e.\ $x \in \Omega$, we assume that $R(x,y)$ maps $[\underline{u}(x), \overline{u}(x)]$ into $\R$ and that $y \mapsto R(x,y)$ is strictly convex and differentiable on the open interval $(\underline u(x) , \overline{u}(x))$.
Denoting the corresponding superposition operator $\mathcal{R}(u)(x) = R(x,u(x))$ and its derivative $\mathcal{R}'(u)(x) = \frac{\partial}{\partial_y} R(x,u(x))$ \rami{for $u \in \operatorname{dom}(\mathcal R') = \{ u \in  L^\infty(\Omega) \mid  \operatorname{ess} \inf (u - \underline u) > 0 \text{ and } \operatorname{ess} \sup (u - \overline{u} ) < 0$\}}, we express the Bregman divergence $\mathcal{D}(\cdot,\cdot)$ as follows \cite{BREGMAN1967200}:
\begin{equation}\label{eq:D_formal}
\mathcal{D}(v,u^{k-1}) =   \mathcal{R}(v) - \mathcal{R}(u^{k-1})  -  \mathcal{R}'(u^{k-1}) (v -u^{k-1})\,.
\end{equation}

\rami{The algorithm given by \eqref{eq:Bregman_proximal_map} is well defined \cite[Theorem 4.7]{keith2023proximal} and  converges to the true solution $u$ of the VI. These properties were established for the unilateral obstacle problem ($\overline{u} = +\infty$) in \cite[Theorem 4.13]{keith2023proximal} under additional assumptions: the source term $f \in L^{\infty}(\Omega)$, the boundary data $g \in H^1(\Omega) \cap C(\overline \Omega)$, and obstacle $\phi \in H^1(\Omega) \cap C(\overline \Omega)$ with $\Delta \phi \in L^\infty(\Omega)$. In fact, the following estimate holds for all $k$: 
\begin{equation*}
\|\nabla u - \nabla u^k\|_{L^2(\Omega)}^2 \leq \frac{2 } {\sum_{\ell =1 }^k \alpha^\ell} \int_{\Omega } \mathcal{D}(u - \underline u, u^0 - \underline u) \mathrm{d} x, 
\end{equation*} 
where $u^0 = \mathcal{U}(\psi^0)$. 
We note that $u$ is not a fixed point of \eqref{eq:Bregman_proximal_map} since $u \notin \mathrm{dom} (\mathcal{R}') = \{ u \in L^\infty(\Omega) \mid \operatorname{ess} \inf(u - \underline u) > 0 \}$ and hence $\mathcal{D}(v, u)$ is not well defined. The proof of well-posedness and convergence of (11) for the double obstacle problem follows the same proof techniques presented in \cite{keith2023proximal} with some details carried out in \cite[Theorem 3.4]{keith2024analysis}. Finally, we note that the analysis put forth in this paper does not rely on the well-posedness of the continuous level \Cref{alg:main_alg} and \eqref{eq:Bregman_proximal_map}. All our results directly handle the discrete  \Cref{alg:main_alg_disc,alg:main_alg_disc_mixed}. This allows us to dispense with the assumptions of \cite{keith2023proximal,keith2024analysis}.}

\subsection{Further properties}
For technical reasons, we require $R(x,\cdot)$ to have the following supercoercivity property at a.e.\ $x\in\Omega$:
\begin{equation}
\label{eq:Supercoercivity}
    R(x,y)/|y| \rightarrow \infty
        ~~\text{ as }
    |y| \rightarrow \infty
    \,.
\end{equation}
This ensures that the convex conjugate of $R(x,\cdot)$ is well-defined over all of $\R$ \cite{bauschke1997legendre,dokken2023latent}, see \eqref{eq:convex_conj_S} below. We also require that $\domain R(x,\cdot) = [\underline u(x), \overline{u}(x)]$ to ensure convergence of the iterates~\eqref{eq:Bregman_proximal_map} to the unique minimizer of $J$; cf.\ \Cref{thm:conv_mixed_vi} below. 
\Cref{cond:key_assump} is then sufficient to derive \Cref{alg:main_alg} from \eqref{eq:Bregman_proximal_map}.
\rami{
\begin{example}[Fermi--Dirac entropy]  \label{example:fermi_Dirac_P1}
Consider the Carath\'eodory function given by
\begin{subequations}
\begin{equation}
\label{eq:FermiDirac}
    R(x,y)  
    = (y-\underline u(x)) \ln (y - \underline u(x)) + (\overline{u}(x) - y )\ln (\overline{u}(x) - y)
                \end{equation}
if $y \in [\underline{u}(x), \overline{u}(x)]$ and $R(x,y) = +\infty$ otherwise.
The corresponding superposition operator,
\begin{equation}
\label{eq:FermiDirac_superposition}
    \mathcal{R}(u)  
    = (u-\underline u) \ln (u - \underline u) + (\overline{u} - u )\ln (\overline{u} - u)
    \,,
    \quad u\in K
    \,,
\end{equation}
also known as the (generalized) Fermi--Dirac entropy,
is continuous over the feasible set $K$.
Moreover, if $u \in L^{\infty}(\Omega)$ with $\operatorname{ess\,inf} (\underline u -  u) > 0$ and $\operatorname{ess\,inf} (u - \overline{u}) > 0 $, then
\begin{equation}
    \mathcal{R}'(u) = \ln(u - \underline u ) - \ln (\overline{ u} - u)  
\end{equation}
belongs to $L^\infty(\Omega)$. 
\end{subequations} 
Note that \eqref{eq:Supercoercivity} is satisfied since $R(x,y)/|y| = +\infty$ for any $y \notin [\underline{u}(x), \overline{u}(x)]$ by the properties of extended arithmetic \cite[Chapter 1]{rockafellar1998variational}. 
\end{example} 
}
\begin{condition}\label{cond:key_assump} For any given $\varphi \in C^\infty_c(\Omega)$, there exists a constant  $\delta > 0 $ such that the solution $u^{k}$ to \eqref{eq:Bregman_proximal_map} satisfies 
\begin{equation} \label{eq:key_point}
    u^{k} \pm \delta \varphi \in K \cap H^1_g(\Omega).
\end{equation} 
\end{condition}
\noindent This condition can be realized by choosing $R$ whose derivatives are sufficiently singular at $\underline u$ and $\overline{u}$ and ensuring that the functions $f,\underline{u},\overline{u}$ are sufficiently regular.
In particular, if $f \in L^\infty(\Omega)$ and $\underline{u},\overline{u} \in \mathbb{R}$, then the operators in \Cref{example:fermi_Dirac,example:obstacle} imply~\Cref{cond:key_assump} \cite{keith2023proximal}.

\begin{example}[Shannon entropy] \label{example:obstacle}
 The derivation of \Cref{alg:main_alg} was first given in \cite{keith2023proximal} for the choice
\begin{align}
\mathcal{R}(u) = u \ln u - u , 
\end{align}
leading to the (extended) Kullbach--Liebler divergence
\[
    \mathcal{D}(v,u) = v\ln(v/u) - v + u
    \,.
\]
We refer to \cite{keith2023proximal} for further properties of this setting.
\end{example}

\subsection{From inequality to equality}
The solution of \eqref{eq:Bregman_proximal_map} satisfies the following variational inequality: For all $ v \in K \cap H^1_g(\Omega)$,  
\begin{equation}\label{eq:bregman_primal}
\frac{1}{\alpha^k} ( \mathcal{R}'(u^{k}) - \mathcal{R}'(u^{k-1}) , v - u^k  ) + (A\nabla u^{k}, \nabla (v - u^k) )  \geq (f,v - u^k).
\end{equation}
However, \Cref{cond:key_assump} allows us to transform this variational inequality \eqref{eq:bregman_primal} to a variational \textit{equation}. Indeed, testing \eqref{eq:bregman_primal} with $v = u^{k} \pm \delta \varphi$ and rescaling by $\delta$, we find that   
\begin{equation} \label{eq:after_applying_cond}
  \frac{1}{\alpha^k} (  \mathcal{R}'(u^{k}) - \mathcal{R}'(u^{k-1}) , \varphi  ) + (A \nabla u^{k}, \nabla \varphi )  =  (f, \varphi)  \quad \forall \varphi \in C^{\infty}_c(\Omega) .   
\end{equation}

\subsection{The latent variable}
We can now formally introduce the latent variable  
\begin{equation}\label{eq:latent0}
  \psi^{k} = \mathcal{R}'(u^k) \quad\iff\quad  ( \mathcal{R}' )^{-1}(\psi^k) = u^k, 
\end{equation}
where $\mathcal{R}^\prime$ is invertible because $\mathcal{R}$ is strictly convex.
Utilizing the sequence of latent variables $\psi^k$ given by \eqref{eq:latent0} and the density of $C^\infty_c(\Omega)$ in $H_0^1(\Omega)$, we arrive at the latent variable proximal point method:
For each $k =1,2,\ldots$, solve for $(u^{k} , \psi^{k})\in H^1_g(\Omega) \times L^\infty(\Omega)$ satisfying 
\begin{subequations}\label{eq:bregman_latent}
    \begin{alignat}{2}
 \frac{1}{\alpha^k} (\psi^{k} - \psi^{k-1}, v) + (A \nabla u^{k}, \nabla v) & = (f,v) && \quad \forall v \in H^1_0(\Omega), \label{eq:bregman_latent_0} \\ 
 (u^{k}, \varphi) - (( \mathcal{R}' )^{-1}(\psi^{k}), \varphi) & = 0 && \quad \forall \varphi \in L^\infty(\Omega).  
 \label{eq:bregman_latent_1}
\end{alignat}
\end{subequations}
\Cref{alg:main_alg} is recovered from~\eqref{eq:bregman_latent} by setting $\mathcal{U} = ( \mathcal{R}' )^{-1}$, multiplying \eqref{eq:bregman_latent_0} by $\alpha^k$, and rearranging terms.
We note that $(\psi^{k} - \psi^{k-1})/\alpha^k$ acts as an approximation to the Lagrange multiplier for the bound constraints in~\eqref{eq:max_principle} \cite{keith2023proximal}.
It can also be interpreted as a discretized rate of change of the latent gradient flow $\frac{\partial}{\partial\alpha} \psi = -\nabla J(u)$ over the manifold defined by~\eqref{eq:latent0}.

\begin{example}[Fermi--Dirac entropy, part 2]  \label{example:fermi_Dirac_P2}
Recall the setting of \Cref{example:fermi_Dirac_P1} and observe that
\[
    (\mathcal{R}')^{-1}(\psi)
    =
   \frac{\underline u + \overline u \exp \psi}{1+\exp \psi}
   \,.
\]
In turn, notice that $\mathcal{U}(\psi) = ( \mathcal{R}' )^{-1}(\psi)$ coincides with the superposition operator generated by~\eqref{eq:Z_1}.
\end{example}

\subsection{Duality}
As evident by the derivation above, the generating function $R$ is closely related to the generating function $\Upsilon$ in \Cref{alg:main_alg}. This relation is made precise in \Cref{prop:Conjugate} below, which involves the definition of the convex conjugate of $R(x,\cdot)$:
\begin{subequations}
\begin{equation}\label{eq:convex_conj_S}
    R^*(x,z)
    = \sup_{y \in \mathbb{R}}  \big\{ zy - R(x,y) \big\}
    \,.
\end{equation}
Likewise, we define the associated superposition operator
\begin{equation}\label{eq:convex_conj_S_superposition}
    \mathcal{R}^*(u)(x)
    =
    R^*(x,u(x))
    \,.
\end{equation}
\end{subequations}
We refer the interested reader to \cite{bauschke2011convex} for more general and abstract results.
\begin{prop}
\label{prop:Conjugate}
For a smooth and strictly increasing bijection $\Upsilon(x,\cdot):\R \rightarrow (\underline u(x), \overline{u}(x))$, there exists a strictly convex proper function $R(x,\cdot)\colon \mathbb{R} \rightarrow \R\cup\{+\infty\}$ such that
\begin{equation} \label{setting:main}
    \Upsilon(x,z)
    = 
    \left( \frac{\partial}{\partial y} R(x,\cdot) \right)^{-1}(z) = \frac{\partial}{\partial z}R^*(x,z)
\end{equation}
for all $z \in (\underline u(x), \overline{u}(x))$.
Moreover, $R$ satisfies~\eqref{eq:Supercoercivity}, and the corresponding superposition operators satisfy 
\begin{equation}
\label{setting:main_superposition}
\mathcal{U}(\psi) = (\mathcal{R}^*)^\prime(\psi) \quad \forall \psi \in L^{\infty}(\Omega). 
\end{equation}

\end{prop}
\begin{proof}[Sketch of proof]
The first equality in \eqref{setting:main} follows from the observations that (i) the inverse of a smooth, strictly increasing monotone function is also a smooth, strictly increasing monotone function, and (ii) any smooth, strictly increasing monotone function can be written as the derivative of a convex function. The second equality in \eqref{setting:main} follows from the definition of $R^*$.
\Cref{setting:main_superposition} follows readily from \eqref{setting:main} and the definitions of $\mathcal{U}$ and $\mathcal{R}$.
Meanwhile, $R = (R^\ast)^\ast$ satisfies~\eqref{eq:Supercoercivity} due to \cite[Proposition~2.16]{bauschke1997legendre}.
\end{proof}

\begin{example}[Fermi--Dirac entropy, part 3]  \label{example:fermi_Dirac}
We again recall the setting of \Cref{example:fermi_Dirac_P1}.
Denoting 
$$\Upsilon(x,z) = \left( \frac{\partial}{\partial y} R(x,\cdot) \right)^{-1}(z) = \frac{\underline u(x) + \overline u(x) \exp z}{1+\exp z},$$
we recover the function presented in \eqref{eq:Z_1}. Further, one can readily derive the expression for $R^*$:
\begin{equation} \label{eq:convex_conjugate}
    R^*(x,z) = z \Upsilon(x,z) -  R(x,\Upsilon(x,z)).
\end{equation}
\end{example}

\section{Spatial Discretization} 
\label{sec:Discretization}
We employ a locally conservative discretization of \eqref{eq:saddle_point_problem}. As a starting point, consider the first-order system formulation of the PDE \eqref{eq:saddle_point_problem} by introducing the flux variable $\bm{q}^k=-A \nabla u^k$:
\begin{subequations} 
\label{eq:alpha_step_1st}
\begin{alignat}{2}
  \alpha^k \nabla \cdot \bm q^{k}  +  \psi^{k}                 & = \alpha^k f + \psi^{k-1}  &&  \quad \text{in } \Omega,    \\   
A^{-1}\bm q^{k} +  \nabla u^{k} &  = 0 && \quad \text{in } \Omega, \\ 
u^{k}  - \mathcal{U}(\psi^{k})& = 0   &&\quad  \text{in } \Omega. 
\end{alignat}
\end{subequations} 
\subsection{Finite element spaces}
In this subsection, we introduce the discrete spaces.  
On the reference element $\hat{T}$, we denote by $\mathcal{P}_p(\hat T)$ either the space of all polynomials of degree at most $p$   ($\mathbb P_p(\hat T)$) if $\hat T$ is simplicial or the space
of polynomials of degree at most $p$ in each direction ($\mathbb{Q}_p(\hat T))$ if $\hat T$ is a quadrilateral or parallelotope.    
We also use the Raviart--Thomas element \cite{RaviartThomas77} of degree $p$ on $\hat T$, 
$\mathrm{RT}_p(\hat{T}):=[\mathcal{P}_p(\hat{T})]^d+  \mathbf{x}\cdot\mathcal{P}_{p}(\hat{T})$. 

We consider the following finite element spaces:
\begin{subequations}
    \label{fes}
    \begin{align}
       \bm  \Sigma_h^p = &\;\left\{
        \mathbf{r}_h\in [L^2(\Omega)]^d: \quad \mathbf{r}_h{}_{|_{T}}
        = \frac{1}{\mathrm{det} \Phi'_T}\Phi_T'\hat{\mathbf{r}} 
        \circ\Phi_T^{-1},\;
        \hat{\mathbf{r}} \in \mathrm{RT}_p(\hat T), \; \forall T \in \mathcal{T}_h
        \right\},\\
        V_h^p = &\;\left\{
        {v}_h\in L^2(\Omega): \quad \quad {v}_h{}_{|_{T}}
        = \hat{{v}} 
        \circ\Phi_T^{-1},\;
        \hat{{v}} \in \mathcal{P}_p(\hat T),\; \forall T \in \mathcal{T}_h
        \right\}.
    \end{align}
\end{subequations}
The jump of $v$ on a face $E \in \mathcal{E}_h$ is defined as 
\[
[v_h]_{\vert_E} = v_h \vert_{T_E^1} - v_h \vert_{T_E^2},  \]
where $E = \partial T_{E}^1 \cap \partial T_{E}^2$ and the normal $\bm n_E$ is chosen to point from $T_E^1$ to $T_E^2$. This choice is arbitrary. If $E \subset \partial \Omega$, then $[v_h]_{\vert_E}$ is taken as the single valued trace of $v_h$. We drop the subscript ``$\vert_{E}$" to simplify notation. The jumps of vector-valued functions are defined similarly. 

We define the $H(\mathrm{div}; \Omega)$ conforming Raviart--Thomas space,
\begin{equation}
\bm \Sigma_{h,\mathrm{div}}^{p} = \bm \Sigma_h^p \cap H(\mathrm{div};\Omega) = \{\bm r_h \in \bm \Sigma_h^p: \quad [\bm r_h] \cdot \bm n_E =0 \,\,\,\, \forall E \in \mathcal{E}_h \backslash \partial \Omega \}, 
\end{equation}
and
we consider the standard discontinuous Galerkin (DG) norm 
\[
\|v\|_{\mathrm{DG}}^2 = \sum_{T \in \mathcal{T}_h} \|A^{1/2} \nabla v\|_{L^2(T)}^2 + \sum_{E \in \mathcal{E}_h} h_E^{-1} \|[v]\|^2_{L^2(E)}  \quad \forall v \in H^1(\mathcal{T}_h). 
\]
\rami{Note that since boundary facets are included in the definition of $\mathcal{E}_h$, the above defines a norm.} We also recall the following Poincar\`e inequality \cite[Lemma 3.2]{lasis2007hp}, which holds for all $p \in [1,6]$ when $d = 3$ and for all $p \in [1,\infty)$ when $d=2$, 
\begin{equation}\label{eq:Poincare}
\|v_h\|_{L^p(\Omega)} \lesssim \|v_h\|_{\mathrm{DG}} \quad \forall v_h \in V_h^p. 
\end{equation}
In what follows, we will make use of the linear lifting operators  $\bm L: V_h^p \rightarrow \bm \Sigma_{h,\mathrm{div}}^{p}$ and $\bm L_\Gamma: L^2(\partial \Omega) \rightarrow \bm \Sigma_{h,\mathrm{div}}^{p}$ defined such that 
\begin{alignat}{2}
    (A^{-1} \bm L(u) , \bm r) &= (u, \nabla \cdot \bm r)  && \quad  \forall \bm r  \in \bm \Sigma_{h,\mathrm{div}}^{p}, \label{eq:first_lift} \\ 
     (A^{-1} \bm L_\Gamma (g) , \bm r)& = - \langle g, \bm r \cdot \bm n\rangle_{\partial \Omega} && \quad  \forall \bm r  \in \bm \Sigma_{h,\mathrm{div}}^{p}. \label{eq:second_lift} 
\end{alignat}   
\begin{lemma}\label{lemma:dg_to_L2}
Let $u_h \in V_h^p$ and define $\bm q_h = \bm L(u_h)$ by \eqref{eq:first_lift}. Then, 
\begin{equation}
\rami{c_L \|A^{-1/2} \bm q_h\|_{L^2(\Omega)}}  \leq \|u_h\|_{\mathrm{DG}} \leq C_{L} \|A^{-1/2} \bm q_h\|_{L^2(\Omega)},
\end{equation}
where \rami{$c_L$} and $C_L$ \rami{are mesh-independent positive constants}. 
\end{lemma}
\begin{proof}
The proof \rami{ of the second inequality} follows identically to that of \cite[Theorem 3.2]{gao2018error} for simplicial meshes.
See also a similar result in \cite[Theorem 2.3]{CFQ18}, which covers the parallelotope mesh case. \rami{For the first bound, 
we use the definition and integrate by parts locally on each element: 
\begin{align} \label{eq:second_bound_lifting}
\|A^{-1/2} \bm q_h\|_{L^2(\Omega)}^2 = (A^{-1} \bm L(u_h), \bm L(u_h)) = (u_h , \nabla \cdot \bm q_h) = (\nabla_h u_h , \bm q_h) -  \sum_{E \in \mathcal{E}_h } \int_{E} \bm q_h \cdot \bm n_E [u_h] \dd s. 
\end{align}
For a fixed edge, we employ the trace inequality 
\begin{equation}
    \| \bm q_h \cdot \bm n_E\|_{L^2(E)} \leq C_{\mathrm{tr}} h_E^{-1/2}\|\bm q_h\|_{L^2(T_E)},  \label{eq:normal_trace_discrete}
\end{equation} 
for some element $T_E \in \mathcal{T}_h$ sharing the facet $E$. Applying \eqref{eq:normal_trace_discrete}, shape regularity of the mesh, and Cauchy-Schwarz inequality, we bound  
\begin{align*}
\sum_{E \in \mathcal{E}_h}
\int_E \bm q_h \cdot \bm n_E [u_h] 
& \leq  \left( \sum_{E \in \mathcal{E}_h} h_E \| \bm q_h \cdot \bm n_E \|^2_{L^2(E)} \right)^{1/2} \left( \sum_{E \in \mathcal{E}_h } h_E^{-1} \|[u_h]\|^2_{L^2(E)}  \right) \\ 
& \lesssim C_\mathrm{tr} \|\bm q_h\|_{L^2(\Omega)} \|u_h\|_{\mathrm{DG}}, 
\end{align*}
Along with Cauchy-Schwarz inequality for the first term in \eqref{eq:second_bound_lifting}, the above arguments show that
\[ 
    \|A^{-1/2} \bm q_h\|_{L^2(\Omega)}^2  \leq C_{\mathrm{tr}}\|\bm q_h\|_{L^2(\Omega)}
    \|u_h\|_{\mathrm{DG}}. 
\]
We conclude the result by noting that $\|\bm q_h\|_{L^2(\Omega)} = \| A^{1/2} A^{-1/2} \bm q_h\|_{L^2(\Omega)} \leq \|A^{1/2}\|_{L^\infty(\Omega)} \|A^{-1/2} \bm q_h\|_{L^2(\Omega)}$. 
}
\end{proof}
We also introduce the $L^2$-projection, $\Pi_h: H^1(\mathcal{T}_h) \rightarrow V_h^p$,  and we recall the following properties:
\begin{equation}
    (\Pi_h v, q_h) = (v, q_h) \,\,\,\, \forall q_h \in V_h^p,  \qquad \|\Pi_h v\|_{\mathrm{DG}} \lesssim \|v\|_{\mathrm{DG}}.  \label{eq:local_L2_project}
\end{equation}
\rami{The proof of the stability bound above is standard and follows from the properties of $\Pi_h$ on elements and facets, see, e.g., \cite[Lemma 1.58 and Lemma 1.59]{di2011mathematical}.\footnote{\rami{
For any $v \in H^1(\mathcal{T}_h)$, we have that $|\Pi_h v|_{H^1(T)} \lesssim |v|_{H^1(T)}$ \cite[Lemma 1.58]{di2011mathematical}. In addition, for any $E \in \mathcal{E}_h$ with $E =  \partial T_E^1 \cap \partial  T_E^2$ for some $T_E^1, T_E^2 \in \mathcal{T}_h$, we have that \cite[Lemma 1.59]{di2011mathematical}
$$ \|[v - \Pi_h v]\|_{L^2(E)} \leq \|(v - \Pi_h v) \vert_{T_E^1}\|_{L^2(E)} + \|(v - \Pi_h v) \vert_{T_E^2}\|_{L^2(E)} \lesssim  h^{1/2}_{T_E^1}|v|_{H^1(T_E^1)} + h_{T_E^2}^{1/2} |v|_{H^1(T_E^2)}. $$ 
A similar estimate holds for $E \subset \partial \Omega$. Summing over mesh elements and over facets, using the shape regularity of the mesh, and applying the triangle inequality, shows the required estimate. 
}
}
}
\subsection{The first-order system proximal Galerkin method} \label{sec:mixed_lvpp}
We are now ready to present the mixed finite element discretization of \eqref{eq:saddle_point_problem}.  We start with defining the discrete solutions $(\bm q_h^k, u_h^k) \in \bm \Sigma_{h,\mathrm{div}}^{p} \times V_h^p$ and $\psi_h^k \in V_h^p$  satisfying the Galerkin discretization of \eqref{eq:alpha_step_1st}: 
\begin{subequations}\label{eq:nonlinear_subproblem_mixed}
    \begin{alignat}{2}
  \alpha^k (\nabla \cdot \bm q_h^{k}, v_h) +  (\psi_h^{k} , v_h) & = \alpha^k (f,v_h) + (\psi_h^{k-1}, v_h)  && \quad \forall v_h \in V_h^p, \label{eq:nonlinear_subproblem_mixed_0} \\ 
   (A^{-1} \bm q_h^{k}, \bm v_h) - (\nabla \cdot \bm v_h, u^k_h ) & = - \langle g , \bm v_h \cdot \bm n\rangle_{\partial \Omega} && \quad \forall \bm v_h \in\bm \Sigma_{h,\mathrm{div}}^{p}, \label{eq:nonlinear_subproblem_mixed_1}
  \\
   (u_h^{k}, q_h)  - (\mathcal{U}(\psi_h^{k}) + \mathcal S(\psi_h^k), q_h) &  = 0  && \quad \forall q_h \in V_h^p. \label{eq:nonlinear_subproblem_mixed_2}
 \end{alignat}
 \end{subequations}
In the above, $\mathcal S$ is an \emph{optional} stabilization/regularization term allowing for additional control over the latent variable $\psi_h$. For example, one can choose $\mathcal S(\psi_h^k) = \epsilon \psi_h^k$ for some $0 \leq \epsilon \ll 1$.
This choice will limit the magnitude of $\psi_h^k$, which can reduce the round-off error in the computed solution.
We discuss different choices of $\mathcal S$ in \Cref{sec:numerics}. In what follows, we analyze \eqref{eq:nonlinear_subproblem_mixed} with $ \mathcal S = 0$.  

Using \eqref{eq:nonlinear_subproblem_mixed_2} in \eqref{eq:nonlinear_subproblem_mixed_1} and the fact that $\nabla \cdot \bm v_h \in V_h^p$ for any $\bm v_h \in \bm \Sigma_{h,\mathrm{div}}^{p}$, we arrive at the two-variable formulation of the FOSPG method given in \Cref{alg:mixed_PG}. 
\begin{algorithm}[t] 
\caption{The First-Order System Proximal Galerkin Method} 
\label{alg:mixed_PG}
\begin{algorithmic}[1]\label{alg:main_alg_disc_mixed}
    \State \textbf{input:}  A discrete latent solution guess $\psi_h^0 \in V_h^p$ and a sequence of positive step sizes $\{\alpha^k\}$.
        \State \textbf{output:} A bound-preserving approximate solution $\mathcal{U}(\psi_h) \approx u$ and a locally-conservative flux approximation $\bm q_h\approx -A\nabla u$.
    \State Initialize \(k = 1\).
    \State \textbf{repeat}
    \State \quad Solve the following (nonlinear) discrete 
 saddle-point problem:
  Find $(\bm q_h, \psi_h) \in \bm \Sigma_{h,\mathrm{div}}^{p} \times V_h^p$ such that  
\begin{subequations} 
\label{eq:nonlinear_subproblem_2var}
 \begin{alignat}{2}
   (A^{-1} \bm q_h, \bm v_h) - (\nabla \cdot \bm v_h, \mathcal{U}(\psi_h) + \mathcal S(\psi_h) ) & = - \langle g , \bm v_h \cdot \bm n\rangle_{\partial \Omega} && \quad \forall \bm v_h \in\bm \Sigma_{h,\mathrm{div}}^{p}, 
    \\ 
  \alpha^k (\nabla \cdot \bm q_h, v_h) +  (\psi_h , v_h) & = \alpha^k (f,v_h) + (\psi_h^{k-1}, v_h)  && \quad \forall v_h \in V_h^p.
    \end{alignat}
   \end{subequations}
\State \quad Assign  \(\psi_h^{k} \gets \psi_h\) and \(k \gets k + 1\).
    \State \textbf{until} a convergence test is satisfied.
\end{algorithmic}
\end{algorithm}
\begin{remark}[The two-variable \eqref{eq:nonlinear_subproblem_2var} and three-variable formulations \eqref{eq:nonlinear_subproblem_mixed}]
Note that \eqref{eq:nonlinear_subproblem_2var} is equivalent to \eqref{eq:nonlinear_subproblem_mixed}. This can be immediately observed after defining $u_h^k = \Pi_h (\mathcal{U}(\psi_h) + \mathcal S(\psi_h))$.
We will use this equivalence throughout the paper and refer to the two formulations interchangeably when no confusion arises.  
Moreover, while the two-variable formulation \eqref{eq:nonlinear_subproblem_2var} is more memory efficient than \eqref{eq:nonlinear_subproblem_mixed}, one may compute with both; see, e.g., \Cref{alg:main_alg_disc} below. The three-variable formulation makes it easier to implement more exotic choices of $\mathcal S(\psi_h)$, such as  $(\mathcal S(\psi_h), q_h)  = \epsilon (\nabla_h \psi_h, \nabla_h q_h) $ for some $0 \leq \epsilon \ll 1$.
\end{remark}

 \begin{remark}[Bound preserving solution $\mathcal{U}(\psi_h)$ and an average bound preserving solution $u_h$] \label{remark:bd_preserv_avg}  In addition to the bound preserving discrete solution $\mathcal{U}(\psi_h)$, we obtain a solution $u_h$ that has bound preserving local averages whenever $(\mathcal S(\psi_h), 1)_T = 0$. Indeed, 
testing \eqref{eq:nonlinear_subproblem_mixed_2} with the indicator function of one element $T \in \mathcal{T}_h$, we obtain that 
\begin{align}
\frac{1}{|T|} \int_{T} u_h \dd x   = \frac{1}{|T|} \int_{T} \mathcal{U}(\psi_h) \dd x , \quad \forall T \in \mathcal{T}_h.
\end{align}
Therefore, if $\underline u , \overline u \in \R$, we obtain that 
\begin{align}
   \underline u <  \frac{1}{|T|} \int_{T} u_h \dd x  < \overline u, \quad \forall T \in \mathcal{T}_h.
\end{align}
Such a property is useful if one chooses to post-process $u_h$; see \Cref{remark:limiter} for further insight and a numerical example. 
\end{remark}
\begin{theorem}[Existence and uniqueness of solutions]\label{thm:ext_mixed}
For all $k \geq 1$, there exists a unique solution to \eqref{eq:nonlinear_subproblem_mixed} and to \eqref{eq:nonlinear_subproblem_2var}. 
\end{theorem}
\begin{proof}
 For simplicity, we denote  $\tilde f = f + (1/\alpha^k) \psi_h^{k-1}$ and drop the superscript $k$. Recall that  $\mathcal{U} = (\mathcal{R}^*)'$ where $\mathcal{R}^*$ is given in \eqref{eq:convex_conj_S_superposition}. We then consider the  following minimization problem: 
\begin{align} \label{eq:min_problem_L}
\inf_{\bm v_h \in \bm \Sigma_{h,\mathrm{div}}^{p} } L(\bm v_h); \quad 
L(\bm v_h) 
= 
\frac12 \|A^{-1/2} \bm v_h\|^2_{L^2(\Omega)} 
+ \int_{\Omega} \mathcal{R}^*(\Pi_h \tilde f - \nabla \cdot \bm v_h) \dd x + \langle g , \bm v_h \cdot \bm n \rangle_{\partial \Omega}, 
\end{align}
where we recall that $\Pi_h$ is the $L^2$-projection onto $V_h^p$, see \eqref{eq:local_L2_project}. 
As we show below,~\eqref{eq:min_problem_L} admits a solution $\bm q_h$ because $L$ is convex, continuous, and coercive  \cite[Theorem 9.3--1]{ciarlet2013linear}. 

Coercivity follows from the convexity of $\mathcal{R}^*$ and equivalence of norms in finite dimensions. Indeed,  by the subgradient inequality, 
\[ 
\int_{\Omega} \mathcal{R}^*(\Pi_h \tilde f - \nabla \cdot \bm v_h) \dd x \geq  \int_{\Omega} \mathcal{R}^*(\Pi_h \tilde f) \dd x   - \int_{\Omega} (\mathcal{R}^*)'(\Pi_h \tilde f) \nabla \cdot \bm v_h \dd x . 
\]
Hence, with Cauchy--Schwarz inequality and a local trace estimate, we bound $L(\bm v_h)$ as follows 
\begin{multline}
    L(\bm v_h) \geq \frac12 \|A^{-1/2} \bm v_h\|^2_{L^2(\Omega)}  + \int_{\Omega} \mathcal{R}^*(\Pi_h \tilde f)  \dd x \\ - \|(\mathcal{R}^*)'(\Pi_h \tilde f)\|_{L^2(\Omega) } \|\nabla \cdot \bm v_h\|_{L^2(\Omega)}  -  C_{\mathrm{tr}}\Big(\max_{E  \subset \partial \Omega}  h_E^{-1/2}\Big) \|g\|_{L^2(\partial \Omega)} \|\bm v_h \|_{L^2(\Omega)}. 
\end{multline}
Therefore, from the equivalence of norms in finite dimensions, we deduce that $L(\bm v_h) \rightarrow \infty$ as $\|\bm v_h\|_{L^2(\Omega)} \rightarrow \infty$.
This establishes the existence of a solution. The uniqueness of the solution follows the strict convexity of $L$.

The solution of \eqref{eq:min_problem_L}, $\bm q_h \in \bm \Sigma_{h,\mathrm{div}}^{p}$, satisfies the following optimality condition:  
\begin{align*}  
(L'(\bm q_h), \bm v_h) 
= 
(A^{-1} \bm q_h, \bm v_h)  -  \int_{\Omega}  (\mathcal{R}^*)' (\Pi_h \tilde f - \nabla \cdot \bm q_h) \nabla \cdot \bm v_h \dd x  + \langle g, \bm v_h \cdot \bm n \rangle_{\partial \Omega} = 0. 
\end{align*} 
Define $ \psi_h \in V_h^p$  and $u_h \in V_h^p$ as 
\begin{equation}
 \psi_h =   \alpha \Pi_h \tilde f -  \alpha \nabla \cdot \bm q_h , \quad  u_h = \Pi_h (\mathcal{U}(\psi_h)). 
\end{equation}
It then easily follows that $(\bm q_h, u_h , \psi_h) \in \bm \Sigma_{h,\mathrm{div}}^{p} \times V_h^p \times V_h^p$  is the unique solution to \eqref{eq:nonlinear_subproblem_mixed}.
\end{proof}
For each $k$, we show a stability estimate with respect to the mesh size $h$.  To this end, we define the following dual norm: 
\begin{align}
\|v\|_{H^1(\mathcal{T}_h)^*} = \sup_{w \in H^1(\mathcal{T}_h)} \frac{(v, w)}{\|w\|_{\mathrm{DG}}} \quad \forall v \in L^2(\Omega). \label{eq:negative_norm}
\end{align}
The stability result we provide below depends on the choice of $\mathcal{U}$. To this end, observe that \eqref{eq:Z_1} and \eqref{eq:Z_2} can be written in the following form: 
\begin{equation}\label{eq:general_form_U}
\mathcal{U}(\psi)(x) = \Upsilon(x , \psi(x)) = \phi_0(x) \Upsilon_0(\psi(x)) + \phi_1(x), 
\end{equation}
where $\phi_0, \phi_1 \in L^{\infty}(\Omega)$ with $\phi_0 (x) \geq \underline {\phi_0} > 0$ a.e.\ in $\Omega$ for some positive constant $\underline{\phi_0}$ and $\Upsilon_0 \colon \mathbb{R} \to (-1,1)$ is a monotonically increasing bijection with $\Upsilon_0(0) = 0 $ and $\lim_{z \rightarrow \pm \infty} \Upsilon_0(z) = \pm 1.$  In particular, for examples \eqref{eq:Z_1} and \eqref{eq:Z_2}, we have $\phi_0 = 1/2( \overline u  - \underline u ) $ and $\phi_1 = 1/2 (\overline u + \underline u)$. 

\begin{lemma}[Stability] \label{lemma:stability}
    \rami{ Assume that the user specified initial guess $\psi_h^0$ is chosen such that $\|\mathcal{U}(\psi_h^0)\|_{\mathrm{DG}} + \|\psi_h^0\|_{L^\infty(\Omega)} \lesssim 1$.  Then, there exists a constant $C_0$ independent of $h$, $k$ and $\{\alpha^\ell\}_{\ell = 1,\ldots, k}$ such that for all $k \geq 1$:
    \begin{equation} \label{eq:stability}
    \|\bm q_h^k\|_{L^2(\Omega)} + \|u_h^k\|_{\mathrm{DG}} + \frac{\|\psi_h^k\|_{H^1(\mathcal{T}_h)^*}}{\sum_{\ell=1}^k \alpha^\ell} \leq C_0. 
    \end{equation}
    In addition, if either $\underline u, \overline{u} \in \R$ with $\underline u < 0 < \overline{u}$  or   $\mathcal{U}$ takes the form \eqref{eq:general_form_U} with $\phi_0 \in L^\infty(\Omega)$, $\phi_0(x) \geq \underline{\phi_0} > 0$ a.e. in $\Omega$, and  $\phi_1 \in H^1_0(\Omega)$, then 
for any $k \geq 1$
\begin{equation} \label{eq:stability_2}
\underline{\phi_0}\| \psi_h^k \|_{L^1(\Omega)} \leq C_1 \sum_{\ell = 1}^k \alpha^\ell, 
  \end{equation}
  where $C_1$ is independent of $h$, $k$ and $\{\alpha^\ell\}_{\ell = 1,\ldots, k}$. 
    } 
\end{lemma} 
\begin{remark}\label{remark:boundary_data}
For non-homogenous boundary data $g \in H^{1/2}(\partial \Omega)$, one can consider a lift $w \in H^1_g(\Omega)$ and  set $u_g = u + w$ where $u$ solves the homogenous problem. The discrete solution $u_{h,g}$ is then given by $u_{h,g} = u_h + \mathcal{I}_h w_g$, where $\mathcal{I}_h$ is a suitable quasi-interpolation operator that preserves the boundary data, see for example \cite{ern2017finite}.  Stability estimates for the approximation $(u_{h,g}, \bm q_{h,g}) = (u_h + \mathcal{I}_h w_g, \bm q_h  - A \nabla  \mathcal{I}_h w_g)$ of $(u_g, -A \nabla u_g)$  will follow from the triangle inequality, the stability of $(u_h, \bm q_h)$, and the stability properties of the interpolant $\mathcal{I}_h$. 
Moreover, the assumption $\underline u < 0 < \overline{u}$ appears to be a limitation of the current theory \rami{to obtain the $L^1$ bound on $\psi_h^k$}. In practice, setting $\underline u = 0$ can be numerically stable, as evidenced by the numerical experiment in~\Cref{example:biactive}.
\end{remark}
\begin{proof}
\rami{
\textit{Step 1. Bounding $\| \bm q_h^k \|_{L^2(\Omega)}$ and $\|u_h^k\|_{\mathrm{DG}}$}. We adapt the ideas of \cite[Appendix A]{keith2025priori} to this setting.  
Using the definition of $\bm L$ \eqref{eq:first_lift} and \eqref{eq:nonlinear_subproblem_mixed_1}, we first observe that $ \bm q_h^k = \bm L (u_h^k) $ for all $ k \geq 1$ since by assumption $g = 0$. Define $u_h^0 = \Pi_h (\mathcal{U}(\psi_h^0))$ and $\bm q_h^0 = \bm L(u_h^0)$. With \eqref{eq:first_lift}, we also rewrite   \eqref{eq:nonlinear_subproblem_mixed_0} as 
\begin{equation}
\alpha^{k} (\bm q_h^{k}, A^{-1}\bm L (v_h))  + (\psi_h^{k} - \psi_h^{k-1}, v_h) = \alpha^{k}(f,v_h) \quad \forall v_h \in V_h^p.  \label{eq:rewrite_first_eq}
\end{equation} 
Testing the above with $v_h = u_h^{k} -u_h^{k-1}$ and using  \eqref{eq:nonlinear_subproblem_mixed_2}, we obtain that 
\begin{equation*}
(A^{-1}(\bm q_h^{k} - \bm q_h^{k-1}), \bm q_h^{k})-  (f, u_h^{k} - u_h^{k-1})  = - \frac{1}{\alpha^k} (\mathcal{U}(\psi_h^{k}) - \mathcal{U}(\psi_h^{k-1}), \psi_h^k - \psi_h^{k-1}) \leq 0,   
\end{equation*}
where we used the monotonicity of $\mathcal{U}$ to conclude the last inequality. Rewriting the above inequality yields the following energy dissipation property 
\begin{align}
\frac12 \|A^{-1/2} \bm q_h^{k}\|_{L^2(\Omega)}^2 - (f,u_h^{k}) \leq \frac12 \|A^{-1/2} \bm q_h^{k-1}\|_{L^2(\Omega)}^2 - (f,u_h^{k-1}) \quad \forall k \geq 1. 
\end{align}
Iterating the above and using Cauchy-Schwarz inequality followed by \Cref{lemma:dg_to_L2} and \eqref{eq:Poincare}, we obtain that 
\begin{align}
\frac12 \|A^{-1/2}\bm q_h^k\|_{L^2(\Omega)}^2 \leq \frac{1}{2}\|A^{-1/2}\bm q_h^0\|_{L^2(\Omega)}^2  + \|f\|_{L^2(\Omega)}\|u_h^0\|_{L^2(\Omega)} + C_P C_L \|f\|_{L^2(\Omega)} \|A^{-1/2} \bm q_h^k\|_{L^2(\Omega)}. \nonumber
\end{align}
From here, it is easy to see that 
\begin{align}
\frac{1}{4} \|A^{-1/2} \bm q_h^k\|_{L^2(\Omega)}^2 \leq \frac{1}{2}\|A^{-1/2}\bm q_h^0\|_{L^2(\Omega)}^2  + \|f\|_{L^2(\Omega)}\|u_h^0\|_{L^2(\Omega)} + C_P^2 C_L^2 \|f\|_{L^2(\Omega)}^2  \leq \tilde C_0, \end{align}
where $\tilde C_0$ is a constant depending on $\|\mathcal{U}(\psi_h^0)\|_{\mathrm{DG}}$ which can be seen from \Cref{lemma:dg_to_L2} and the stability of the $L^2$ projection. 
An application of \Cref{lemma:dg_to_L2} provides the bound on $\|u_h^k\|_{\mathrm{DG}}$. 
}

\rami{
\textit{Step 2. Bounding $ \|\psi_h^k\|_{H^1(\mathcal{T}_h)^*}$. 
} 
Summing \eqref{eq:rewrite_first_eq} and defining $\overline{\bm q}_h^k = \left( \sum_{\ell  = 1}^k \alpha^\ell \bm q_h^\ell  \right)/ \left( \sum_{\ell  = 1}^k \alpha^\ell \right)$, we obtain 
\begin{equation}
    (\overline{\bm q}_h^{k}, A^{-1}\bm L (v_h))  + \frac{1}{\sum_{\ell = 1}^k \alpha^\ell} (\psi_h^{k} - \psi_h^{0}, v_h) = (f,v_h) \quad \forall v_h \in V_h^p. 
\end{equation}
for any $k\geq 1$. Therefore,  Cauchy Schwarz inequality, \Cref{lemma:dg_to_L2}  and \eqref{eq:Poincare} yield for any $v_h \in  V_h^p$
\begin{equation}
(\psi_h^{k}, v_h)  \leq  \left( \sum_{\ell = 1}^k \alpha^\ell \right) \left(c_L^{-1} \|A^{-1/2} \overline{\bm q}_h^{k}\|_{L^2(\Omega)} + C_P \|f\|_{L^2(\Omega)} \right) \|v_h\|_{\mathrm{DG}} + \|\psi_h^0\|_{H^1(\mathcal{T}_h)^*}\|v_h\|_{\mathrm{DG}}. 
\end{equation}
With the help of the $L^2$-projection, we estimate 
\begin{align}
\|\psi_h^k\|_{H^1(\mathcal{T}_h)^*} =  \sup_{w \in H^1(\mathcal{T}_h)} \frac{(\psi_h^k, w)}{\|w\|_{\mathrm{DG}}} =\sup_{w \in H^1(\mathcal{T}_h)} \frac{(\psi_h^k, \Pi_h w)}{\|w\|_{\mathrm{DG}}}  \lesssim \sup_{w \in H^1(\mathcal{T}_h)} \frac{|(\psi_h^k, \Pi_h w)|}{\|\Pi_h w\|_{\mathrm{DG}}} . \label{eq:negative_norm_estimate_psihk}
\end{align}
Combing the above two bounds provides the estimate 
\begin{align}
\|\psi_h^k\|_{H^1(\mathcal{T}_h)^*} \lesssim  \left( \sum_{\ell = 1}^k \alpha^\ell \right) \left(c_L^{-1} \|A^{-1/2}\overline{\bm q}_h^{k}\|_{L^2(\Omega)} + C_P \|f\|_{L^2(\Omega)} \right) + \|\psi_h^0\|_{H^1(\mathcal{T}_h)^*}. 
\end{align}
With noting that $\|A^{-1/2}\overline{\bm q}_h^{k}\|_{L^2(\Omega)} \leq  \tilde C_0$ and $\|\psi_h^0\|_{H^1(\mathcal{T}_h)^*} \leq C_P\|\psi_h^0\|_{L^2(\Omega)}$, we conclude the bound on $\|\psi_h^k\|_{H^1(\mathcal{T}_h)^*}$. 
}

\rami{
\textit{Step 3. Bounding $ \|\psi_h^k\|_{L^1(\Omega)}$. 
}We provide a detailed proof for the double obstacle problem case, and we discuss the modifications to the proof for the anisotropic diffusion case at the end. 
We test \eqref{eq:nonlinear_subproblem_mixed_0} with $v_h = u_h^k$, \eqref{eq:nonlinear_subproblem_mixed_1} with $\bm v_h = \alpha^k  \bm q_h^k$, and \eqref{eq:nonlinear_subproblem_mixed_2} with $q_h = -\psi_h^k$, then we add the resulting equalities. \rami{Using \eqref{eq:general_form_U}}, this delivers 
\begin{equation}\label{eq:stab_0}
\alpha^k \|A^{-1/2} \bm q_h^k\|^2_{L^2(\Omega)}  + (\phi_0 \Upsilon_0(\psi^k_h), \psi^k_h) = \alpha (\tilde f , u^k_h) - (\phi_1, \psi^k_h) . 
\end{equation}
Observe that since $\Upsilon_0$ is monotone,  $\Upsilon_0(0) = 0$, and $\phi_0 \geq \underline{\phi_0} > 0$, the second term in \eqref{eq:stab_0} is positive. In addition, we have that 
\[
\Upsilon_0( \psi^k_h)\psi^k_h  \geq  |\psi^k_h| - b \quad \text{a.e. in } \Omega,  
\] 
where $b = \min_{z \in \R} (z  \Upsilon_0(z) - |z| ) < 0$. Note that $b$ is well defined since $\lim_{z \rightarrow \pm \infty } (z\Upsilon_0(z) -  |z| ) = 0 $. Therefore, $b$ is the minimum value of a continuous function over some bounded interval. We arrive at
\begin{equation}\label{eq:stab_1}
\alpha^k \|A^{-1/2} \bm q^k_h\|^2_{L^2(\Omega)}  + \underline{\phi_0} \|\psi^k_h\|_{L^1(\Omega)} \leq (\alpha^k f + \psi_h^{k-1} , u^k_h) - (\phi_1, \psi^k_h)  + (b, 1). 
\end{equation}
Proceeding, we bound the terms on the right-hand side of \eqref{eq:stab_1}. With the Cauchy--Schwarz inequality and \eqref{eq:negative_norm}, we deduce that
\begin{align} \label{eq:stab_2}
| (\alpha^k f + \psi_h^{k-1} , u^k_h) |  & \leq \alpha^k \|f\|_{L^2(\Omega)}\| u^k_h\|_{L^2(\Omega)} + \|\psi_h^{k-1}\|_{H^1(\mathcal{T}_h)^*} \| u^k_h \|_{\mathrm{DG}} 
\end{align}
The second term in \eqref{eq:stab_1} is bounded as follows: 
\begin{align} \label{eq:stab_3}
(\phi_1, \psi^k_h)  \leq \|\psi^k_h\|_{H^1(\mathcal{T}_h)^*} \|\phi_1\|_{\mathrm{DG}} =  \|\psi_h^k\|_{H^1(\mathcal{T}_h)^*} \|A^{1/2} \nabla \phi_1\|_{L^2(\Omega)}, 
\end{align}
where the last equality follows since $\phi_1 \in H^1_0(\Omega)$. Combining the above estimate and using the proved stability bounds show the required bound. }

\rami{
If $\underline u, \overline u \in \R$, then $\phi_1 \in \R$. Thus, the equality in \eqref{eq:stab_3} no longer holds, and the DG norm of a constant depends inversely on the mesh size. In this case, we remove the dependence on $\phi_1$ by shifting the problem as follows. Since $\underline u < 0 < \overline u$ and since $\mathcal{U}$ is smooth and monotone, there exists a unique $c \in \R$ such that $\mathcal{U}(c) = 0$ by the intermediate value theorem. Set $\tilde \psi^k_h = \psi^k_h  - c$ and $\tilde{\mathcal U} (z) = \mathcal U (z + c), \; z \in \R$. Then, the solution $(\bm q^k_h, u^k_h, \tilde \psi^k_h) \in  \bm \Sigma_{h,\mathrm{div}}^{p} \times V_h^p \times V_h^p$ satisfies   
\begin{subequations} 
\label{eq:nonlinear_subproblem_mixed_shifted}
\begin{alignat}{2}
  \alpha (\nabla \cdot \bm q^k_h, v_h) +  (\tilde \psi^k_h , v_h) & = \alpha (\tilde f - c ,v_h)   && \quad \forall v_h \in V_h^p,\label{eq:nonlinear_subproblem_mixed_shifted_0}
 \\ 
   (A^{-1} \bm q^k_h , \bm v_h) - (\nabla \cdot \bm v_h, u^k_h ) & = 0 && \quad \forall \bm v_h \in\bm \Sigma_{h,\mathrm{div}}^{p},   \label{eq:nonlinear_subproblem_mixed_shifted_1}
\\
   (u^k_h, q_h)  - (\tilde{\mathcal{U}}(\tilde \psi^k_h),q_h) &  = 0  && \quad \forall q_h \in V_h^p.\label{eq:nonlinear_subproblem_mixed_shifted_2}
 \end{alignat}
 \end{subequations}
From here, the arguments follow in the same manner as before, where we use that 
\[
 \tilde{\mathcal{U}}(\tilde{\psi}^k_h ) \tilde{\psi}^k_h  \geq  a |\tilde \psi^k_h| - b, 
 \]
 with $a = \min(- \underline u , \overline u) > 0$ and $b = \min_{z \in \R} (z \tilde{\mathcal U}(z) - a |z| )$.  The final bounds on the norms of $\psi^k_h$ then follow from the triangle inequality.
}
\end{proof}

\subsection{The hybridized first-order system proximal Galerkin method} \label{sec:hybridization}
An appealing way to reduce the computational cost of FOSPG is through hybridization \cite{BBF13, AB85}.
After static condensation, this leads to symmetric positive definite linear systems for each linearized subproblem. 

On the reference facet $\hat E$, we denote by
$\mathcal{P}_p(\hat{E})$ the set of all polynomials of degree at most $k$ on $\hat{E}$ for $d=2$. For $d=3$ and for quadrilateral elements, we set $\mathcal{P}_p(\hat{E}) = \mathbb{Q}_p(\hat E)$. Consider the following space of polynomials defined locally on each facet $E \in \mathcal{E}_h$ 
\begin{align}    \label{fes_facet}
   M_{h,g}^p = &\;\bigg\{
        {\mu}_h\in L_2(\mathcal{E}_h): \quad {\mu}_h{}_{|_{E}}
        =\hat{{\mu}} \circ\Phi_E^{-1},  \; \hat{{\mu}} \in \mathcal{P}_p(\hat E) , \; \forall E\in\mathcal{E}_h\\
       & \nonumber \quad \quad \quad \quad \quad \quad \quad \quad  
       \quad \quad \quad \quad \quad
       \;\mu_h{}_{|_{E'}} = P_p(g){}_{|_{E'}},\; \forall E'\subset \partial\Omega 
        \bigg\},
    \end{align}
where $P_p(g)|_{E'}$ is the $L^2$-projection of the Dirichlet boundary data $g$  onto the polynomial space $\mathcal{P}_p(E')$. 

Following \cite{egger2010hybrid}, we consider the following energy norm. For $\bm q \in L^2(\Omega)^d, v \in H^1(\mathcal{T}_h)$ and $\hat v \in L^2(\partial \mathcal{T}_h)$, define 
\begin{equation} \label{eq:def_triple_norm}
\vvvert(\bm q, v, \hat v)\vvvert^2 = \|A^{-1/2} \bm q\|^2_{\mathcal{T}_h} + \|A^{1/2} \nabla_h v\|^2_{\mathcal{T}_h} + \sum_{T \in \mathcal{T}_h} h_T^{-1} \| v - \hat v\|^2_{L^2(\partial T)} . 
\end{equation}
From the  triangle inequality and mesh regularity, see the derivation of \cite[Bound 2.20]{10.1093/imanum/drad075} for details, we remark that for any $v \in H^1(\mathcal{T}_h)$, $\bm q \in L^2(\mathcal{T}_h)^d$ and $q \in L^2(\partial \mathcal{T}_h)$ with $ q= 0 $ on $\partial \Omega$, 
\begin{equation}
\|v\|_{\mathrm{DG}} \lesssim \vvvert(\bm q , v , q) \vvvert. \label{eq:dg_to_triple} 
\end{equation}

The hybrid mixed discretization relies on a bilinear form $\mathcal{B}_h: \bm{\Sigma}_h^p \times H^1(\mathcal{T}_h) \times L^2(\partial \mathcal{T}_h) \rightarrow \R $ to reimpose  the continuity of the normal fluxes over the facets, see \cite{egger2010hybrid} for more details, 
\begin{equation} \label{eq:def_Bh}
\mathcal{B}_h(\bm q_h, (v, \hat v)) = ( \bm q_h, \nabla_h v_h)_{\mathcal{T}_h} - (v - \hat v, \bm q_h \cdot \bm n)_{\partial \mathcal T_h}.
\end{equation} 
We also define the following form $\mathcal{A}_h: (\bm{\Sigma}_h^p \times V_h^p \times M_{h}^p)^2 \rightarrow \R$:  
\begin{equation} \label{eq:def_form_A}
\mathcal{A}_h((\bm q_h, u_h, \hat u_h), (\bm r_h, v_h, \hat v_h)) =  (A^{-1} \bm q_h, \bm r_h)_{\mathcal{T}_h} + \mathcal{B}_h (\bm r_h, (u_h, \hat u_h)) -   \mathcal{B}_h (\bm q_h, (v_h, \hat v_h)). 
\end{equation} 
\begin{lemma}[Properties of $\mathcal{A}_h$]\label{lemma:inf_sup}
For all $(\bm q_h, u_h, \hat u_h) \in \bm \Sigma_h^p \times V_h^p \times M_{h,0}^p$, there exists $\bm r_h \in \bm \Sigma_h^p$ such that 
\begin{equation} \label{eq:inf_sup}
\mathcal{A}_h((\bm q_h, u_h, \hat u_h), (\bm r_h, u_h, \hat u_h)) \gtrsim \vvvert (\bm q_h, u_h, \hat u_h) \vvvert^2. 
\end{equation}
Further, the form $\mathcal{A}_h$ is continuous over $\bm \Sigma_h^p \times V_h^p \times M_{h,0}^p$ 
\begin{equation} \label{eq:continuity_A}
  \mathcal{A}_h((\bm q_h, u_h, \hat u_h), (\bm r_h, v_h, \hat v_h)) \lesssim \vvvert (\bm q_h, u_h, \hat u_h) \vvvert \, \vvvert (\bm r_h, v_h, \hat v_h)  \vvvert
\end{equation}
\end{lemma}
\begin{proof} 
The proof closely follows \cite{egger2010hybrid}; we provide details in  Appendix~\ref{sec:proof_Ah} for completeness. 
\end{proof}
We are now ready to present the hybridized FOSPG method, given in~\Cref{alg:main_alg_disc}. 
\begin{algorithm}[t]
\caption{The Hybridized First Order System Proximal Galerkin Method} 
\begin{algorithmic}[1]\label{alg:main_alg_disc}
    \State \textbf{input:}  A discrete latent solution guess $\psi_h^0 \in V_h^p$ and a sequence of positive step sizes $\{\alpha^k\}$.
    \State \textbf{output:} A bound-preserving approximate solution $\mathcal{U}(\psi_h) \approx u$, a piecewise-polynomial approximate solution $u_h \approx u$, and a locally-conservative flux approximation $\bm q_h\approx -A\nabla u$. 
    \State Initialize \(k = 1\). 
    \State \textbf{repeat}
    \State \quad Solve the following (nonlinear) discrete 
 saddle-point problem:
  Find $(\bm q^{k}_h, u_h^{k} , \hat u^{k}_h) \in \bm \Sigma_h^p  \times V_h^p \times  M_{h,g}^p$ and $\psi_h^{k} \in V_h^p$ such that  
\begin{subequations} 
\label{eq:nonlinear_subproblem}
\begin{alignat}{2}
  - \alpha^k \mathcal{B}_h (\bm q_h^{k}, (v_h, \hat v_h)) +  (\psi_h^{k} , v_h) & = \alpha^k (f,v_h) + (\psi_h^{k-1}, v_h) && \quad \forall (v_h, \hat v_h) \in V_h^p \times M_{h,0}^p, \label{eq:nonlinear_subproblem_0}\\ 
   (A^{-1} \bm q_h^{k}, \bm v_h) + \mathcal{B}_h (\bm v_h, (u^{k}_h, \hat u^{k}_h)) & = 0  && \quad \forall \bm v_h \in \bm \Sigma_h^p, \label{eq:nonlinear_subproblem_1}\\
   (u_h^{k}, q_h)  - (\mathcal{U}(\psi_h^{k}) + \mathcal S(\psi_h^k),q_h) &  = 0  && \quad \forall q_h \in V_h^p.\label{eq:nonlinear_subproblem_2}
 \end{alignat}
 \end{subequations}
 \State \quad Assign \(k \gets k + 1\).
    \State \textbf{until} a convergence test is satisfied.
\end{algorithmic}
\end{algorithm}
\begin{theorem} \label{thm:exist_hybrid_FOSPG}
[Existence and uniqueness of solutions]
For all $k \geq 1$, there exists a unique solution to \eqref{eq:nonlinear_subproblem}. 

In addition, \rami{under the conditions of \Cref{lemma:stability}}, the solution $ (\bm q_h^k, u_h^k, \hat{u}_h^k, \psi_h^k) \in  \bm \Sigma_h^p  \times V_h^p \times  M_{h,g}^p \times V_h^p$ to \eqref{eq:nonlinear_subproblem} satsifies the stability bounds \eqref{eq:stability}--\eqref{eq:stability_2} with $\vvvert (\bm q_h^k, u_h^k, \hat{u}_h^k)\vvvert $ replacing ($\|A^{-1} \bm q^k_h\|_{L^2(\Omega)} + \|u_h^k\|_{\mathrm{DG}}$). 
\end{theorem} 
\begin{proof} 
The proof relies on showing the uniqueness of solutions to \eqref{eq:nonlinear_subproblem} and that the solution to \eqref{eq:nonlinear_subproblem_mixed}, guaranteed by \Cref{thm:ext_mixed}, defines the solution to \eqref{eq:nonlinear_subproblem}. 
As in the proof of \Cref{thm:ext_mixed}, we denote $\tilde f = f + (1/\alpha^k) \psi_h^{k-1}$ and drop the superscript $k$.  Suppose that we are given two solutions $(\bm q^1_h, u_h^1, \hat u_h^1, \psi_h^1)$ and $(\bm q^2_h, u_h^2, \hat u_h^2, \psi_h^2)$ to \eqref{eq:nonlinear_subproblem}. Denote by $(\bm \xi_h, \xi_h ,\hat \xi_h) = (\bm q^1_h, u_h^1, \hat u_h^1) - (\bm q^2_h, u_h^2, \hat u_h^2)$. Then, utilizing the definition of $\mathcal{A}_h(\cdot, \cdot)$, we obtain that 
\begin{align}
\alpha \mathcal{A}_h((\bm \xi_h, \xi_h, \hat \xi_h), (\bm v_h, v_h,\hat v_h)) + (\psi_h^1 -\psi_h^2, v_h) & = 0, \label{eq:diff_0} \\ 
(\xi_h , q_h) - (\mathcal{U}(\psi_h^1)-\mathcal{U}(\psi_h^2), q_h) & = 0, \label{eq:diff_1}
\end{align}
for all $(\bm v_h, v_h, \hat v_h , q_h) \in  \bm \Sigma_h^p  \times V_h^p \times M_{h,0}^p \times V_h^p$. Testing \eqref{eq:diff_0} by $(\bm v_h , \xi_h, \hat \xi_h)$ and \eqref{eq:diff_1} by $\psi_h^1-\psi_h^2$ and subtracting yields 
 \[ 
 \alpha \mathcal{A}_h((\bm \xi_h, \xi_h, \hat \xi_h), (\bm v_h, \xi_h ,\hat \xi_h)) + (\mathcal{U}(\psi_h^1) -\mathcal{U}(\psi_h^2), \psi_h^1 - \psi_h^2) = 0, 
 \]
 for all $\bm v_h \in \bm \Sigma_h^p$. Using the strict monotonicity of $\mathcal{U}$  and \Cref{lemma:inf_sup}, we conclude that $(\bm \xi_h, \xi_h, \hat \xi_h) = (\bm 0, 0, 0)$. From \eqref{eq:diff_0}, we further obtain that $\psi_h^1 = \psi_h^2$. Hence, \eqref{eq:nonlinear_subproblem} must have unique solutions. 
 
  It is standard to show that the solution to \eqref{eq:nonlinear_subproblem_mixed} defines the solution to \eqref{eq:nonlinear_subproblem}. Indeed, since
 $$\mathcal{B}_h(\bm q_h, (v_h, \hat v_h)) = -(\nabla \cdot \bm q_h,  v_h )_{\mathcal{T}_h} + (\hat v_h, \bm q_h \cdot \bm n)_{\partial \mathcal{T}_h},$$ 
 and $\bm q_h \in\bm \Sigma_{h,\mathrm{div}}^{p}$, \eqref{eq:nonlinear_subproblem_0} is satisfied. Further, the facet multipliers can be uniquely determined as $\hat u_h = \lambda_h $ on $E \in \mathcal{E}_h \backslash \partial \Omega$ and $\hat u_h = P_p(g) $ on $\partial \Omega$ where $\lambda_h \in M_{h,0}^p$, see \cite[Chapter 5, Lemma 1.1]{fortin1991mixed}, uniquely solves 
 \begin{equation}\label{eq:constructing_multiplier}
     (\lambda_h, \bm v_h \cdot \bm n)_{\partial \mathcal{T}_h} = - (A^{-1} \bm{q}_h, \bm v_h) + (\nabla \cdot \bm v_h, u_h)   - \langle g , \bm v_h \cdot \bm n \rangle_{\partial \Omega}, \quad \forall \bm v_h \in \bm \Sigma_h^p, 
 \end{equation}
 which is a rewrite of \rami{\eqref{eq:nonlinear_subproblem_1}}.

 \rami{The proof of the stability estimates on $\|A^{-1/2}\bm q_h^k\|_{L^2(\Omega)}$,  $\|u_h^k\|_{\mathrm{DG}}$, $\|\psi_h^k\|_{H^1(\mathcal{T}_h)^*}$, and $\|\psi_h^k\|_{L^1(\Omega)}$ follow from the equivalence that we have shown.  To show the bound on $\|h_T^{-1/2} (\hat u^k_h -u^k_h)\|^2_{\partial \mathcal{T}_h}$, we note 
that there exists $\bm{\tau}^k_h \in \bm \Sigma_h^p$ such that 
\begin{align*}
\mathcal{B}_h( \bm \tau^k_h, (u^k_h , \hat u^k_h)) =  \|A^{-1/2} \nabla_h u^k_h\|^2_{\mathcal{T}_h} +  \|h_T^{-1/2} ( u^k_h - \hat u^k_h ) \|^2_{\partial \mathcal{T}_h}, \\ 
\|\bm \tau^k_h \|_{\mathcal{T}_h}^2 \lesssim \|A^{-1/2} \nabla_h u^k_h\|^2_{\mathcal{T}_h} +  \|h_T^{-1/2} ( u^k_h - \hat u^k_h ) \|^2_{\partial \mathcal{T}_h}. 
\end{align*}
We refer to \cite[Lemma 3.1]{egger2010hybrid} and \Cref{sec:proof_Ah} for more details. 
Proceeding, we test \eqref{eq:nonlinear_subproblem_1} with $\bm \tau_h^k$, use the above properties, Cauchy--Schwarz inequality,  and the bound on $\|A^{-1/2}\bm q_h^k\|_{L^2(\Omega)}$. This concludes the result.  
}
\end{proof}

\section{A quadrature rule space discretization and its properties} 
\label{sec:conv_mixed_vi}
In this section, we show that for a quadrature rule space discretization (in the sense of \Cref{cond:reason_for_quads} and \Cref{remark:reason_for_quads} below), the solution $u_h^k$ is bound preserving on quadrature points. This property is demonstrated in \Cref{subsec:bound_preserving_uh} and plays a key role in showing that  \Cref{alg:main_alg_disc_mixed} converges to the solution of a discrete mixed VI given in \eqref{eq:discrete_mixed_VI}, see \Cref{subsec:convergence_mixed_VI}.   This convergence result allows us to show a local mass conservation property in  \Cref{subsec:local_mass_conserv} and to derive a priori error estimates for $p=0$ in \Cref{subsec:error_estimates_p0}

\subsection{Bound preservation of $u_h^k$ on quadrature points} \label{subsec:bound_preserving_uh} 
A key feature of the proximal Galerkin method is equality \eqref{eq:nonlinear_subproblem_mixed_2}, which allows us to extract distinctive properties for the solution $u_h^k$. For example, we recover that $u_h^k$ has local averages satisfying the bound constraints, see \Cref{remark:bd_preserv_avg}.

One can exploit \eqref{eq:nonlinear_subproblem_mixed_2} further.  In particular, under certain conditions, bound preservation of  $u_h^k$ on quadrature points can be established.
Indeed, consider if one implements the inner products in \eqref{eq:nonlinear_subproblem_mixed}  by a quadrature rule  $\hat{\mathcal{R}}_p:=\{(\hat x_i, \hat w_i):\; i=1,\cdots, n_p
\}$ with positive weights $\hat w_i>0$  on $\hat T$ that is exact for $\mathcal{P}_{2p}(\hat T)$. In this case, equality~\eqref{eq:nonlinear_subproblem_mixed_2} is approximated by 
\begin{equation} \label{eq:nonlinear_quad}
(u_h^k, q_h)  - (\mathcal{U}(\psi_h^k) ,q_h)_h  = 0   \quad \forall q_h \in V_h^p, 
\end{equation}
where for all piecewise continuous functions $f$ and $g$, 
\begin{align}
\label{int}
(f, g )_h:=\sum_{T\in\mathcal{T}_h}\sum_{i=1}^{n_p}
    f( x_i^T) g( x_i^T)\omega_i^T . 
\end{align}
In the above, $w_i^T = |T|/|\hat T| \hat w_i$ and $x_i^T = \Phi_T(\hat x_i)$ for each $T \in \mathcal{T}_h$. Note that $(f,g)_h = (f,g)$ for all $f,g \in V_h^p$. For a given element $T \in \mathcal{T}_h$,  define 
the polynomial  functions $\psi_i^T$ satisfying
\begin{equation}\label{eq:def_psi_i}
\psi_i^T(x_j^{T}) = \left\{
\begin{tabular}{ll}
     $1$& if $i=j$, \\[1ex]
    $0$ & otherwise, 
\end{tabular}\right.\quad \forall 1\le j\le n_p. 
\end{equation}
For example, one can define $\psi_i^T$ as the Lagrange polynomial with quadrature points as the Lagrange nodes; i.e. (with $x[m]$ denoting the $m$-th Cartesian coordinate of a point $x \in \R^d$), 
\[ \psi_i^T(x)  = \prod_{m=1}^{d} \prod_{\substack{\ell = 1
\\  x_\ell^T[m] \neq x_i^T[m] } }^p 
\frac{x[m] - x_{\ell}^T[m]}{x_i^T[m] - x_\ell^T[m]}.   \] 
\begin{condition}\label{cond:reason_for_quads}
We say that the discretization of \Cref{alg:main_alg_disc_mixed} or \Cref{alg:main_alg_disc} is a quadrature rule space discretization 
if the extension by zero of  the polynomial $\psi_i^T$ that satisfies \eqref{eq:def_psi_i} \rami{with minimal polynomial degree} belongs to the broken polynomial space $V_h^p$, i.e., 
\begin{equation}
    \psi_i^T \in \mathcal{P}_p(T) \quad \forall  1 \leq i \leq n_p , \quad \forall T \in \mathcal{T}_h. \label{eq:reason_for_quads}
\end{equation}  
\end{condition}
In this section, we will assume that \Cref{cond:reason_for_quads} holds. We comment on its validity in the following remark. 
\begin{remark}[Validity of \Cref{cond:reason_for_quads}] \label{remark:reason_for_quads}
 Condition \eqref{eq:reason_for_quads}  can be easily verified for quadrilateral or hexahedral elements. For example, one can choose $n_p = (p+1)^d$ with $(p+1)$ Gauss-Legendre quadrature points in each Cartesian direction of $\hat T$. Then, the corresponding quadrature rule is exact for $\mathbb{Q}_{2p+1}(\hat T)$ and \eqref{eq:reason_for_quads} holds. In addition, \eqref{eq:reason_for_quads} holds for $p=0$ for simplicial elements also. For condition \eqref{eq:reason_for_quads} to hold for simplicial elements with $p \ge  1$, one needs to use an inexact quadrature to evaluate the inner product $(u_h^k,q_h)$ in~\eqref{eq:nonlinear_quad}. We choose not to explore the consequences of that setting here. 
\end{remark} 
\begin{lemma}[Bound preservation of $u_h^k$ on quadrature points] If \Cref{cond:reason_for_quads} holds, then for any $k \geq 1$ 
\begin{align}
  \underline{u}(x_i^T) < 
  u^k_h(x_i^T) <  
  \overline{u}(x_i^T)  
  \quad \forall 1 \le i\le n_p, \; \forall T \in \mathcal{T}_h.
\end{align}
\begin{proof} 
For any $T \in \mathcal{T}_h$, one can choose $q_h{}_{\vert T} = \psi_i^T $ and $q_h = 0$ elsewhere  in \eqref{eq:nonlinear_quad} thanks to \eqref{eq:reason_for_quads}. Using that $(u^k_h,q_h) = (u^k_h, q_h)_{h}$, we derive that
\begin{equation}
 u^k_h(x_i^T) = \mathcal{U}(\psi^k_h)(x_i^T) \quad \forall  1 \leq i \leq n_p, \; \forall T \in \mathcal{T}_h.  \label{eq:quad_pt_bound}
\end{equation}
The result follows from the choice of $\mathcal{U}$. 
\end{proof}

\end{lemma}
\subsection{Convergence to a discrete mixed VI}\label{subsec:convergence_mixed_VI}
The main goal of this section, realized in \Cref{thm:conv_mixed_vi}, is to show that the proximal iterates generated by \Cref{alg:mixed_PG} converge to the solution of the following discrete mixed variational inequality. 

Find $(\bm q_h^*, u_h^*) \in \bm \Sigma_{h,\mathrm{div}}^{p} \times \Lambda_h$ such that 
\begin{subequations}
\label{eq:discrete_mixed_VI}
\begin{alignat}{2}
  (A^{-1} \bm q_h^*, \bm r_h) - (u_h^* , \nabla \cdot \bm r_h) &  = -\langle g, \bm r_h \cdot \bm n\rangle  && \quad \forall \bm r_h \in \bm \Sigma_{h,\mathrm{div}}^{p},    \label{eq:discrete_mixed_VI_0}
 \\ 
  (\nabla \cdot \bm q_h^* - f , v_h - u_h^* ) & \geq 0   && \quad \forall v_h \in \Lambda_h, \label{eq:discrete_mixed_VI_1}
\end{alignat}
\end{subequations}
where we define the closed and convex set  
\begin{equation} \label{eq:def_discrete_Lambda}
\Lambda_h = \{
v_h \in V_h^p,
\quad 
 \underline u( x_i^T ) \leq v_h (x_i^T)
 \leq \overline{u}(x_i^T) , \;  \; 
\forall 1 \leq i \leq n_p, \;\; \forall 
T \in \mathcal{T}_h 
\}. 
\end{equation} 
\begin{lemma}
There exists a unique solution $(\bm q_h^*, u_h^*) \in \bm \Sigma_{h,\mathrm{div}}^{p} \times \Lambda_h$  to \eqref{eq:discrete_mixed_VI}.
\end{lemma}
\begin{proof}
Recall the lifting operators \eqref{eq:first_lift}-\eqref{eq:second_lift}, and 
 consider the minimization problem:
\begin{align}
\min_{u_h \in \Lambda_h} \frac12 \|A^{-1/2} \bm L(u_h)\|^2 - (f, u_h) \rami{-} \langle g, \bm L (u_h) \cdot \bm n \rangle. 
\end{align}
 Note that $\Lambda_h$ is nonempty since by \eqref{eq:reason_for_quads} $\sum_{T \in \mathcal{T}_h} \sum_{i=1}^{n_p}  \overline{u}(x_i^T) \psi_i^T \in \Lambda_h$ (here, $\psi_i^T$ are extended by zero). 
  The above problem is coercive since $\|\bm L (\cdot) \|$ defines a norm on $V_h^p$  \footnote{$\|\bm L (\cdot) \|$ defines a norm on $V_h^p$ since if $(u_h, \nabla \cdot \bm r_h) = 0 \; \forall \bm r_h    \in \bm{\Sigma}_{h,\mathrm{div}}^{p}$ then $u_h = 0. $  This follows from the discrete inf-sup condition, see for e.g. \cite[Lemma 51.10]{ern2021finite}.
   } 
and since norms are equivalent in finite dimensions. Therefore, there exists a solution $u_h^* \in \Lambda_h$ satisfying the variational inequality  \begin{align}\label{eq:euler_ineq}
     \mathcal W:= (A^{-1} \bm L(u_h^*), \bm L(v_h - u_h^*)) - (f,v_h - u_h^*) \rami{-} \langle g, \bm L(v_h - u_h^* ) \cdot \bm n \rangle \geq 0 , \quad \forall v_h \in \Lambda_h. 
 \end{align}  
 The uniqueness of $u_h^*$ follows from strict convexity (a composition of a strictly convex function $\|\cdot \|^2$ with a linear function $A^{-1/2}\bm L (\cdot) $ is strictly convex).  
 Defining
 \begin{equation}
     \bm q_h^* = \bm L(u_h^*) + \bm L_\Gamma (g), 
 \end{equation} 
 we see that \eqref{eq:discrete_mixed_VI_0} holds and from \eqref{eq:euler_ineq}, we derive that 
 \begin{align*}
    \mathcal{W} &=  (A^{-1} \bm L(u_h^*), \bm L(v_h - u_h^*)) - (f,v_h - u_h^*) + (A^{-1} \bm L_{\Gamma}(g) ,  \bm L(v_h - u_h^*))  \\ 
    & = (A^{-1} \bm q_h^*, \bm L(v_h - u_h^*) ) -  (f,v_h - u_h^*) \\
    &= (\nabla \cdot \bm q_h^* - f, v_h - u_h^*) \geq 0 .
 \end{align*}
\end{proof}

\begin{remark}
The discrete variational inequality \eqref{eq:discrete_mixed_VI}  is an approximation of the following mixed variational inequality. Find $(\bm q^*,  u^*) \in H(\mathrm{div}; \Omega) \times L^2(\Omega)$ such that 
\begin{subequations}
\label{eq:mixed_VI}
\begin{alignat}{2}
\label{eq:mixed_VI_0}
  (A^{-1} \bm q^*, \bm r) - (u^* , \nabla \cdot \bm r) &  = -\langle g, \bm r \cdot \bm n\rangle  && \quad \forall \bm r \in  H(\mathrm{div};\Omega), \\ 
\label{eq:mixed_VI_1}
  (\nabla \cdot \bm q ^* - f , v - u^* ) & \geq 0   && \quad \forall v \in \Lambda , 
\end{alignat}
\end{subequations}
where $\Lambda \subset L^2(\Omega)$ is given by 
\begin{equation}
\Lambda = \{v \in L^2(\Omega), \; \underline u \leq v  \leq \overline{u}   ~\text{ a.e.\ in } \Omega\}. 
\end{equation}
The well-posedness of the above problem can be established by showing that under sufficient regularity assumptions,  the solution to \eqref{eq:VI} defines the solution to \eqref{eq:mixed_VI}. We refer to \cite{brezzi1978error} for additional details. 
\end{remark}

Define the discrete Bregman distance that discretizes the integral of \eqref{eq:D_formal} 
\[ 
D_h(u_h, v_h) = (\mathcal{R}(u_h) -\mathcal{R}(v_h), 1)_h -  (\mathcal{R}'(v_h), u_h - v_h)_h  
\quad 
\forall u_h 
\in \Lambda_h,  v_h \in \tilde \Lambda_h, 
\]
where $\tilde \Lambda_h$ is given by 
\[ \tilde \Lambda_h =  \{
v_h \in V_h^p,
\quad 
 \underline u(x_i^T) <  v_h (x_i^T)
 <  \overline{u}(x_i^T) \;  \; 
\forall 1 \leq i \leq n_p, \;\; \forall 
T \in \mathcal{T}_h 
\}.  \]
\begin{theorem}[Convergence to the discrete mixed VI] \label{thm:conv_mixed_vi}
Assume that condition \eqref{eq:reason_for_quads} holds. 
Then, the proximal iterates  $(\bm q_h^{k+1}, u_h^{k+1})$ generated by \eqref{eq:nonlinear_subproblem_mixed} converge to the solution $(\bm q_h^*, u_h^*) \in \bm \Sigma_{h,\mathrm{div}}^{p} \times \Lambda_h$ of the discrete mixed variational inequality \eqref{eq:discrete_mixed_VI}. 
Further, the following bound holds for any $\ell \geq 1$:
\begin{equation}\label{eq:conv_mixed_vi_estimate}
   \|u_h^\ell  - u_h^*\|_{\mathrm{DG}}^2 + \|A^{-1/2} (\bm q_h^\ell - \bm q_h^*)\|_{L^2(\Omega)}^2  \lesssim  \frac{D_h(u_h^*, u_h^0)}{\sum_{k=1}^{\ell} \alpha_k} , 
\end{equation}
where $u_h^0 \in V_h^p$ is defined by 
\[ 
(u_h^0, q_h) = (\mathcal{U}(\psi_h^0), q_h)_h \quad \forall q_h \in V_h^p. 
\]
\end{theorem}
\rami{
\begin{remark} The result of \Cref{thm:conv_mixed_vi} also holds for the hybridized FOSPG method presented in \Cref{alg:main_alg_disc}.  This is true since the unique solutions $(\bm q_h^k, u_h^k)  \in \bm \Sigma_h^p \times V_h^p$ and $\psi_h^k \in V_h^p$ to \eqref{eq:nonlinear_subproblem} are also the unique solutions to \eqref{eq:nonlinear_subproblem_mixed} of \Cref{alg:main_alg_disc_mixed}.  We refer to the proof of \Cref{thm:exist_hybrid_FOSPG} for more details. 
\end{remark}
}
\begin{proof}
\textit{Step 1} (Reformulation). 
Considering \eqref{eq:quad_pt_bound} and recalling that $\mathcal{U} = (\mathcal{R}')^{-1}$ \eqref{setting:main}, we have that $u_h^k \in \tilde \Lambda_h$ for all $k \geq 0$ and that 
\begin{equation}
\psi^k_h(x^T_i) = \mathcal{R}'(u_h^k)(x_i^T)  \quad \forall 1\leq i\leq n_p , \; \forall T\in \mathcal{T}_h. 
\end{equation}
It also follows that for $u_h, v_h \in \tilde \Lambda_h$, the derivative of $D_h$ with respect to the first argument is 
\[ 
D_h'(u_h,v_h)(w) = (\mathcal{R}'(u_h) - \mathcal{R}'(v_h), w)_h  \quad  \forall w \in \Lambda_h .  \]
The above arguments allow us to write that for any $v \in \Lambda_h$
\begin{equation}
\frac{1}{\alpha^k} (\psi_h^{k} -\psi_h^{k-1}, v)_h = \frac{1}{\alpha^k} (\mathcal{R}'(u_h^k) - \mathcal{R}'(u_h^{k-1}), v)_h = \frac{1}{\alpha^k} D_h'(u_h^k, u_h^{k-1})(v).  
\end{equation}
Therefore, from \eqref{eq:nonlinear_subproblem_mixed_0}, we deduce that
\begin{align}
\label{eq:rewritea}
 \frac{1}{\alpha^k} D_h'(u_h^k, u_h^{k-1})(v) + (\nabla \cdot \bm q^k_h, v) - (f,v) = 0  \quad \forall v \in \Lambda_h. 
\end{align}
Using the lifting operator \eqref{eq:first_lift} in \eqref{eq:rewritea}, we obtain that 
\begin{equation}
\label{eq:D_h_rewrite0}
    \frac{1}{\alpha^k} D_h'(u_h^k, u_h^{k-1})(v) + (A^{-1} \bm L(v), \bm q_h^k) - (f,v)  = 0 \quad \forall v \in \Lambda_h.
 \end{equation}
We now rewrite $\bm q_h^k$ using the lifting operators \eqref{eq:first_lift}--\eqref{eq:second_lift} and \eqref{eq:nonlinear_subproblem_mixed_1}: 
\begin{equation}
    \bm q_h^k = \bm L(u^k_h) + \bm L_{\Gamma}(g).  \label{eq:expression_qh}
\end{equation}
Substituting the expression \eqref{eq:expression_qh} for $\bm q_h^k$ into \eqref{eq:D_h_rewrite0} 
 and using \eqref{eq:second_lift} yields
\begin{equation}
\frac{1}{\alpha^k} D_h'(u_h^k, u_h^{k-1})(v)  + (A^{-1} \bm L (v), \bm L (u_h^k)) - \langle g, \bm L (v) \cdot \bm n \rangle_{\partial \Omega}  - (f,v) = 0, \quad \forall v \in \Lambda_h. \label{eq:error_eq_vi}
\end{equation}
Since  $u_h^k \in \tilde \Lambda_h \subset \Lambda_h$, we see from the above that $u_h^k$ solves the following strictly convex and coercive minimization problem:\footnote{The coercivity of the problem follows from the fact $D_h (v, u_h^{k-1})$ is always positive and that $\|\bm L (\cdot) \|$ defines a norm on $V_h^p$ since if $(u_h, \nabla \cdot \bm r_h) = 0 \; \forall \bm r_h    \in \bm \Sigma_{h,\mathrm{div}}^{p}$ then $u_h = 0. $  This follows from the discrete inf-sup condition. 
} 
\begin{align} \label{eq:min_Dh}
\min_{ v \in \Lambda_h } \frac{1}{\alpha^k} D_h(v,u_h^{k-1}) + J_h(v);  \; \; J_h(v) =  \frac{1}{2} (A^{-1} \bm L (v), \bm L (v)) - \langle g, \bm L (v) \cdot \bm n \rangle_{\partial \Omega}  - (f,v). 
\end{align}
Therefore, $u_h^k$ is the unique solution to \eqref{eq:min_Dh}. From here, we see that for any $k \geq 1$ 
\begin{equation} \label{eq:energy_decay}
    J_h(u_h^k) \leq  \frac{1}{\alpha^k} D_h(u_h^{k-1},u_h^{k-1}) + J_h(u_h^{k-1}) =  J_h(u_h^{k-1}) . 
\end{equation}
\textit{Step 2}. (Convergence of $J_h(u_h^\ell)$). Here, we adapt  the arguments from \cite[proof of Theorem 4.13]{keith2023proximal}.  We test \eqref{eq:error_eq_vi} with $v = u_h^k$ and with $v = u_h^* \in \Lambda_h$ and subtract the resulting equations:  
\begin{multline}\label{eq:testing_u*}
\frac{1}{\alpha^k} D_h'(u_h^k, u_h^{k-1})(u_h^k - u_h^*)  + (A^{-1} \bm L (u_h^k  - u_h^*), \bm L (u_h^k)) \\  - \langle g, \bm L (u_h^k -u_h^*) \cdot \bm n \rangle_{\partial \Omega}  - (f,u_h^k -u_h^*) = 0.  
\end{multline}
We readily check that the three points identity \cite[Proposition 4.6]{keith2023proximal} also holds for the discrete distance: \begin{align*}
 D_h'(u_h^k, u_h^{k-1})(u_h^k - u_h^*)  & = 
 D_h(u_h^*, u_h^k) - D_h(u_h^*, u_h^{k-1}) + D_h(u_h^k, u_h^{k-1}). 
 \end{align*}
Using the definition of $J_h$ and the above in \eqref{eq:testing_u*}, we obtain 
\begin{multline}
 D_h(u_h^k, u_h^{k-1}) 
 +  (D_h(u^*, u_h^k) - D_h(u_h^*, u_h^{k-1}))  
\\  +  \alpha_k \left( J_h(u_h^k) - J_h(u_h^*) +\frac12 \|A^{-1/2} \bm L(u_h^k - u_h^*)\|^2 \right) 
 \leq 0. 
\end{multline}
Summing the above from $k=1, \ldots, \ell$ yields 
\begin{equation}
D_h(u^*, u_h^\ell) + \sum_{k = 1}^{\ell} D_h(u_h^k, u_h^{k-1}) + \sum_{k=1}^{\ell} \alpha_k (J_h(u_h^k) - J_h(u_h^*) ) \leq D_h(u^*, u_h^0).  
\end{equation}
Utilizing \eqref{eq:energy_decay}, we obtain that 
\begin{align}
  D_h(u^*, u_h^\ell) +  \sum_{k = 1}^{\ell} D_h(u_h^k, u_h^{k-1}) +  (J_h(u_h^\ell) - J_h(u_h^*) ) \sum_{k=1}^{\ell} \alpha_k  \leq D_h(u^*, u_h^0).    
\end{align} 
Therefore, 
\begin{equation} \label{eq:conv_in_energy}
J_h(u_h^\ell)  \leq  J_h(u_h^*)  +  \frac{D_h(u^*, u_h^0)}{\sum_{k=1}^{\ell} \alpha_k} .    
\end{equation}
\textit{Step 3}. (Convergence of $\bm q_h^*$ and $u_h^*$) Using the definitions of $\bm L$ and $\bm L_\Gamma$ \eqref{eq:first_lift}--\eqref{eq:second_lift}, the discrete solution $u_h^*$ satisfies 
\begin{align}
(A^{-1} \bm L(v-u_h^*), \bm L(u_h^*)) - (f, v-u_h^*)  - \langle g, \bm L(v-u_h^*) \cdot \bm n \rangle_{\partial \Omega}  \geq 0, \quad \forall v \in \Lambda_h. 
\end{align}
Choosing $v = u_h^\ell$ and rewriting, the above reads 
\begin{equation}
J_h(u_h^\ell) - J_h(u_h^*)   - \frac{1}{2}\|A^{-1/2} \bm L(u_h^\ell - u_h^*)\|^2 \geq 0 .
\end{equation}
Rearranging and using \eqref{eq:conv_in_energy}, we obtain 
\begin{equation}
\frac{1}{2}\|A^{-1/2} \bm L(u_h^\ell - u_h^*)\|^2 \leq J_h(u_h^\ell) - J_h(u_h^*)   \leq \frac{D_h(u^*, u_h^0)}{\sum_{k=1}^{\ell} \alpha_k}  .
\end{equation}
Since $\bm q_h^* =  \bm L(u_h^*) + \bm L_{\Gamma} (g) $ and $\bm q_h^\ell =\bm L(u_h^\ell) + \bm L_{\Gamma} (g)$, we obtain the required bound on $\|A^{-1/2} (\bm q^\ell_h - \bm q_h^*)\|$. 
Utilizing \Cref{lemma:dg_to_L2}, one obtains the required bound on $\|u_h^\ell - u^*_h\|_{\mathrm{DG}}$.
\end{proof}
\subsection{Local mass conservation} \label{subsec:local_mass_conserv}
We now show that for a fixed mesh, the iterates of \Cref{alg:main_alg_disc_mixed} converge (with $\alpha_k$) to a discrete solution satisfying a local mass conservation property on elements that do not intersect the obstacle. 
\begin{corollary}[Local mass conservation] \label{cor:mass_conservation}
Fix $T  \in \mathcal{T}_h$. If the solution $u_h^*$ to \eqref{eq:discrete_mixed_VI} satisfies
\begin{equation}
\underline u < u^*_h  < \overline{u}\quad \text{ in } T,  
\label{eq:condition_in_T}
\end{equation}
then the following property holds 
\begin{equation}
    |(\nabla \cdot \bm q_h^\ell - f, 1)_{T}| \lesssim  \frac{D_h(u_h^*, u_h^0)}{\sum_{k=1}^{\ell} \alpha_k} \, h_T^{(2d-3)/2}. \label{eq:mass_conserv_one_element}
\end{equation}
\end{corollary}
\begin{proof}
Define the discrete function $\lambda_h^{\ell} = (\psi_h^{\ell-1} - \psi_h^{\ell})/\alpha^\ell \in V_h^p$ for $ \ell \geq 1 $. Further, let $\lambda_h^* = (\nabla \cdot \bm q_h^* -f) \in L^2(\Omega)$ be the discrete Lagrange multiplier.  Then, considering \eqref{eq:nonlinear_subproblem_mixed_0}, we have for any $v_h \in V_h^p$
\begin{align*}
(\lambda_h^\ell - \lambda_h^*, v_h) = (\nabla \cdot (\bm q_h^\ell - \bm q_h^*), v_h) = - (\bm q^\ell_h - \bm q_h^*, \nabla v_h )+ \sum_{E \in \mathcal{E}_h} \int_{E} ( \bm q_h - \bm q_h^*) \cdot \bm n_{E} [v_h]. 
\end{align*}
Considering \eqref{eq:negative_norm} and \rami{using a similar argument to \Cref{lemma:dg_to_L2}}
and \Cref{thm:conv_mixed_vi}, we obtain that 
\begin{equation}\label{eq:mass_conserv_0}
\|\lambda_h^\ell - \lambda_h^*\|_{H^1(\mathcal{T}_h)^*} \lesssim \|\bm q_h^\ell - \bm q_h^*\|_{L^2(\Omega)}
\lesssim
\frac{D_h(u_h^*, u_h^0)}{\sum_{k=1}^{\ell} \alpha_k}.
\end{equation}

Now, test \eqref{eq:nonlinear_subproblem_mixed_0} with $\chi_T$ the indicator function of the element $T\in \mathcal{T}_h$:
\begin{align} \label{eq:mass_conserv_corr}
(\nabla \cdot \bm q_h^\ell - f, 1)_{T} = (\lambda_h^\ell - \lambda_h^*, \chi_T) + (\lambda_h^*, 1)_{T}. 
\end{align}
Observe that if $T$ satisfies \eqref{eq:condition_in_T}, the second term above evaluates to zero. This can be seen from \eqref{eq:discrete_mixed_VI_1} where we test with $v_h = u_h^* \pm \delta \chi_T  \in  \Lambda_h$ for small enough constant $\delta$. Namely, one can choose  \rami{$\delta = \min (\inf_{ x \in T} ( u_h^* - \underline u) ,  \sup
_{x \in T} (u_h^* - \underline u ) ).$}
Bounding the first term in \eqref{eq:mass_conserv_corr} with $\|\lambda_h^k - \lambda_h^*\|_{H^1(\mathcal{T}_h)^*} \|\chi_T\|_{\mathrm{DG}}$, using that $\|\chi_T\|_{\mathrm{DG}} \lesssim h_T^{(2d-3)/2}$, and invoking \eqref{eq:mass_conserv_0} shows \eqref{eq:mass_conserv_one_element}. 

\end{proof}
\subsection{A priori error estimates} \label{subsec:error_estimates_p0}
\Cref{thm:conv_mixed_vi} allows us to complete the analysis and provide an error estimate between the iterations $(\bm q_h^k, u_h^k) \in  \bm \Sigma_{h,\mathrm{div}}^{p} \times \Lambda_h$ and the solution $(\bm q^*, u^*) = (-A\nabla u, u)$ of \eqref{eq:model_problem}.  The proof of such a bound follows from the triangle inequality $\|\bm q_h^\ell - \bm q^*\| \leq \|\bm q_h^\ell - \bm q_h^*\| + \|\bm q_h^* - \bm q^*\|$ and \eqref{eq:conv_mixed_vi_estimate} after establishing an estimate for $\|\bm q_h^* - \bm q^*\|$. In Theorem \ref{cor:complete_story_p0}, we provide an estimate for the latter error when $p=0$. For $p \geq 1$, the convergence of discrete mixed VI \eqref{eq:discrete_mixed_VI} requires further investigation.
\begin{theorem}[Convergence of the discrete mixed VI for $p=0$]\label{cor:complete_story_p0} 
Let $(\bm q_h^*, u_h^*) \in \bm \Sigma_{h,\rm{div}}^0 \times V_h^0$ be the solution to the discrete mixed VI \eqref{thm:conv_mixed_vi}. 

\textit{(Anisotropic diffusion)}. If $\underline u , \overline{u} \in \R$ and $(\bm q^*, u^*) = ( -A\nabla u , u) \in H(\mathrm{div}, \Omega) \times L^2(\Omega)$ where $u$ solves \eqref{eq:model_problem}, then  
\begin{multline}\label{eq:conv_disc_VI_anis}
  \|u_h^*  - u^*\|_{L^2(\Omega)} + \|A^{-1/2} (\bm q_h^* - \bm q^*)\|_{L^2(\Omega)}  \lesssim  \inf_{\bm r_h \in \bm \Sigma_{h,\mathrm{div}}^0} \|A^{-1/2}(\bm q^* - \bm r_h) \|_{L^2(\Omega)}  \\ + h \|\nabla \cdot \bm q^*\|_{L^2(\Omega)} + \inf_{v_h \in V_h^0} \|u_h^* - v_h\|_{L^2(\Omega)}. 
 \end{multline} 
 
\rami{ (\textit{Unilateral obstacle problem}, $\overline{u} = \infty$). If $\underline u \in H^2(\Omega)$ and $(\bm q^*, u^*) = ( -A\nabla u , u)  \in H^1(\Omega) \times H^2(\Omega)$  where $u$ solves \eqref{eq:min_pb}, then 
\begin{equation}
\|u_h^*  - u^*\|_{L^2(\Omega)} + \|A^{-1/2} (\bm q_h^* - \bm q^*)\|_{L^2(\Omega)}  \lesssim C_{\mathrm{reg}} h, 
\end{equation}
where $C_{\mathrm{reg}}$ depends on $\|\bm q^*\|_{H^1(\Omega)}, \|u^* \|_{H^2(\Omega)}, \|\underline u\|_{H^2(\Omega)}$, and $ \|f\|_{L^2(\Omega)}$.  } 
 \end{theorem}
 \begin{proof}
 We make use of the interpolation operator introduced in \cite{ern2022equivalence} for simplicial elements  that satisfies $\bm I_h: H(\mathrm{div};\Omega) \rightarrow \bm \Sigma_{h, \mathrm{div}}^0$ with 
 \begin{align}
 (\nabla \cdot \bm{I}_h \bm q^*, v_h ) &= (\nabla \cdot \bm q^*, v_h) \quad \forall v_h \in V_h^0, \label{eq:commuting}\\ 
 \|\bm q^* - \bm{I}_h \bm q^* \|_{L^2(\Omega)} & \lesssim \inf_{\bm r_h \in \bm \Sigma_{h,\mathrm{div}}^0} \|\bm q^* - \bm r_h\|_{L^2(\Omega)}  + h \|\nabla \cdot \bm q^*\|_{L^2(\Omega)} \label{eq:approximation_estimate}. 
 \end{align}
 Note that no additional assumptions on the regularity of $\bm q^*$ are needed; see \cite[Theorem 3.2]{ern2022equivalence}. Considering \eqref{eq:discrete_mixed_VI_0}, we derive that 
 \begin{equation} \label{eq:conv_VI_0}
 (A^{-1} (\bm q^* - \bm q_h^*), \bm r_h) = (u^* - u_h^*, \nabla \cdot \bm r_h) \quad \forall \bm r_h \in \bm \Sigma_{h,\mathrm{div}}^0. 
 \end{equation}
 The above property and \eqref{eq:commuting} allow us to derive that 
 \begin{align}\label{eq:error_estimate_p0_0}
 \|A^{-1/2} (\bm q^* - \bm q_h^*)\|_{L^2(\Omega)}^2 &= (A^{-1} (\bm q^* - \bm q_h^*), \bm q^* - \bm{I}_h \bm q^*) + (u^* - u^*_h, \nabla \cdot ( \bm{I}_h \bm q^* - \bm q_h^*))     \end{align}
Denoting $\Pi_h u^* \in V_h^0$ as the $L^2$-projection onto $V_h^0$, there holds
\begin{align}
\label{eq:s0}
    (u^* - u^*_h, \nabla \cdot (\bm I_h \bm q^* - \bm q_h^*))
    & = 
        (\Pi_h u^* - u_h^*, \nabla \cdot ( \bm I_h \bm q^* - \bm q_h^*)) \\  
&  =   (\Pi_h u^* - u_h^*, \nabla \cdot  \bm q^* -  f) + (\Pi_h u^* - u_h^*, f-  \nabla \cdot \bm q_h^*) \nonumber \\ \nonumber
& \rami{\text{ = $Z_1 + Z_2.$}}
\end{align} 
We now consider two cases. 

\textit{Case 1}: If $u^*$ solves \eqref{eq:model_problem}, then 
\rami{$Z_1 = 0 $ since $\nabla \cdot \bm q^* - f= 0$ and by testing \eqref{eq:discrete_mixed_VI_1} with $v_h=\Pi_hu^* \in K_h$, $Z_2 \leq 0$.} A simple application of Cauchy--Schwarz inequality for the first term in \eqref{eq:error_estimate_p0_0} and the approximation property \eqref{eq:approximation_estimate} yields the required bound on $ \|A^{-1/2} (\bm q_h^* - \bm q^*)\|_{L^2(\Omega)}$ in \eqref{eq:conv_disc_VI_anis}. 
 
\textit{Case 2}: If $u^*$ solves \eqref{eq:min_pb},  then the result essentially follows from \cite[Section 4]{brezzi1978error} \rami{with key technical differences that we detail next}. \rami{ To simplify the write--up, define $\underline u_h$ as the piecewise constant function evaluating to $\underline u (x_1^T)$ on $T$ where we recall the definition of $\Lambda_h$ \eqref{eq:def_discrete_Lambda} in this case is $\Lambda_h = \{ v_h \in V_h^0, \underline u(x_1^T) \leq v_h, \; \forall T \in \mathcal{T}_h\}$ and that $n_0 = 1$. We begin by bounding $Z_2$ and write 
\begin{align*}
\Pi_h u^* - u_h^* = (\Pi_h u^* - \Pi_h \underline u) + \underline u_h  - u_h^*  +   (\Pi_h \underline u - \underline u_h)  
\end{align*} 
Testing \eqref{eq:discrete_mixed_VI_1} with test function $v_h=(\Pi_h u^* - \Pi_h \underline u) + \underline u_h \in \Lambda_h$ since $(\Pi_h u^* - \Pi_h \underline u) \geq 0$, we bound 
\begin{align}\label{eq:s-2}
Z_1 = (\Pi_h u^* - u_h^*, f-  \nabla \cdot \bm q_h^*)   \leq  (\Pi_h \underline u - \underline u_h , f - \nabla \cdot \bm q_h^*). 
\end{align}
Taking $x_1^T$ as the centroid of $T$, we write on an element $T$
\[ 
\Pi_h \underline u - \underline u_h = \Pi_h (\underline u - \underline u_I) + (\underline u_I - \underline u) (x_1^T), 
\]
where $\underline u_I$ denotes the linear Lagrange interpolant of $\underline u$ and where we used that $\Pi_h \underline u_I = u_I (x_1^T)$. The stability of $\Pi_h$ and the approximation properties of the Lagrange interpolant  (see for e.g. \cite[Proposition 4.6.3]{bernardi2024mathematics}) gives
\begin{align} \label{eq:s-1}
\|\Pi_h \underline u - \underline u_h\|_{L^2(T)} \leq \|\underline u - \underline u_I\|_{L^2(T)} + |T|^{1/2} \|\underline u - \underline u_I\|_{L^\infty(T)} \lesssim h_T^2 \|\underline u\|_{H^2(T)}. 
\end{align}
We then proceed by writing
\begin{align*}
(\Pi_h \underline u - \underline u_h , f - \nabla \cdot \bm q_h^*)_T  & = (\Pi_h \underline u - \underline u_h , f - \nabla \cdot \bm q^*)_T + (\Pi_h \underline u - \underline u_h ,  \nabla \cdot (\bm I_h \bm q^* - \bm q_h^*))_T  \\ 
& \lesssim h^2 \|\underline u\|_{H^2(T)}\|f - \nabla \cdot \bm q^*\|_{L^2(T)} +  h \|\underline u\|_{H^2(T)} \|\bm I_h \bm q^* - \bm q_h^* \|_{L^2(T)},  
\end{align*}
where we used the local inverse estimate $\|\nabla \cdot (\bm I_h \bm q^* - \bm q_h^*)\|_{L^2(T)}\lesssim h_T^{-1} \|\bm I_h \bm q^* - \bm q_h^*\|_{L^2(T)}$. Summing over all elements and using \eqref{eq:s-2}, we obtain that 
\begin{align}
   Z_1 \lesssim h^2  \|\underline u\|_{H^2(\Omega)}\|f - \nabla \cdot \bm q^*\|_{L^2(\Omega)} +  h \|\underline u\|_{H^2(\Omega)} \|\bm I_h \bm q^* - \bm q_h^* \|_{L^2(\Omega)}.  
\end{align}
For the second term $Z_2$ of \eqref{eq:s0}, we write 
\begin{align} \nonumber
Z_2 = (\Pi_h u^* - u_h^*, \nabla \cdot  \bm q^* -  f) 
& =  (\Pi_h (u^* - \underline u) - (u^* - \underline u) ,  \nabla \cdot  \bm q^* -  f) \\ \nonumber  & + (u^* - (u_h^* -\underline u_h + \underline u),  \nabla \cdot  \bm q^* -  f)  + (\Pi_h \underline u - \underline u_h,  \nabla \cdot  \bm q^* -  f) \\ 
& \nonumber \leq  (\Pi_h (u^* - \underline u) - (u^* - \underline u) ,  \nabla \cdot  \bm q^* -  f) + (\Pi_h \underline u - \underline u_h,  \nabla \cdot  \bm q^* -  f), 
\end{align}
where we tested \eqref{eq:mixed_VI} by $ v = u_h^* - \underline u_h + \underline u \geq \underline u$ to obtain the last inequality. Proceeding, we note that on each element $T$, either $\nabla \cdot \bm q^* - f =0$ or $u^* - \underline u = 0 $ on a subset of positive measure. For the latter case, applying the  extended Steklov--Poincar\'e inequality twice (see \cite[Lemma 3.30]{Ern:booki}) gives that $\| u^* -  \underline u \|_{L^2(T)} \lesssim h_T^2 \|u^* - \underline u \|_{H^2(T)}$. This observation, Cauchy--Schwarz inequality, the stability of $\Pi_h$, and \eqref{eq:s-1} yield that 
$$Z_2 \lesssim h^2 (\|u^*\|_{H^2(\Omega)} + \| \underline u \|_{H^2(\Omega)} ) \|\nabla \cdot \bm q^* - f\|_{L^2(\Omega)} .$$
Using the estimates on $Z_1$ and $Z_2$ in \eqref{eq:s0}, Cauchy--Schwarz  and triangle inequalities in \eqref{eq:error_estimate_p0_0} gives 
\begin{align}
\|A^{-1/2} (\bm q^* - \bm q_h^*)\|_{L^2(\Omega)}^2  & \lesssim \|A^{-1/2} (\bm q^* - \bm q_h^*)\|_{L^2(\Omega)} \|\bm I_h\bm q^*-\bm q^*\|_{L^2(\Omega)} \\ \nonumber & +  h^2 (\|u^*\|_{H^2(\Omega)} + \| \underline u \|_{H^2(\Omega)} ) \|\nabla \cdot \bm q^* - f\|_{L^2(\Omega)} \\ \nonumber &   + h \|\underline u\|_{H^2(\Omega)} ( \|\bm q^* - \bm I_h \bm q^*\|_{L^2(\Omega)} + \|\bm q^* - \bm q_h^*\|_{L^2(\Omega)}). 
\end{align}
The approximation bound  $\|\bm I_h\bm q^*-\bm q^*\|_{L^2(\Omega)} \lesssim  h \|\bm q^*\|_{H^1(\Omega)}$ and appropriate applications of Young's inequality provide the desired bound. 
}

Finally, the inf-sup stability of the finite element spaces yields the required bounds on  $\|u_h^* - u^*\|_{L^2(\Omega)}$; see also \cite[Theorem 2.3]{brezzi1978error}. 
\end{proof}
\rami{
\begin{remark}
The result of \Cref{cor:complete_story_p0} holds for the double obstacle problem if $\underline u , \overline{u} \in V_h^0$ where the proof simplifies. Extending this result for general double obstacles requires additional technical arguments and special assumptions on the mesh. We reserve this extension for future work. 
 \end{remark}
}
\section{Analysis of the Linearized sub-problems from Newton's method} \label{sec:linear_analysis}
The convergence property established in \Cref{sec:conv_mixed_vi} relies on condition \eqref{eq:reason_for_quads}, which holds for $p=0$ or for quadrilateral or hexahedral elements with \eqref{eq:reason_for_quads} enforced. For general meshes and spaces, we present here an analysis of the stability and convergence properties of the linearized sub-problems resulting from Newton's method. We view this as an important step towards establishing the convergence of the hybrid first-order system  proximal Galerkin method 
and its mesh independence properties.

For simplicity, we drop the superscripts $k$ to simplify the notation and we only consider the homogeneous problem and its solution At each step $k$, one solves a sequence of linearized subproblems of \eqref{eq:nonlinear_subproblem} with Newton's method of the following form:

Given $\psi \in L^{\infty}(\Omega), \tilde f \in L^2(\Omega)$, $\phi \in H^1_0(\Omega)$ and $\chi \in L^2(\Omega)$, find $(\bm  q_h, u_h, \hat u_h) \in \bm \Sigma_h^p \times V_h^p \times M^{p}_{h,0}$ and $\delta_h \in V_h^p$ solving 
\begin{subequations} \label{eq:linearized_subproblem}
    \begin{alignat}{2}
(\delta_h, v_h) -  \alpha \mathcal{B}_h (\bm q_h, (v_h, \hat v_h)) & = (\tilde f,v_h) && \quad \forall (v_h, \hat v_h) \in V_h^p \times M_{h,0}^p,  \label{eq:linearized_scheme0}   \\  
 (A^{-1} \bm q_h, \bm r_h) + \mathcal{B}_h (\bm r_h, (u_h, \hat u_h)) & = 0  && \quad \forall \bm r_h \in \bm \Sigma_h^p,  \\
   (u_h, q_h)  - (\mathcal{U}'(\psi) \delta_h ,q_h)&  = (\phi + \chi, q_h)  && \quad \forall q_h \in V_h^p.
\end{alignat}
\end{subequations}
\rami{We note that $\mathcal{U}$ is  Fr\`echet differentiable over $L^\infty(\Omega)$. }
With the definition of $\mathcal{A}_h$ \eqref{eq:def_form_A}, we write the above scheme as follows: Find $(\bm  q_h, u_h, \hat u_h) \in \bm \Sigma_h^p \times V_h^p \times M^{p}_{h,0}$ and $\delta_h \in V_h^p$  such that 
\begin{subequations}
\label{eq:summary_scheme}
\begin{alignat}{2}
 \alpha \mathcal{A}_h ((\bm q_h, u_h, \hat u_h), (\bm r_h, v_h, \hat v_h) ) + (\delta_h, v_h) & = (\tilde f,v_h) && \;  \; \forall (\bm r_h, v_h,\hat v_h) \in   \bm \Sigma_h^p \times V_h^p \times M^{p}_{h,0} ,\label{eq:summary_scheme0}\\ 
(u_h, q_h)  - (\mathcal{U}'(\psi) \delta_h ,q_h)&  = (\phi + \chi, q_h)  && \; \; \forall q_h \in V_h^p. \label{eq:summary_scheme1}
\end{alignat}
\end{subequations}
\subsection{Stability} We establish existence and uniqueness and stability estimates for \eqref{eq:summary_scheme}. The proof utilizes similar arguments to \cite[Section 4.3.2 and Section 5.5.3]{BBF13} and \cite[Appendix B]{keith2023proximal}. Since our setting is nonconforming, we provide the details. 
\begin{lemma}[First stability estimate]  
The linearized problem \eqref{eq:summary_scheme} admits a unique solution satisfying 
\begin{multline} \label{eq:stability_linearized}
\alpha \vvvert(\bm q_h, u_h, \hat u_h)\vvvert + \|\delta_h \|_{H^1(\mathcal{T}_h)^*} + \|\mathcal{U'}(\psi)^{1/2} \delta_h\|_{L^2(\Omega)} \\  \lesssim \|\tilde f \|_{H^1(\mathcal{T}_h)^*} + \|\nabla \phi\|_{L^2(\Omega)} + \|\mathcal{U}'(\psi)^{-1/2} \chi \|_{L^2(\Omega)} .
\end{multline}
\end{lemma} 
\begin{proof}
The existence of solutions follows by uniqueness since this is a square linear system. Therefore, it suffices to show the stability estimate. 

Testing \eqref{eq:summary_scheme0} with $(\bm r_h, u_h,\hat u_h )$ where $\bm r_h$ satisfies \eqref{eq:inf_sup} (see \Cref{lemma:inf_sup}) and \eqref{eq:summary_scheme1} with $\delta_h$,  subtracting the resulting equations, and   using  Cauchy-Schwarz inequality, we obtain 
\begin{multline} \label{eq:stability_1_0}
\alpha  \vvvert (\bm q_h, u_h, \hat u_h) \vvvert^2 + \|\mathcal{U}'(\psi)^{1/2} \delta_h\|_{L^2(\Omega)}^2 \\ \lesssim \|\tilde f\|_{H^1(\mathcal{T}_h)^*}\|u_h\|_{\mathrm{DG}}  + 
\|\mathcal{U}'(\psi)^{-1/2} \chi \|_{L^2(\Omega)} \|\mathcal{U}'(\psi)^{1/2} \delta_h\|_{L^2(\Omega)} + \|\nabla \phi\|_{L^2(\Omega)} \|\delta_h\|_{H^1(\mathcal{T}_h)^*}.
\end{multline}
Proceeding, we obtain a bound on $\|\delta_h\|_{H^1(\mathcal{T}_h)^*}$. 
Recall the definition of the $L^2$ projection \eqref{eq:local_L2_project} and that we can bound 
\begin{align}
    \|\delta_h \|_{H^1(\mathcal{T}_h)^*} =   \sup_{w \in H^1(\mathcal{T}_h)} \frac{(\delta_h, w)}{\|w\|_{\mathrm{DG}}} =\sup_{w \in H^1(\mathcal{T}_h)} \frac{(\delta_h, \Pi_h w)}{\|w\|_{\mathrm{DG}}}  \lesssim  \sup_{w \in H^1(\mathcal{T}_h)} \frac{(\delta_h, \Pi_h w)}{\|\Pi_h w\|_{\mathrm{DG}}}.  \label{eq:dual_norm_estimate}
 \end{align}
 Define $\hat w_h \in M_{h,0}^p$ as follows
 \begin{equation}\label{eq:triple_to_dg}
     \hat w_h =
     \begin{cases}
  \frac12 (\Pi_h w \vert_{T_E^1} + \Pi_h w \vert_{T_E^2})  & E \in \mathcal{E}_h^0 \quad  E =\partial T_E^1 \cap \partial T_E^2 ,  \\ 
  0   & E \subset \partial \Omega .
  \end{cases}
 \end{equation}
It is then easy to see that 
\[\vvvert(\bm  0, \Pi_h w, \hat w_h)\vvvert \lesssim \|\Pi_h w\|_{\mathrm{DG}}. \]
Testing \eqref{eq:summary_scheme0} with $(\bm 0,\Pi_h w, \hat w_h )$ yields 
\begin{align}
(\delta_h, \Pi_h w) & = (\tilde f, \Pi_h w) -\alpha \mathcal{A}_h((\bm q_h, u_h, \hat u_h), (\bm 0,\Pi_h w, \hat w_h ))  \\ & \lesssim \|\tilde f\|_{H^1(\mathcal{T}_h)^*} \|\Pi_h w \|_{\mathrm{DG}} + \alpha \vvvert(\bm q_h, u_h, \hat u_h)\vvvert \vvvert(\bm 0, \Pi_h w, \hat w_h )\vvvert ,  \nonumber
\end{align}
where we used continuity of the form $\mathcal{A}_h$ on the discrete spaces, see \eqref{eq:continuity_A}. 
This implies that 
\begin{align}\label{eq:bound_dual_norm}
    \|\delta_h \|_{H^1(\mathcal{T}_h)^*}  
    \lesssim
    \|\tilde f\|_{H^1(\mathcal{T}_h)^*} +  
    \alpha \vvvert(\bm q_h, u_h, \hat u_h)\vvvert. 
\end{align} 
We use the above in \eqref{eq:stability_1_0} and that $ \|u_h\|_{\mathrm{DG}} \lesssim \vvvert (\bm q_h, u_h, \hat u_h) \vvvert$ which follows from \eqref{eq:dg_to_triple}. Applications of Young's inequality yield the result.
\end{proof} 
\begin{remark}
In the linearized subproblems arising from Newton's method, the function $\chi =   \mathcal{U}(\psi)$. Thus, if $\| \mathcal{U}'(\psi)^{-1/2} \mathcal{U}(\psi) \|_{L^\infty(\Omega)}$ is bounded uniformly as $\mathrm{ess} \inf \psi \rightarrow -\infty$ or as $\mathrm{ess} \sup \psi \rightarrow \infty$, then one obtains a uniform estimate for the first two terms in \eqref{eq:stability_linearized}. This is the case for the unilateral obstacle problems and examples \eqref{eq:U3} and \eqref{eq:U4} with $\mathrm{ess} \inf \psi \rightarrow - \infty$. For the double obstacle problem, the situation is more delicate. In the error analysis below, we obtain bounds in the first two norms of \eqref{eq:stability_linearized} that are uniformly bounded with $\psi$ for $\mathcal{U}$ given in \eqref{eq:Z_1}--\eqref{eq:Z_2}. 
\end{remark}

\subsection{Error Analysis} 
In this subsection, we show error estimates between the discretized solutions of the linear subproblems \eqref{eq:summary_scheme} and their continuous counterparts resulting from applying Newton's method to the nonlinear subproblems \eqref{eq:saddle_point_problem}:  Find $( u, \delta) \in H^1_0(\Omega) \times L^2(\Omega)$ such that 
\begin{subequations} 
\label{eq:linearized_subproblem_cont}
\begin{alignat}{2}
\alpha  (A \nabla u , \nabla v) +  (\delta , v) &=  (\tilde f, v)   && \quad \forall v \in H^1_0(\Omega), \\
    (u , q) - (\mathcal{U}'(\psi) \delta , q) & = (\phi + \chi, q)  &&  \quad \forall q \in L^2(\Omega). 
 \end{alignat}
 \end{subequations}
The well-posedness of \eqref{eq:linearized_subproblem_cont} follows from an immediate application of the  Lax--Milgram Theorem since $\mathcal{U}'(\psi) > 0$ a.e. in $\Omega$. See \cite[Appendix B]{keith2023proximal} for a similar argument.

Since we are interested in anisotropic diffusion, one can not assume $H^{3/2 + \eta}$-regularity for $\eta >0$ for the solution $u$ of the linearized subproblems \eqref{eq:linearized_subproblem_cont}. However, this regularity is needed for strong notions of consistency. In particular, it is required in order to use $(-A\nabla u, u,\hat u)$ as the first argument in the form $\mathcal{A}_h$, as is done in \cite{egger2010hybrid}. This is due to the presence of the normal flux on element boundaries in the form $\mathcal{B}_h$.   

Here, we use the tools developed in \cite{ern2021quasi} to define a suitable extension of the form $\mathcal{A}_h$ that it is well defined 
for the following space. For $s \in (0,1/2)$, 
\begin{equation}
    V^{s} = \{ u \in H^{1+s} (\Omega): \,\, \nabla \cdot (A \nabla u) \in L^2(\Omega) \}.
\end{equation}
First, we define the normal trace of the flux by duality. 
This requires the definition of the \textit{face-to-element lifting} \cite[Lemma 3.1]{ern2021quasi}. Let $\rho = 2d/(d-2s)$, $\rho'$ be its conjugate ($1/\rho +1/\rho' = 1$), and let $\gamma_{\partial T}$ denote the Dirichlet trace operator. For any $E \subset \partial T$, there exists a lifting operator:
\begin{equation}
L^T_E: W^{\frac{1}{\rho}, \rho'}(E) \rightarrow W^{1, \rho'}(T), \quad \gamma_{\partial T} ( L_E^T(\varphi) ) = \begin{cases} \varphi & \mathrm{on} \,\,\, E, 
\\ 
0 & \mathrm{otherwise}. 
\end{cases}
\end{equation}
With this map, the normal trace of a function belonging to the space 
\begin{equation}
\bm S^d = \{\bm \tau \in L^\rho(\Omega): \,\, \rami{ \nabla \cdot \bm \tau{}_{\vert_{T}}} \in L^2(T), \quad \forall T \in \mathcal{T}_h\} 
\end{equation}
is defined as a functional on $W^{\frac1\rho, \rho'}(E)$ as follows 
\begin{equation}
\langle ( \bm \tau \cdot \bm n_K)_{\vert E}, \varphi\rangle_{E} = \int_T  (\bm \tau \cdot \nabla L_E^T(\varphi) + (\nabla \cdot \bm \tau) L_E^T(\varphi)  )\, \mathrm{d}x.  \label{eq:def_duality_pair}
\end{equation}
Observe that the second term above is well-defined since $W^{1,\rho'}(T) \hookrightarrow L^2(T)$ \cite[Theorem 2.31]{Ern:booki}.
Following \cite[Section 4.5]{ern2021quasi}, we then define the form $\eta_{\sharp}: \bm S^d \times V_h^p \times M_h^p \rightarrow \R$ 
\begin{equation}
    \eta_{\sharp}(\bm \tau , (v_h, \hat v_h)) = \sum_{T \in \mathcal{T}_h} \sum_{E  \in \mathcal{F}_T } \langle ( \bm \tau \cdot \bm n_T)_{\vert E}, (v_h - \hat v_h)_{\vert E} \rangle_{E}, 
\end{equation}
where $\mathcal{F}_T $ is the collection of all facets $E$ of the element $T$. The above form allows us to define the following extension of $\mathcal{A}_h$ given by 
$\tilde{\mathcal{A}}:(\bm S^d \times H^1(\mathcal{T}_h) \times L^2(\mathcal{E}_h)) \times (\bm \Sigma_h^p \times V_h^p \times M^{p}_{h,0}) \rightarrow \R $ 
\begin{multline}
    \tilde{\mathcal{A}}  ( (\bm q, u ,\hat u) , (\bm r_h, v_h ,\hat v_h)) = (A^{-1} \bm q,\bm r_h)_{\mathcal{T}_h} + \mathcal{B}_h (\bm r_h, (u, \hat u)) \\  
    + \eta_{\sharp}(\bm q, (v_h, \hat v_h)) - (\bm q, \nabla_h v_h)_{\mathcal{T}_h} . 
\end{multline}
Finally, recall that $\rho = 2d/(d-2s)$ and equip the space $\bm S^d$  with the following weighted norm 
\begin{align}
\|\bm \tau\|^2_{\bm{S}^d } = \sum_{T \in \mathcal{T}_h}  \left( h_T^{2s}\|A^{-1/2} \bm \tau \|^2_{L^{\rho}(T)} + h_T^{2} \|A^{-1/2} \nabla \cdot \bm \tau \|_{L^2(T)}^2 \right).  \label{eq:def_norm_Sd}
\end{align}
\begin{lemma}[Consistency] \label{lemma:consistency}
Let $u \in V^s \cap H^1_0(\Omega)$ and $\delta \in L^2(\Omega)$ be the solutions of the linearized sub-problem \eqref{eq:linearized_subproblem_cont}. Denote by $\bm \sigma(u) = - A \nabla u $.  Then, 
\begin{subequations}
\begin{align}
\alpha \tilde{\mathcal{A}} ( (\bm \sigma(u), u , u) , (\bm r_h, v_h ,\hat v_h))  + (\delta, v_h) &  =  (\tilde f, v_h) \label{eq:consistency_0} \\ 
(u, q_h) - (\mathcal{U}'(\psi) \delta ,q_h) & = (\phi + \chi , q_h),  \label{eq:consistency_1} 
\end{align}
\end{subequations}
for all $(\bm r_h, v_h, \hat v_h) \in \bm \Sigma_h^p \times V_h^p \times M^{k}_{h,0}$ and for all $q_h \in V_h^p$.
\end{lemma}
\begin{proof}
We first check that $(\bm \sigma(u), u , u) \in \bm S^d \times H^1(\mathcal{T}_h) \times L^2(\mathcal{E}_h)$. To this end,  observe that if $u \in V^{s}$ then $\nabla u \in H^s(T)^d$ for any element $T \in \mathcal{T}_h$. Then, by the Sobolev embedding theorem \cite[Theorem 2.31]{Ern:booki} and the assumption that $A \in L^{\infty}(\Omega)$, we have that $A \nabla u \in L^\rho(T)$ where $\rho = 2d/(d-2s)$ for any $T \in \mathcal{T}_h$. From trace theory, it also follows that $u \in L^2(\mathcal{E}_h)$. 

To show the consistency property, we use \cite[Lemma 4.16]{ern2021quasi} to obtain that
\begin{align}
\eta_{\sharp}(\bm \sigma(u) , (v_h, \hat v_h)) - (\bm \sigma (u) , \nabla_h v_h)_{\mathcal{T}_h}
=   (\nabla \cdot \bm \sigma(u), v_h)_{\mathcal{T}_h}. 
\end{align}
Further, observe that by definition 
\[ \mathcal{B}_h (\bm r_h , u,u ) = 
 (\bm r_h, \nabla u )_{\mathcal{T}_h} = - ( A^{-1} \bm \sigma (u), \bm r_h)_{\mathcal{T}_h}. 
\]
Collecting the above and using that $ \alpha \nabla \cdot \bm \sigma(u) =  \tilde f  - \delta \in L^2(\Omega)$ yield  \eqref{eq:consistency_0}. The derivation of \eqref{eq:consistency_1} is immediate since $V_h^p \subset L^2(\Omega)$. 
\end{proof}
The next Lemma shows that the form $\rami{\tilde{\mathcal{A}}}$ indeed defines an extension of $\mathcal{A}_h$. 
\begin{lemma}[Extension]\label{lemma:B_h_to_etah}
For $\bm q_h \in \bm \Sigma_h^p$, we have that 
\begin{equation}
  \mathcal{B}_h(\bm q_h , (v_h, \hat v_h)) = -\eta_{\sharp}(\bm q_h, (v_h, \hat v_h) ) + (\bm q_h, \nabla_h v_h)_{\mathcal{T}_h} \quad \forall (v_h , \hat v_h) \in V_h^p \times M_h^p.
  \end{equation}
\end{lemma}
\begin{proof}
Fix $T \in \mathcal{T}_h$ and let $E \in \mathcal{F}_T$ . By definition of the duality pairing \eqref{eq:def_duality_pair} and Green's theorem since $\bm q_h{}_{\vert T}$ is smooth, we write 
\begin{align*}
\langle ( \bm q_h \cdot \bm n_T)_{\vert E}, (v_h - \hat v_h)\vert_{E} \rangle &=  \int_T  (\bm q_h  \cdot \nabla L_E^T(v_h - \hat v_h) + (\nabla \cdot \bm q_h) L_E^T(v_h - \hat v_h)) \dd x \\
& = \int_{\partial T}  \bm q_h \cdot \bm n_{T} L_E^T(v_h - \hat v_h) \dd s  \\ 
& = \int_E \bm q_h \cdot \bm n_{T} (v_h - \hat v_h) \dd s . 
\end{align*}
The last equality holds since $L_E^T(v_h - \hat v_h)$ vanishes on $\partial T \backslash E$. Summing over all $E \in \mathcal{F}_T$ and over all $T \in \mathcal{T}_h$ yields the result.
\end{proof}
Proceeding, we derive error equations. Observe that there is no Galerkin orthogonality property here. However, the previous results allow us to adopt an argument similar to Strang's Second Lemma. 
\begin{lemma}[Error equations] Let $(\bm q_h, u_h, \hat u_h) \in \bm \Sigma_h^p \times V_h^p \times M^{p}_{h,0}$ and $\delta_h \in V_h^p$ solve \eqref{eq:linearized_subproblem}. 
For any $(\bm r_h , v_h, \hat v_h) \in \bm \Sigma_h^p \times V_h^p \times M^{p}_{h,0}$ and $\varphi_h \in V_h^p$, let $(\bm e_h , e_h, \hat e_h) = (\bm q_h - \bm r_h, u_h - v_h , \hat u_h - \hat v_h )$ and $e_{\delta,h} = \delta_h -\varphi_h$. 

The following error equations hold for all $  (\bm w_h , w_h , \hat w_h) \in   \bm \Sigma_h^p \times V_h^p \times M^{p}_{h,0} $ and for all $q_h \in V_h^p$,
\begin{align}
    \alpha \mathcal{A}_h((\bm e_h, e_h ,\hat e_h), (\bm w_h , w_h , \hat w_h ))  +    (e_{\delta, h}, w_h) & =  \mathcal{L}_1(\bm w_h , w_h , \hat w_h), \label{eq:first_error_eq}\\  
    (e_h, q_h ) - (\mathcal{U}'(\psi)  e_{\delta,h}, q_h) & = \mathcal{L}_2(q_h),  \label{eq:second_error_eq}
\end{align}
where $\mathcal{L}_1$ and $\mathcal{L}_2$ are given by 
\begin{align} \label{eq:consistency_err0}
\mathcal{L}_1(\bm w_h , w_h , \hat w_h) & = (\delta -\varphi_h, w_h )  \\ & + \alpha \tilde{\mathcal{A}} ( (\bm \sigma(u), u , u) , (\bm w_h, w_h ,\hat w_h))  - \alpha  \mathcal{A}_h((\bm r_h, v_h ,\hat v_h), (\bm w_h , w_h , \hat w_h )) \nonumber  \\ 
\mathcal{L}_2(q_h)  & = (u - v_h, q_h)- (\mathcal{U}'(\psi) (\delta - \varphi_h), q_h). \label{eq:consistency_err1}  
\end{align}
\end{lemma}
\begin{proof}
Since $(\bm q_h, u_h, \hat u_h) $ solve \eqref{eq:summary_scheme0}-\eqref{eq:summary_scheme1} and by linearity of $\mathcal{A}_h$, we have 
    \begin{align}
       \alpha \mathcal{A}_h((\bm e_h, e_h ,\hat e_h), (\bm w_h , w_h , \hat w_h )) =  (\tilde f, w_h )- (\delta_h, w_h) - \alpha  \mathcal{A}_h((\bm r_h, v_h ,\hat v_h), (\bm w_h , w_h , \hat w_h )).
    \end{align}
    With the consistency property \Cref{lemma:consistency}, 
    \begin{multline}
        \alpha \mathcal{A}_h((\bm e_h, e_h ,\hat e_h), (\bm w_h , w_h , \hat w_h )) =  (\delta -\delta_h, w_h )  \\ + \alpha \tilde{\mathcal{A}} ( (\bm \sigma(u), u , u) , (\bm w_h, w_h ,\hat w_h))  - \alpha  \mathcal{A}_h((\bm r_h, v_h ,\hat v_h), (\bm w_h , w_h , \hat w_h )). 
    \end{multline}
    With writing $\delta - \delta_h = \delta - \varphi_h -  e_{\delta,h}$, we obtain \eqref{eq:first_error_eq}. The equality given by \eqref{eq:second_error_eq} is immediate.
\end{proof}

\begin{theorem}[Convergence]\label{thm:conv_linearized}
Let $(\bm q_h, u_h, \hat u_h) $ solve \eqref{eq:summary_scheme} and let $u$ solve the linearized sub-problem \eqref{eq:linearized_subproblem_cont}.   Assume that $u \in V^{s} $ and let $\bm \sigma(u) = -A \nabla u$. The following estimate holds 
\begin{equation}
\vvvert (\bm \sigma(u)-\bm q_h, u- u_h , \hat u - \hat u_h )\vvvert + \|\delta - \delta_h\|_{H^1(\mathcal{T}_h)^* } \lesssim \mathcal{E}(u) + \mathcal{E}(\delta), \label{eq:main_error_estimate}
\end{equation}
where  $\mathcal{E}(u)$ and $\mathcal{E}(\delta)$ are best approximation errors given by
\begin{align}
    \mathcal{E}(u)  = & \inf_{(\bm r_h, v_h,\hat v_h)} \left(  \|u - v_h\|_{\mathrm{DG}} + \|\bm \sigma(u) - \bm r_h \|_{\bm S^d} + \vvvert (\bm \sigma (u) - \bm r_h , u-v_h, u - \hat v_h) \vvvert   \right) \\ 
    \mathcal{E}(\delta)  =& \,\,\,\,\,\,\, \inf_{\varphi_h}\,\,( \|\delta - \varphi_h\|_{H^1(\mathcal{T}_h)^*} +\|\mathcal{U}'(\psi)\|_{L^\infty(\Omega)}^{1/2} \| \delta - \varphi_h\|_{L^2(\Omega)}  ). 
\end{align}
The hidden constant in \eqref{eq:main_error_estimate} is independent of $\psi$.
 \end{theorem}
Error estimates are shown in the following corollary; see \Cref{appendix:error_estimate} for the proof. \begin{corollary}\label{cor:err_estimate}
Assume that $u \in V^s$, $\delta \in H^s(\Omega)$ for $s \in (0,1/2)$ \rami{and that $A$ is piecewise constant}, then 
\begin{align} \label{eq:error_rate}
    \vvvert (\bm \sigma(u) - \bm q_h , u - & u_h , u - \hat u_h )\vvvert + \|\delta - \delta_h\|_{H^1(\mathcal{T}_h)^* } \\ \nonumber & \lesssim  h^s (|u|_{H^{1+s}(\Omega)} + \|\mathcal{U}'(\psi)\|^{1/2}_{L^{\infty}(\Omega)} \|\delta\|_{H^s(\Omega)})\nonumber  \\ 
    & + h ( \|\nabla \cdot \bm \sigma(u) \|_{L^2(\Omega)} + \|\delta \|_{L^2(\Omega)}) .  \nonumber 
\end{align}
\end{corollary}
\begin{remark}
    If the solution has additional regularity $u \in H^{p+1}(\Omega)$ and $\delta \in H^{p-1}(\Omega)$ for $p \geq 1$, then optimal error estimates can be readily derived. One can also easily obtain an error estimate for $\|\mathcal{U}'(\psi)^{1/2} (\delta - \delta_h)\|_{L^2(\Omega)}$. However, the corresponding bound degenerates with $\psi \rightarrow -\infty$ or with $\psi \rightarrow \infty$. In contrast, observe that the bound  \eqref{eq:error_rate} remains bounded.   
\end{remark}
\begin{proof}(\Cref{thm:conv_linearized}) 
\textbf{Step 1}. (Estimating the consistency error $\mathcal{L}_1(\bm w_h, w_h , \hat w_h)$ for any $(\bm w_h, w_h , \hat w_h) \in \bm \Sigma_h^p \times V_h^p \times M^{p}_{h,0}$). Recalling the definition of $\mathcal{L}_1$ in \eqref{eq:consistency_err0}, we bound its first term by the definition of $\|\cdot\|_{H^{1}(\mathcal{T}_h)^*}$ and by \eqref{eq:dg_to_triple}.
\begin{equation}
|(\delta - \varphi_h , w_h)| \leq \|\delta - \varphi_h\|_{H^1(\mathcal{T}_h)^*}\|w_h\|_{\mathrm{DG}} \lesssim  \|\delta - \varphi_h\|_{H^1(\mathcal{T}_h)^*}\vvvert (\bm w_h, w_h, \hat w_h)\vvvert. 
\end{equation}
Consider the last two terms in \eqref{eq:consistency_err0} which we denote by $\alpha Q$. With Lemma \ref{lemma:B_h_to_etah}, we write 
    \begin{multline}
        Q  = (A^{-1}(\bm \sigma(u) - \bm r_h) , \bm w_h)_{\mathcal{T}_h} + \mathcal{B}_h(\bm w_h, u-v_h, u- \hat v_h)  \\  + \eta_{\sharp}(\bm \sigma(u) - \bm r_h , w_h , \hat w_h) - (\bm \sigma(u) - \bm r_h, \nabla_h w_h) = Q_1 + \ldots + Q_4. 
    \end{multline}
From Cauchy-Schwarz inequality and the definition of $\vvvert \cdot \vvvert$ \eqref{eq:def_triple_norm}, it is easy to see that 
\begin{align}
Q_1 + Q_4  & \leq \|A^{-1/2} (\bm \sigma(u) - \bm r_h)\|_{L^2(\Omega)} ( \| A^{-1/2} \bm w_h\|_{L^2(\Omega)} + \|A^{1/2} \nabla_h w_h \|_{L^2(\Omega)} )\\
\nonumber &   \lesssim  \|A^{-1/2} (\bm \sigma(u) - \bm r_h)\|_{L^2(\Omega)} \vvvert (\bm w_h ,w_h ,\hat w_h) \vvvert.
\end{align} 
For $Q_2$, \rami{we use the definition \eqref{eq:def_Bh} and Cauchy--Schwarz inequality}
\rami{
\begin{align*}
Q_2 & = (\bm w_h, \nabla_h (u-v_h))_{\mathcal{T}_h} - (v_h - \hat v_h, \bm w_h \cdot \bm n)_{\partial \mathcal{T}_h}  
\\ 
& \leq \|A^{-1/2} \bm w_h\|_{L^2(\Omega)} \|A^{1/2} \nabla_h (u-v_h)\|_{L^2(\Omega)}  \\ & \quad  + \left( \sum_{E\in\mathcal{E}_h} h_E^{-1} \|v_h - \hat v_h\|_{L^2(E)}^2 \right)^{1/2} \left(\sum_{E\in\mathcal{E}_h}  h_E \|\bm w_h \cdot \bm n\|_{L^2(E)}^2 \right)^{1/2}. 
\end{align*}
We now apply the discrete trace estimate \eqref{eq:normal_trace_discrete} and shape regularity to conclude that }
\begin{multline}
Q_2 \lesssim \|A^{-1/2} \bm w_h\|_{\mathcal{T}_h} \|A^{1/2} \nabla_h (u-v_h)\|_{\mathcal{T}_h} \\ + \left(\sum_{E \in \mathcal{E}_h} h_E^{-1} \|u-v_h-(u-\hat v_h)\|_{L^2(E)}^2  \right)^{1/2}  \|A^{-1/2} \bm w_h\|_{L^2(\Omega)}, 
\end{multline}
\rami{where we also used that 
$\|\bm w_h\|_{L^2(\Omega)} = \| A^{1/2} A^{-1/2} \bm w_h\|_{L^2(\Omega)} \leq \|A^{1/2}\|_{L^\infty(\Omega)} \|A^{-1/2} \bm w_h\|_{L^2(\Omega)}$.}
We utilize \cite[Lemma 3.2]{ern2021quasi} (with $p =\rho$ and $q=2$) to bound $Q_3$. In particular, we have that
\begin{multline}
\langle ( (\bm \sigma(u) - \bm r_h ) \cdot \bm n_T)_{\vert E}, (w_h - \hat w_h)_{\vert E} \rangle_{E}\\  \lesssim  (h_T^{s} \|A^{-1/2} (\bm \sigma(u) - \bm r_h)\|_{L^\rho(T)}  + h_T \|A^{-1/2}(\nabla \cdot (\bm \sigma(u) - \bm r_h))\|_{L^2(T)}) h_E^{-1/2} \|w_h - \hat w_h\|_{L^2(E)}.
\end{multline}
Summing over all $E \in \mathcal{F}_T$ and over all $T \in \mathcal{T}_h$, recalling the definition of $\|\cdot\|_{\bm S^d}$ \eqref{eq:def_norm_Sd} and applying Cauchy-Schwarz inequality for sums, we obtain 
\begin{align}
   Q_3 \lesssim \|\bm \sigma(u) - \bm r_h \|_{\bm S^d}  \vvvert (\bm w_h, w_h , \hat w_h) \vvvert. 
\end{align}
Collecting the estimates above, we conclude that 
\begin{multline}
\sup_{(\bm w_h , w_h ,\hat w_h) } \frac{|\mathcal{L}_1(\bm w_h , w_h ,\hat w_h)|}{\vvvert (\bm w_h , w_h ,\hat w_h) \vvvert} \lesssim \|\delta - \varphi_h\|_{H^1(\mathcal{T}_h)^*} +   \|\bm \sigma(u) - \bm r_h \|_{\bm S^d} + \vvvert (\bm \sigma(u) -\bm r_h , u-v_h, u - \hat v_h) \vvvert  . \label{eq:bounding_L1} 
\end{multline}
\textbf{Step 2.} (Bounding the dual norm of $e_{\delta,h}$). With similar arguments to \eqref{eq:dual_norm_estimate}, observe that 
\begin{align}
      \|e_{\delta,h} \|_{H^1(\mathcal{T}_h)^*}   
      \lesssim 
      \sup_{w \in H^1(\mathcal{T}_h)} \frac{|(e_{\delta,h}, \Pi_h w)|}{\|\Pi_h w\|_{\mathrm{DG}}}.   
\end{align}
Testing \eqref{eq:first_error_eq} with $(\bm 0, \Pi_h w, \hat w_h)$, with $\hat w_h$ defined in \eqref{eq:triple_to_dg}, we write 
\begin{equation}
(e_{\delta,h}, \Pi_h w) =   \mathcal{L}_1(\bm 0, \rami{\Pi_h w} , \hat w_h) -  \alpha \mathcal{A}_h((\bm e_h, e_h ,\hat e_h), (\bm 0 , \rami{\Pi_h w} , \hat w_h )) . 
\end{equation}
From \eqref{eq:bounding_L1}, continuity of $\mathcal{A}_h$, and that  $\vvvert(\bm 0, \Pi_h w, \hat w_h )\vvvert \lesssim \|\Pi_h w\|_{\mathrm{DG}}$, we obtain that 
\begin{equation}
  \|e_{\delta,h} \|_{H^1(\mathcal{T}_h)^*}   \lesssim   \sup_{(\bm w_h , w_h ,\hat w_h) } \frac{\mathcal{L}_1(\bm w_h , w_h ,\hat w_h)}{\vvvert (\bm w_h , w_h ,\hat w_h) \vvvert}  + \vvvert (\bm e_h, e_h ,\hat e_h)\vvvert.  \label{eq:bound_edelta}
\end{equation}
\textbf{Step 3.} (Bounding $\vvvert (\bm e_h, e_h ,\hat e_h)\vvvert$). We split the error into $(\bm e_h, e_h ,\hat e_h) = (\bm \eta_h, \eta_h, \hat \eta_h) +  (\bm \rho_h, \rho_h, \hat \rho_h)$ and $e_{\delta,h} =  \eta_{\delta,h} +  \rho_{\delta,h}$ where $ (\bm \eta_h, \eta_h, \hat \eta_h, \eta_{\delta,h}) $ solves \eqref{eq:first_error_eq}-\eqref{eq:second_error_eq} with $\mathcal{L}_2 = 0$ and  $ (\bm \rho_h, \rho_h, \hat \rho_h,\rho_{\delta,h}) $ solves \eqref{eq:first_error_eq}-\eqref{eq:second_error_eq} with $\mathcal{L}_1 = 0$. The existence of such a decomposition follows by the well-posedness of the formulation and the fact that $\mathcal{L}_1$ and $\mathcal{L}_2$ are bounded linear functionals on $\bm \Sigma_h^p \times V_h^p \times M_{h,0}^p$. 
 
\textit{Step 3.1}. We begin with bounding $\vvvert(\bm \eta_h, \eta_h, \hat \eta_h)\vvvert$.  From Lemma \ref{lemma:inf_sup}, we obtain that there exists $\bm w_h \in \bm \Sigma_h^p$ such that 
    \begin{equation}
    \vvvert (\bm \eta_h, \eta_h, \hat \eta_h) \vvvert^2 
    \lesssim \mathcal{A}_h((\bm \eta_h, \eta_h, \hat \eta_h), (\rami{\bm w_h} , \eta_h , \hat \eta_h )).\label{eq:step_31_0}
    \end{equation}
\rami{In addition, we have that 
\begin{equation}
    \vvvert (\bm w_h , \eta_h, \hat \eta_h) \vvvert \lesssim  \vvvert (\bm \eta_h, \eta_h, \hat \eta_h) \vvvert. \label{eq:step_31_1}
\end{equation}
For the proof of the above bound, we refer to the derivation of \eqref{eq:inf_sup_stab_bound} in  \Cref{sec:proof_Ah}. 
}
 Since $(\bm \eta_h, \eta_h, \hat \eta_h) $ solve \eqref{eq:first_error_eq}, we have 
    \begin{align}
       \alpha \mathcal{A}_h((\bm \eta_h, \eta_h ,\hat \eta_h), (\rami{\bm w_h} , \eta_h , \hat \eta_h )) = - (\eta_{\delta,h}, \eta_h) + \mathcal{L}_1( \rami{\bm w_h} , \eta_h , \hat \eta_h ).
    \end{align}
   Testing \eqref{eq:second_error_eq} with $\eta_{\delta,h}$ and recalling that $\mathcal{L}_2=0$ yield
\begin{align}
       \alpha \mathcal{A}_h((\bm \eta_h, \eta_h ,\hat \eta_h), (\rami{\bm w_h} , \eta_h , \hat \eta_h )) = - (\mathcal{U}'(\psi) \eta_{\delta,h}, \eta_{\delta,h}) + \mathcal{L}_1(\rami{\bm w_h} , \eta_h , \hat \eta_h ) \leq \mathcal{L}_1(\rami{\bm w_h} , \eta_h , \hat \eta_h ).
    \end{align}
    \rami{ 
From here, we combine \eqref{eq:step_31_0} and \eqref{eq:step_31_1} to derive that
\begin{equation} \nonumber
 \vvvert (\bm \eta_h, \eta_h, \hat \eta_h) \vvvert^2 \lesssim \frac{\mathcal{L}_1(\bm w_h , \eta_h , \hat \eta_h) }{\vvvert (\bm w_h , \eta_h , \hat \eta_h) \vvvert} \vvvert (\bm w_h , \eta_h , \hat \eta_h) \vvvert \lesssim  \sup_{(\bm w_h , w_h ,\hat w_h) } \frac{\mathcal{L}_1(\bm w_h , w_h ,\hat w_h)}{\vvvert (\bm w_h , w_h ,\hat w_h) \vvvert} \vvvert (\bm \eta_h , \eta_h ,\hat \eta_h) \vvvert. 
\end{equation}
}
Invoking \eqref{eq:bounding_L1}, we obtain 
\begin{equation} \label{eq:bound_eta}
     \vvvert (\bm \eta_h, \eta_h, \hat \eta_h) \vvvert \lesssim  \|\delta - \varphi\|_{H^1(\mathcal{T}_h)^*} +  \|\bm \sigma(u) - \bm r_h \|_{\bm S^d} + \vvvert (\bm \sigma(u) - \bm r_h , u-v_h, u - \hat v_h) \vvvert . 
\end{equation}

\textit{Step 3.2}. To bound $\vvvert(\bm \rho_h, \rho_h, \hat \rho_h)\vvvert$, we use that $(\bm \rho_h, \rho_h, \hat \rho_h, \rho_{\delta,h})$ solves \eqref{eq:first_error_eq}--\eqref{eq:second_error_eq} with $\mathcal{L}_1 = 0$. Namely, we can write for any $\bm w_h \in \bm \Sigma_h^p $
\begin{align} \label{eq:expansion_rho_h}
   \mathcal{A}_h((\bm \rho_h, \rho_h ,\hat \rho_h), (\bm w_h , \rho_h, \hat \rho_h )) & =   -(\rho_{\delta,h}, \rho_h) \\ 
   \nonumber & = -\mathcal{L}_2(\rho_{\delta, h}) - (\mathcal{U}'(\psi) \rho_{\delta,h}, \rho_{\delta,h}).  
\end{align}
That is, we derived that
\begin{align} \label{eq:expansion_rho_h1}
   \mathcal{A}_h((\bm \rho_h, \rho_h ,\hat \rho_h), (\bm w_h , \rho_h, \hat \rho_h ))  +  (\mathcal{U}'(\psi) \rho_{\delta,h}, \rho_{\delta,h}) =  -\mathcal{L}_2(\rho_{\delta, h}).  
\end{align}
With applications of H\"{o}lder's and  Poincar\'e's inequality along with \eqref{eq:dg_to_triple}, we  estimate
\begin{equation}  \label{eq:bound_L2}
|\mathcal{L}_2(\rho_{\delta,h} )| \lesssim  
\|u-v_h\|_{\mathrm{DG}} \|\rho_{\delta,h}\|_{H^1(\mathcal{T}_h)^*} + 
\|\mathcal{U}'(\psi)^{1/2} ( \delta - \varphi_h ) \|_{L^2(\Omega)}  \|\mathcal{U}'(\psi)^{1/2}\rho_{\delta, h}\|_{L^{2}(\Omega)}. 
\end{equation}
From \eqref{eq:bound_edelta} with $\mathcal{L}_1 = 0$, we have that 
\begin{equation} \label{eq:bound_L21}
\|\rho_{\delta,h}\|_{H^1(\mathcal{T}_h)^*}
\leq 
\vvvert(\bm \rho_h, \rho_h, \hat \rho_h)\vvvert. 
\end{equation} 
Along with Lemma \ref{lemma:inf_sup} and \eqref{eq:expansion_rho_h1}, this allows us to bound 
\begin{align}\label{eq:bound_L22}
  \vvvert(\bm \rho_h, \rho_h, \hat \rho_h)\vvvert 
 + \|\mathcal{U}'(\psi)^{1/2} \rho_{\delta,h} \|_{L^2(\Omega)}  
\lesssim 
\| u - v_h\|_{\mathrm{DG}} 
+ 
\|\mathcal{U}'(\psi)^{1/2} ( \delta - \varphi_h )\|_{L^2(\Omega)}. 
\end{align}
The final result is concluded by combining \eqref{eq:bound_edelta} with  \eqref{eq:bound_eta} and \eqref{eq:bound_L22}  along with applications of the triangle inequality. 
    \end{proof} 

\section{Numerical Results} \label{sec:numerics}
In this section, we consider anisotropic diffusion and obstacle examples. We implement \Cref{alg:main_alg_disc} and denote the converged solution by  $(u_h, \bm q_h , \hat u_h) \in  \bm \Sigma_h^p  \times V_h^p \times  M_{h,g}^p$. The   stopping criteria is $\|u_h^k - u_h^{k-1}\|_{L^2(\Omega)} < \mathtt{tol}$. 
All simulation results were obtained using the NGSolve software \cite{schoberl2014c++}. We share our code at  \href{https://github.com/ramimasri/FOSPG-first-order-system-proximal-Galerkin.git}{https://github.com/ramimasri/FOSPG-first-order-system-proximal-Galerkin.git}.
\subsection{Anisotropic diffusion}  We consider three benchmark examples taken from \cite{herbinbenchmark,li2010anisotropic}. While the standard hybrid mixed method \cite{arnold1985mixed}  may violate the discrete maximum principle, the FOSPG produces a solution $\mathcal{U}(\psi_h)$ satisfying the pointwise bounds everywhere.   \Cref{fig:anisotropy} shows the converged solutions $\mathcal{U}(\psi_h)$ for $p=2$.   
\begin{example}[Oblique flow] \normalfont \label{example:oblique}
We adapt this example from \cite{herbinbenchmark}. Consider $\Omega = (0,1)^2$ and   
\[ A = Q \begin{pmatrix}
    1 & 0 \\ 
    0 & \lambda
\end{pmatrix} Q^T, \quad Q = \begin{pmatrix}
    \cos(\theta) & - \sin(\theta) \\
    \sin(\theta) & \cos(\theta), 
\end{pmatrix},
\] 
where $\lambda = 10^{-3}$ and $\theta = 2\pi/9$. We let $f =0$ and
\[ 
g(x,0) = \begin{cases}
    1 & \text{if } x \leq 0.2\\
    2 - 5x & \text{if } 0.2 < x \leq 0.3, \\
    0.5 & \text{if } x > 0.3, \\
\end{cases} \;\; 
g(x,1) = \begin{cases}
    0.5 & \text{if } x \leq 0.7, \\
    4 - 5x & \text{if } 0.7 < x \leq 0.8, \\
    0 & \text{if } x > 0.8. \\
\end{cases}
\]
Further, we set $g(0,y) = g(y,0)$ and $g(1,y) =  g(y,1)$. 
\end{example}
\begin{example}[Vertical faults] \label{example:vertical_faults} \normalfont We consider the example from \cite{herbinbenchmark} and  consider $ 
\Omega = \Omega_1 \cup \Omega_2$  with $ \Omega_2 = \Omega \setminus \Omega_1, \; \Omega_1 = \Omega_1^\ell \cup \Omega_1^r$ where 
\begin{align*}
\Omega_1^\ell &= (0, 0.5] \times  \bigcup_{k=0}^{4} \left[ 0.05 + 0.2k, 0.15 + 0.2k \right], \; 
\Omega_1^r  = (0.5, 1) \times  \bigcup_{k=0}^{4} \left[ 0.2k, 0.2k+0.1  \right] .
\end{align*}
We set 
\[
A_{\vert_{\Omega_1}} = \begin{pmatrix}
    10^3 & 0 \\
    0 & 10
\end{pmatrix}, \quad 
A_{\vert_{\Omega_2}} = \begin{pmatrix}
    10^{-2} & 0 \\ 
    0 & 10^{-3}
\end{pmatrix},
\]
the source term $f = 0$, and the boundary data $g(x,y) =  1- x$. 
\end{example}
\begin{example}[Punctured domain] \label{example:punctured}\normalfont We adapt the example from \cite{li2010anisotropic}. The domain $\Omega = \Omega_0 \backslash \Omega_1 =  (0,1)^2 \backslash (4/9,5/9)^2 $. We let 
\[ A = Q \begin{pmatrix}
    \lambda  & 0 \\ 
    0 & 1
\end{pmatrix} Q^T, \quad Q = \begin{pmatrix}
    \cos(\theta) & - \sin(\theta) \\
    \sin(\theta) & \cos(\theta) 
\end{pmatrix} , 
\]   
where $\lambda = 10^3$ and $\theta = \pi \sin(x) \sin(y) $. We set $f=0$,  $u \vert_{\partial \Omega_1} = 1$, and $u \vert_{\partial \Omega_2} = 0$.
\end{example} 
\begin{figure}[t]
  \begin{center}
  \hspace{-13em}
     \begin{minipage}{0.3\textwidth}
 \centering
 \includegraphics[scale=0.15]{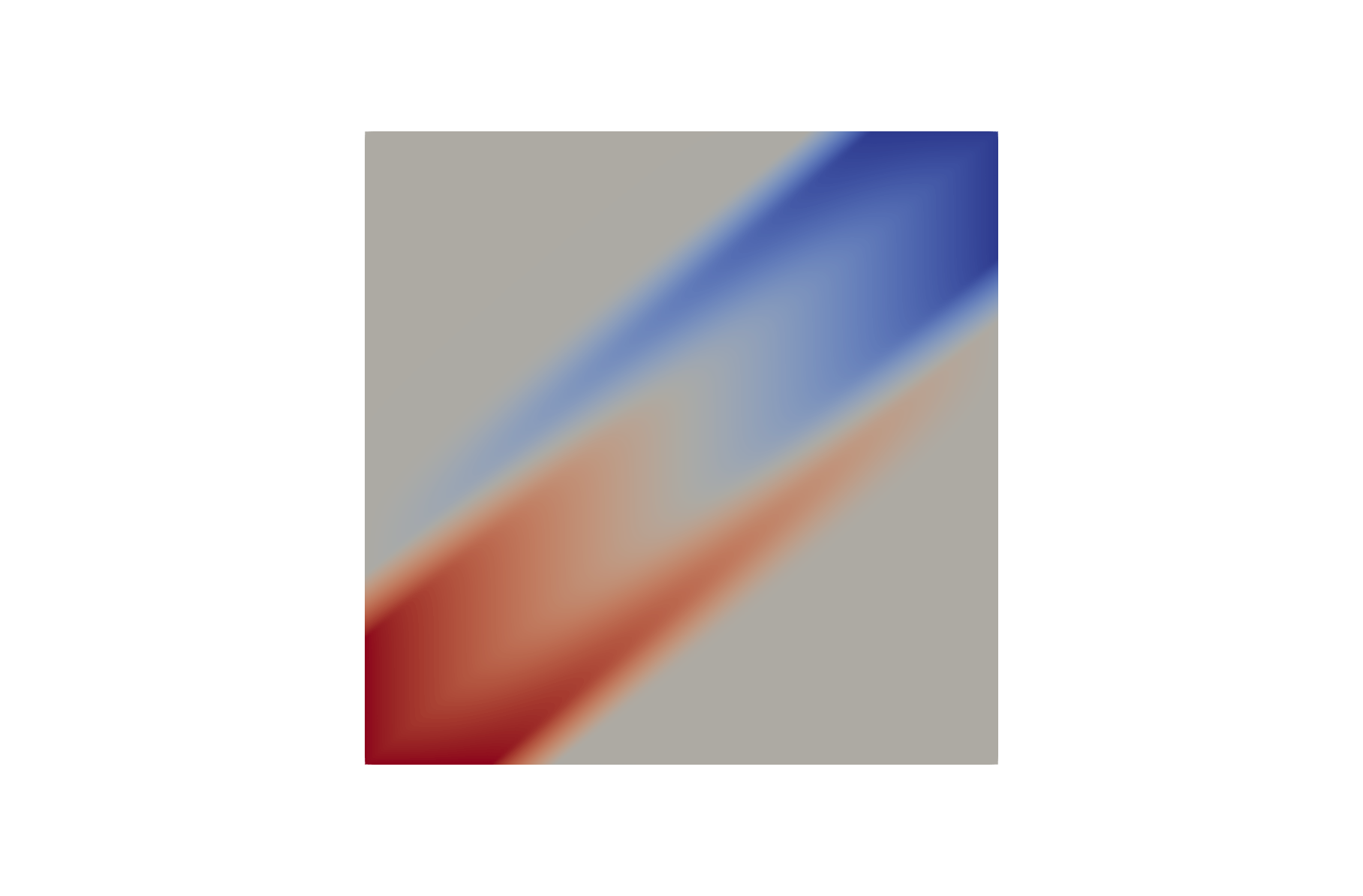}
 \end{minipage}
  \begin{minipage}{0.3\textwidth}
  \centering 
       \includegraphics[scale=0.15]{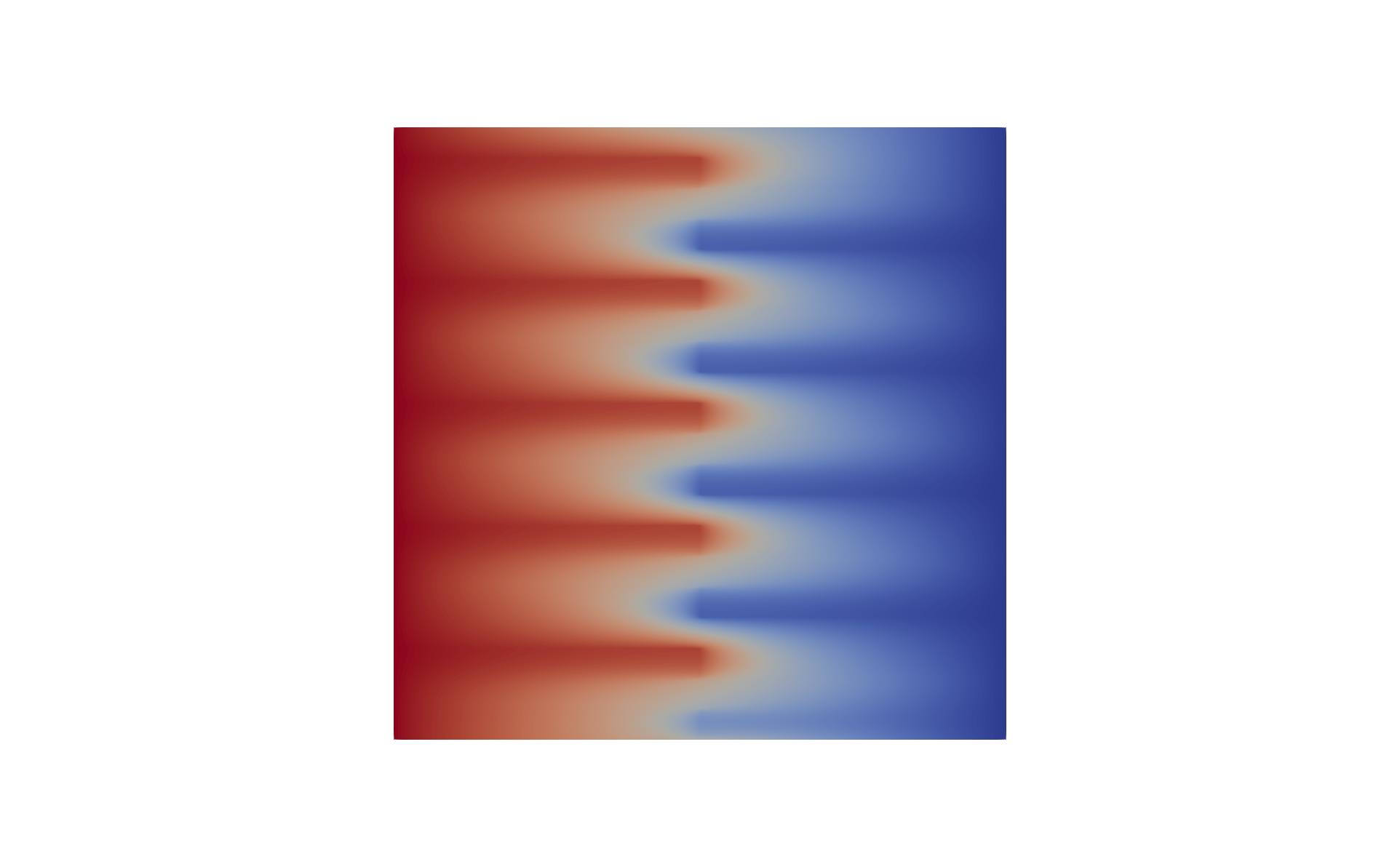}
 \end{minipage}
 \;
  \begin{minipage}{0.3\textwidth}
  \centering 
       \includegraphics[scale=0.15]{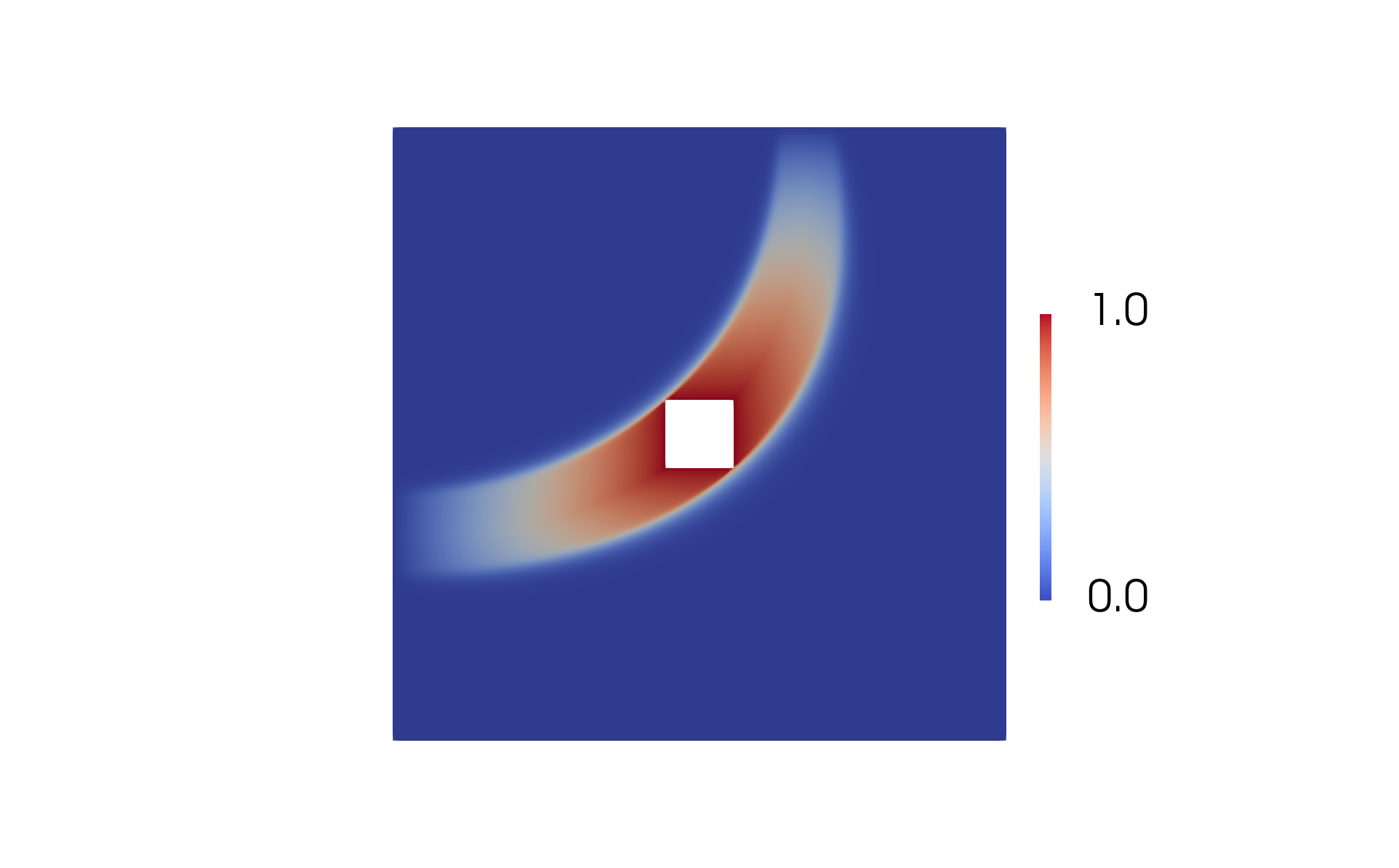}
 \end{minipage}
\caption{Computed solution $\mathcal{U}(\psi_h)$  for the oblique flow \Cref{example:oblique} (left),  for the vertical faults  \Cref{example:vertical_faults} (middle) and for the punctured domain \Cref{example:punctured} (right). We set $p =2$,  $h\approx 0.006$, $\mathrm{tol} = 1e-10$, and select the superposition operator $\mathcal{U}$ given in \eqref{eq:Z_2}. The solutions $\mathcal{U}(\psi_h)$ satisfy the discrete maximum principle by design. Further, the obtained flux approximation $\bm q_h$ retains the local mass conservation property. In particular, $\max_{T \in \mathcal{T}_h} |(\nabla \cdot \bm q_h -f ,1)_T |<  5\cdot 10^{-13}$ for all three examples. 
For \Cref{example:oblique} and \Cref{example:vertical_faults}, we set $\alpha^k = 4^k$ and $\mathcal S(\psi_h^k) = 0$. For \Cref{example:punctured}, we set $\alpha^0= 10^{-4}$, $\alpha^k = 1.5\alpha^{k-1}$, and use $\mathcal S(\psi_h^k)$ given in \eqref{eq:extra_stablization} with $\epsilon_1=\epsilon_2=0.1$.
} 
\label{fig:anisotropy}
  \end{center}
\end{figure}
We found \Cref{example:punctured} to be the most challenging, inspiring us to explore this example further. In particular, we consider the following choice for $\mathcal S(\psi_h^k)$ in  \eqref{eq:nonlinear_subproblem_1}:
\begin{equation} \label{eq:extra_stablization}
(\mathcal S(\psi_h^k), q_h) =   \epsilon_1 h^{p+1}(\psi_h^k, q_h) + \epsilon_2 h^{p+1} (\nabla_h \psi_h^{k}, \nabla_h q_h)   \quad \forall q_h \in V_h^p, 
\end{equation} 
where $\epsilon_1, \epsilon_2 \geq 0$.  This choice provided additional control on the broken $H^1$ norm of $\psi^k_h$ leading to a stable solution $\mathcal U(\psi_h)$.  For $p=0,1$, such stabilization is not needed; see \Cref{rem:mass_conservation}. For $p=1$, we use a quadrature rule of order $2$ that includes the element vertices for a more stable solution $\mathcal{U}(\psi_h)$ in all of this section. We remark that the choice of stabilization whether via special quadrature rules or additional terms like in \eqref{eq:extra_stablization} requires a detailed and additional study.

\Cref{fig:overshoots} compares the standard mixed method solution $u_h^{\mathrm{mixed}}$ \cite{egger2010hybrid} to the solutions $(u_h, \mathcal{U}(\psi_h))$ generated by \Cref{alg:main_alg_disc}. We observe that $u_h^{\mathrm{mixed}}$ violates DMP even with fine meshes for $p=0$. For $p=2$, mesh refinement improves the solution; however, $u_h^{\mathrm{mixed}}$ is still not bound preserving. This is also evident from \Cref{table:min_max}, which reports the minimum and maximum values of $u_h^{\mathrm{mixed}}$ and $\mathcal{U}(\psi_h)$ over the quadrature points. The classical solution $u_h^{\mathrm{mixed}}$ violates the DMP while  the solution $\mathcal{U}(\psi_h)$ remains between $0$ and $1$ everywhere. Furthermore, the solution $u_h$ has bound-preserving local averages, a property that can be optionally used in postprocessing to locally construct a bound-preserving $\tilde u_h$ (see \Cref{remark:limiter}). 
\begin{figure}
\centering
\begin{center}
   \hspace{-2em}
   \begin{minipage}{0.3\textwidth}\centering
 \begin{overpic}[scale=0.05]
     {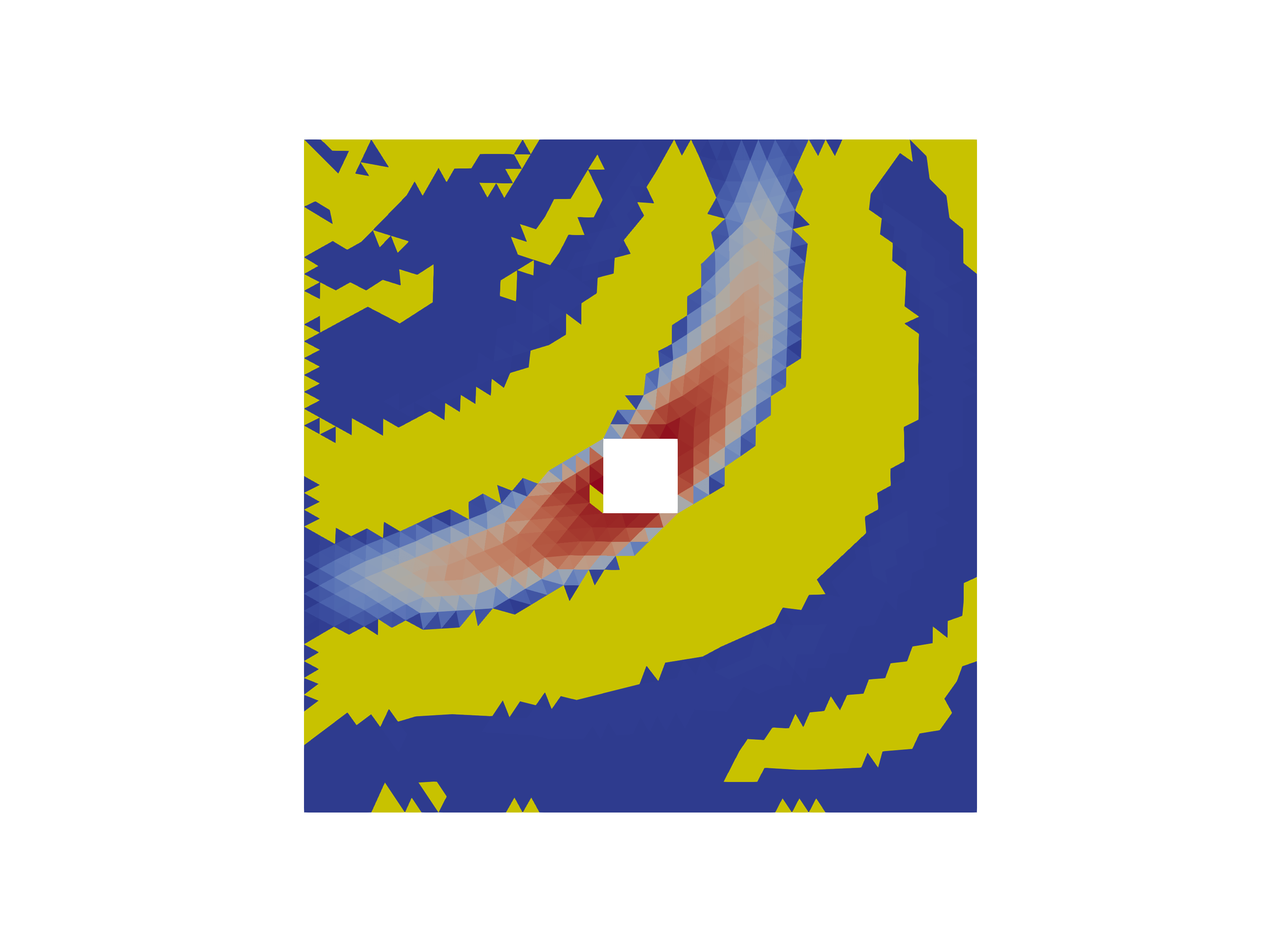}
     \put(44,71){$u_h^{\mathrm{mixed}}$}
     \put(-4.5,40){$h \approx 0.03$} 
     \put(-4.5,30){$p=0$}
 \end{overpic}
 \end{minipage}
  \begin{minipage}{0.3\textwidth}\centering
      \begin{overpic}[scale=0.05]
     {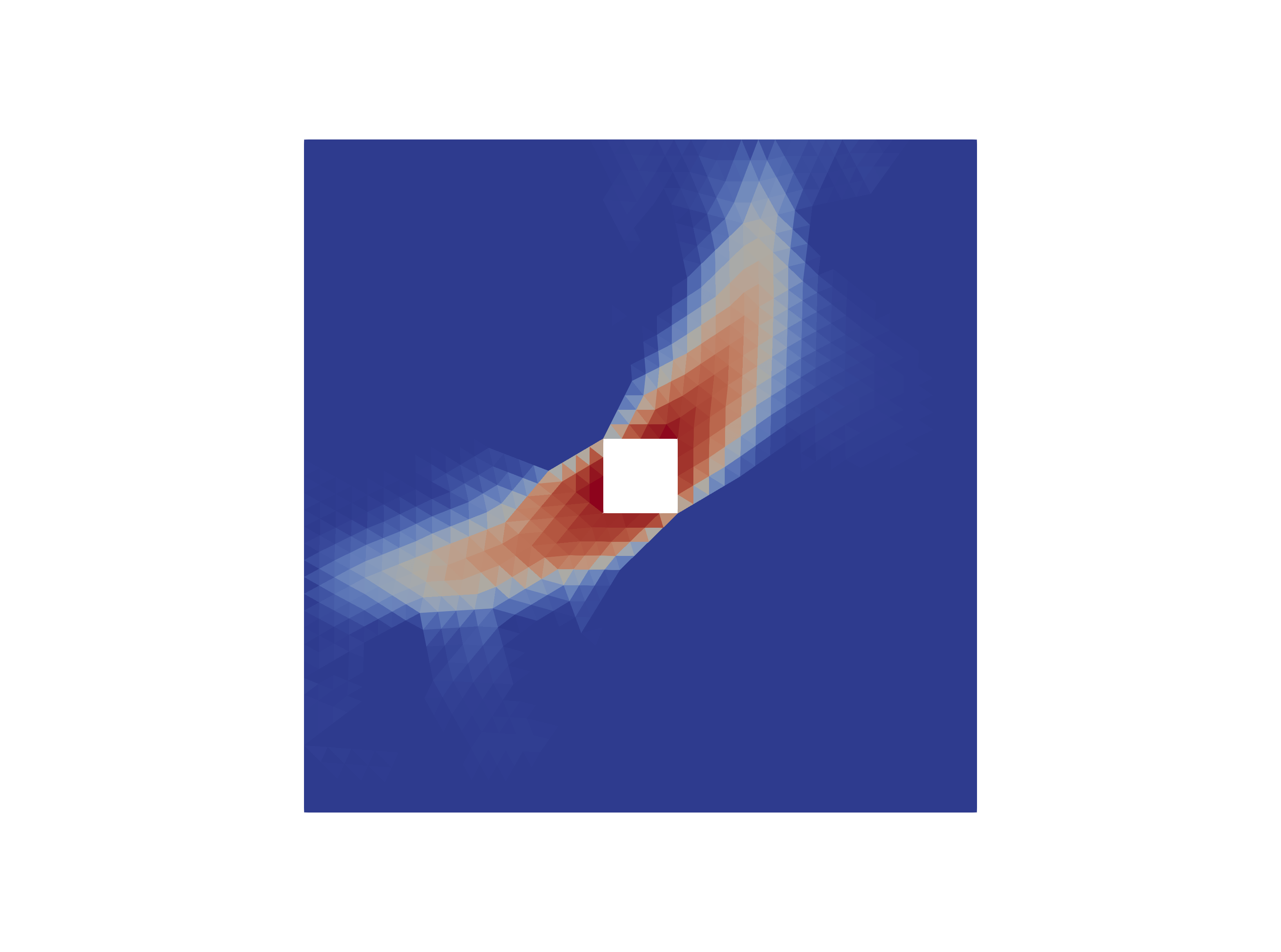}
     \put(48,71){$u_h$}
 \end{overpic}
 \end{minipage}
  \begin{minipage}{0.3\textwidth}\centering
     \begin{overpic}[scale=0.05]
     {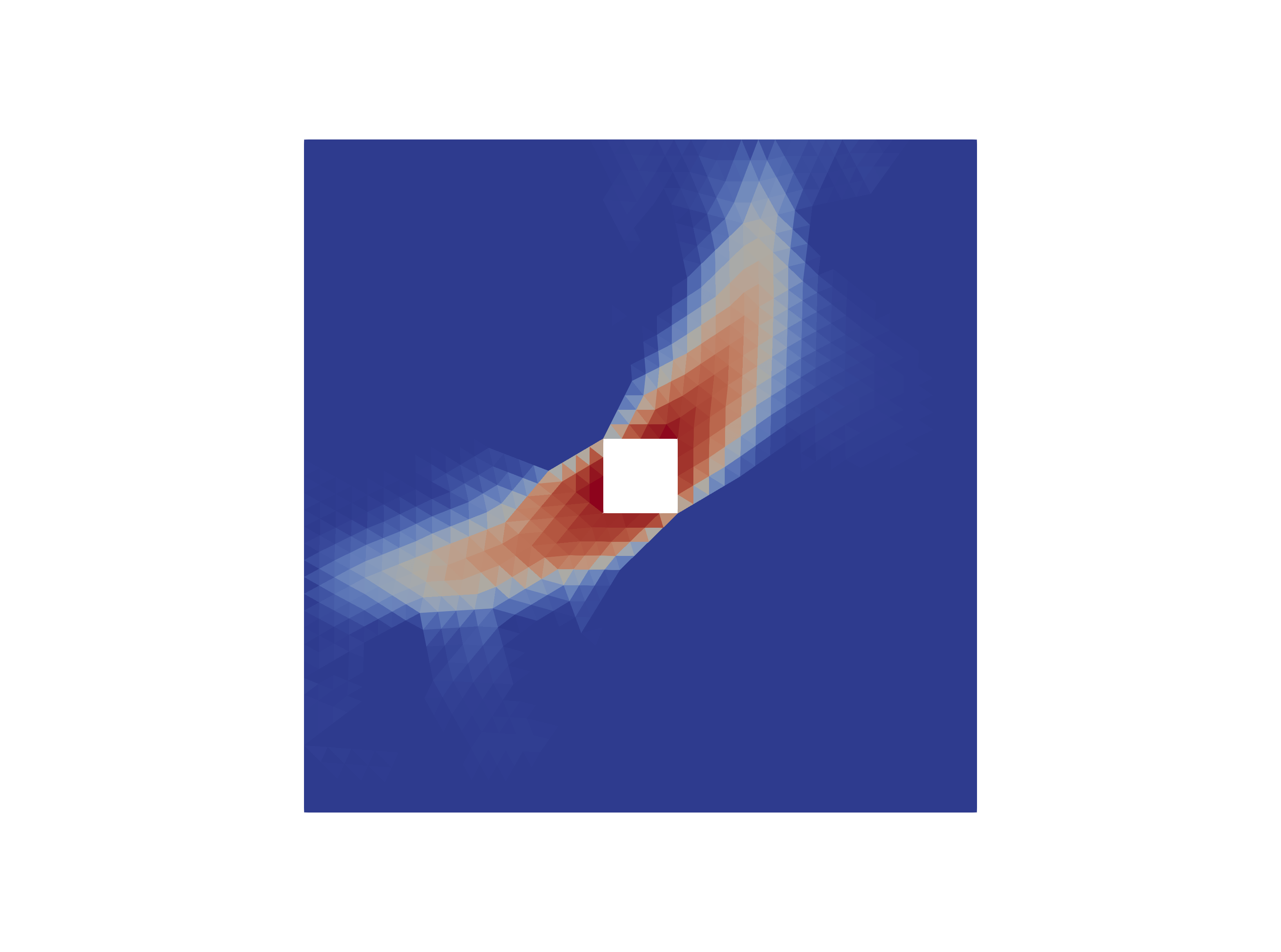}
     \put(45,71){$\mathcal{U}(\psi_h)$}
     \end{overpic}
 \end{minipage} 
 \end{center}
 \begin{center}
    \hspace{-2em}
     \begin{minipage}{0.3\textwidth}\centering
 \begin{overpic}[scale=0.05]{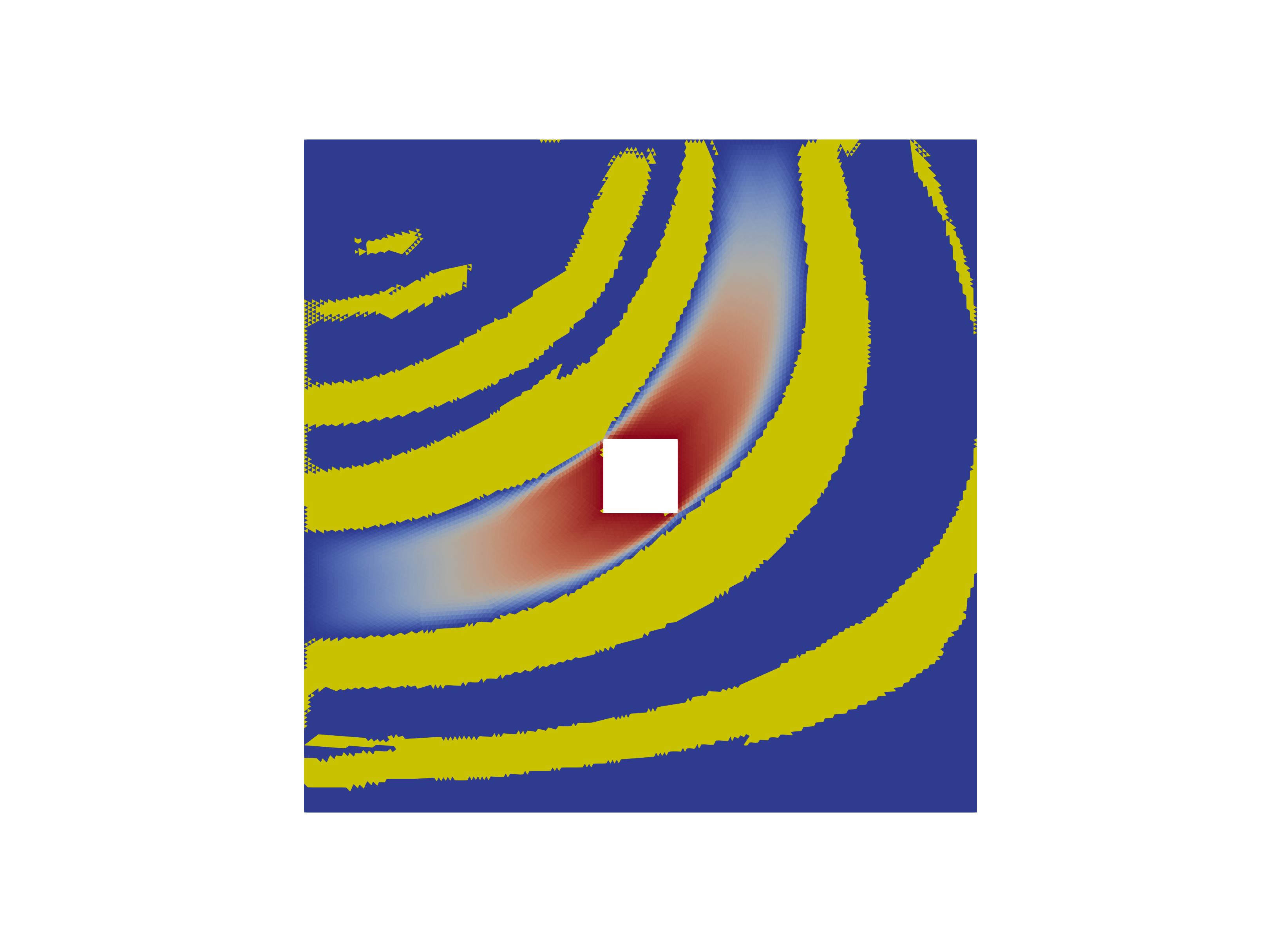}
  \put(-4.5,40){$h \approx 0.007$} 
\put(-4.5,30){$p=0$}
 \end{overpic}
 \end{minipage}
  \begin{minipage}{0.3\textwidth}\centering
     \includegraphics[scale=0.05]{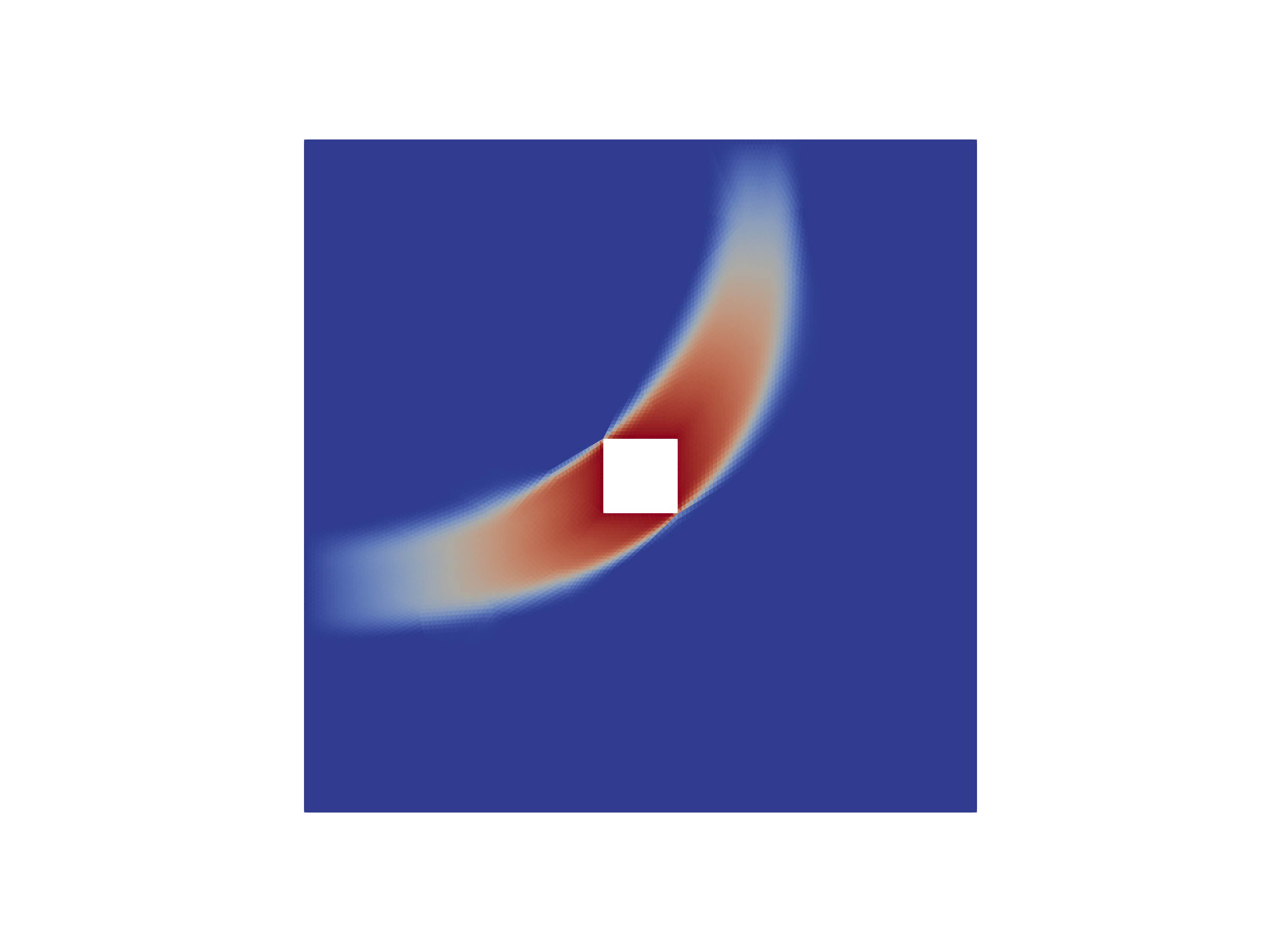}
 \end{minipage}
  \begin{minipage}{0.3\textwidth}\centering
     \includegraphics[scale=0.05]{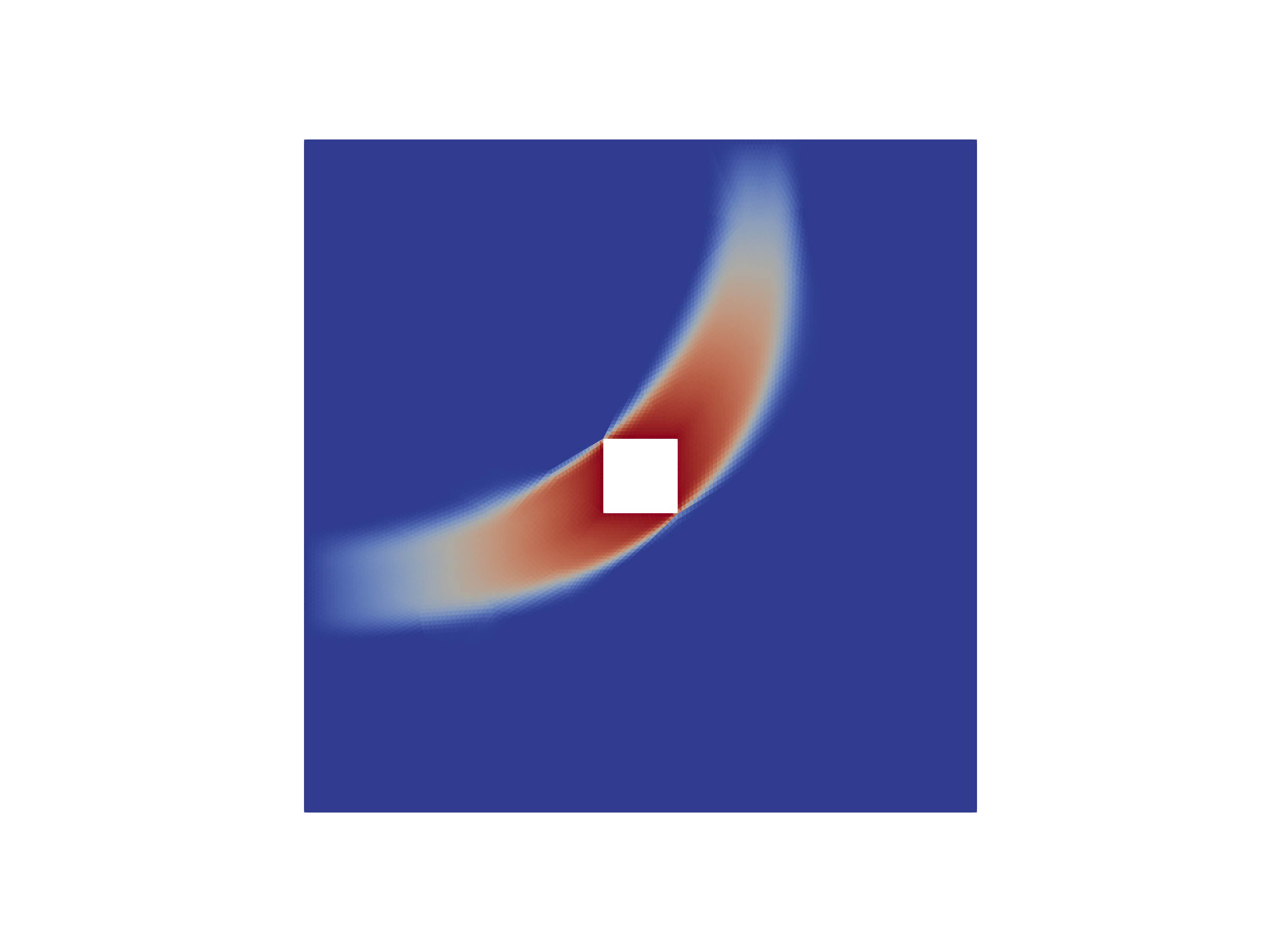}
 \end{minipage}
  \end{center}
  \begin{center}
 \hspace{-2em}
     \begin{minipage}{0.3\textwidth}\centering
 \begin{overpic}[scale=0.05]{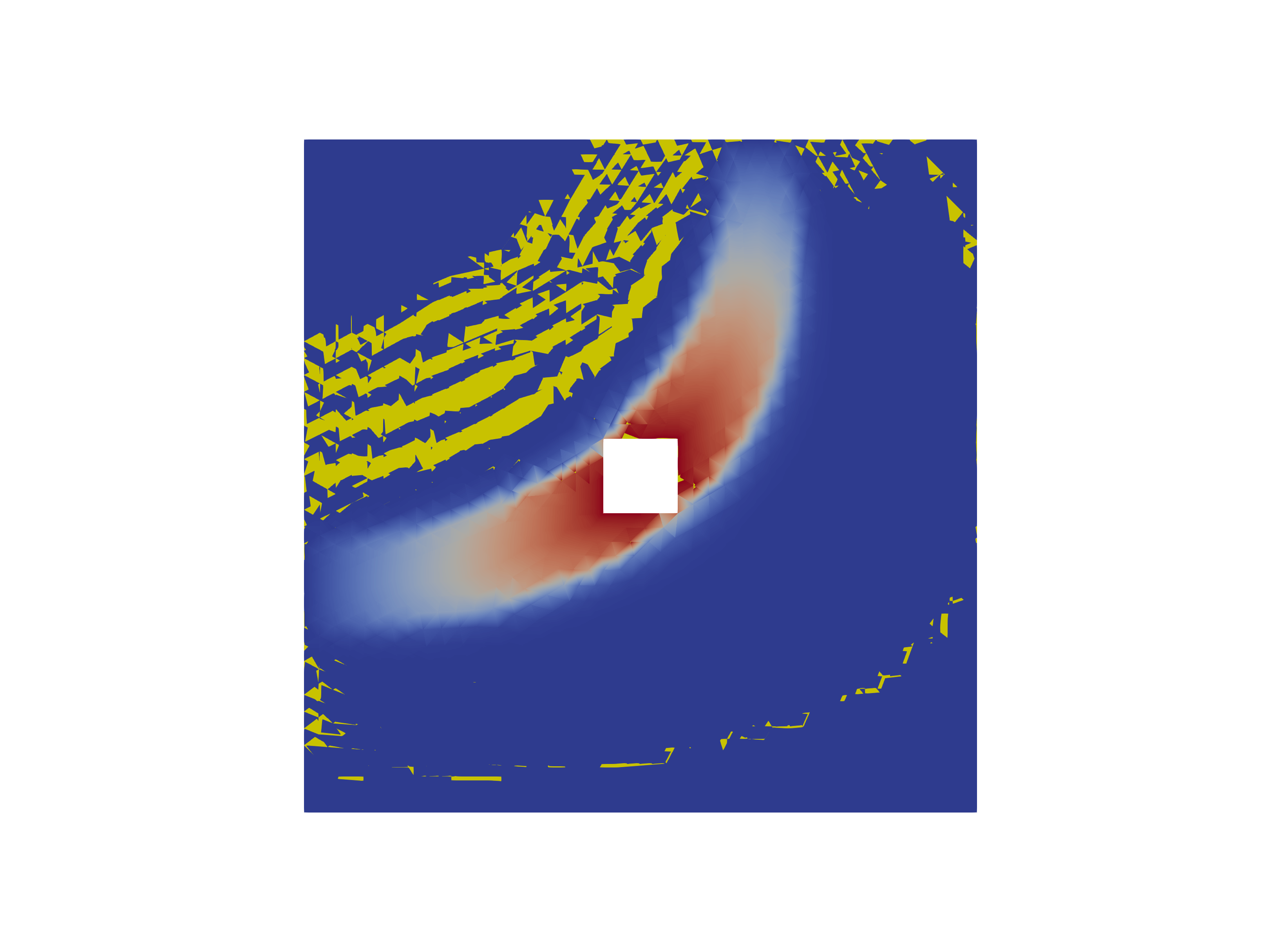}
  \put(-4.5,40){$h \approx 0.03$} 
\put(-4.5,30){$p=2$}
 \end{overpic}
 \end{minipage}
  \begin{minipage}{0.3\textwidth}\centering
     \includegraphics[scale=0.05]{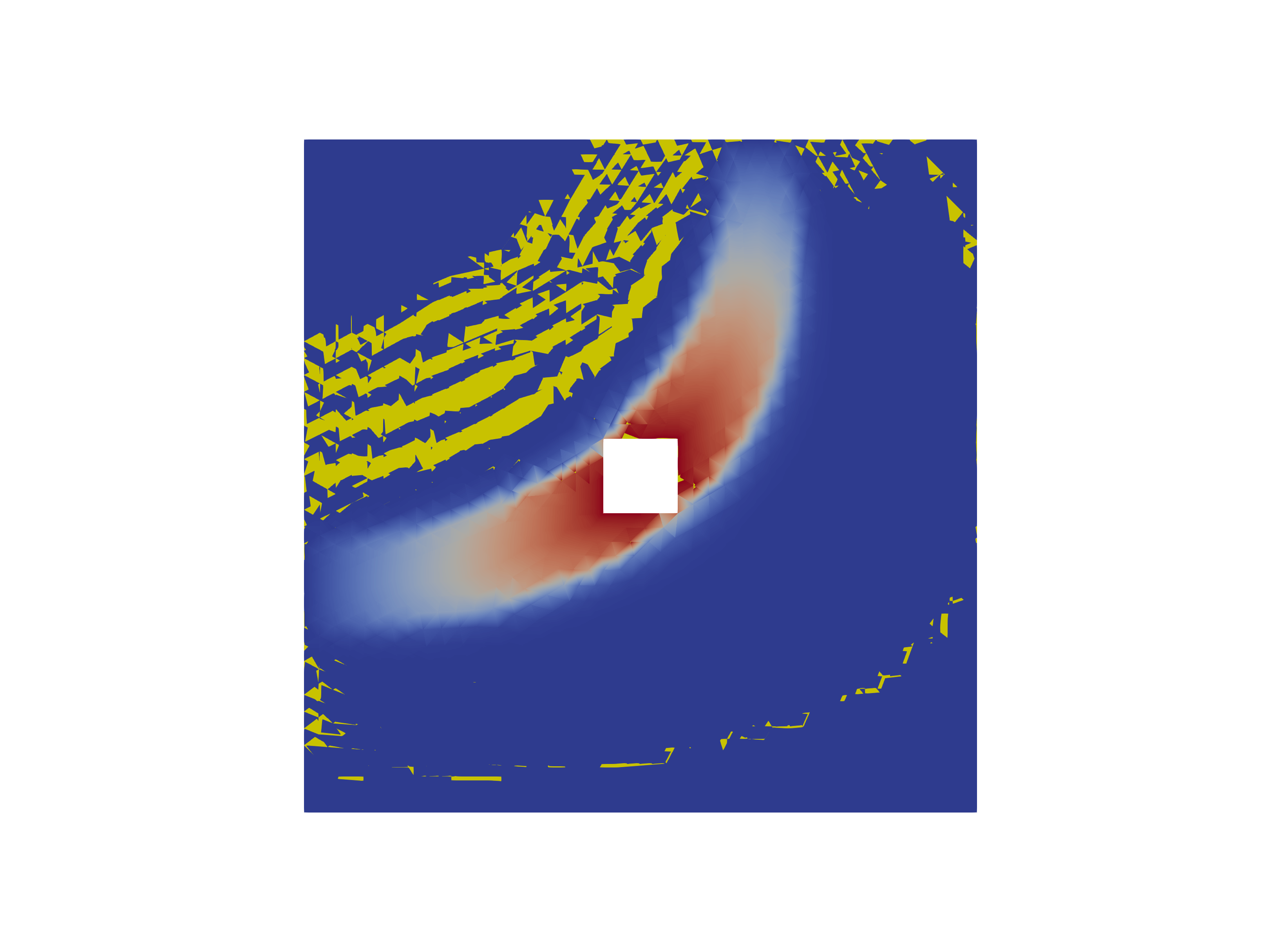}
 \end{minipage}
  \begin{minipage}{0.3\textwidth}\centering
\includegraphics[scale=0.05]{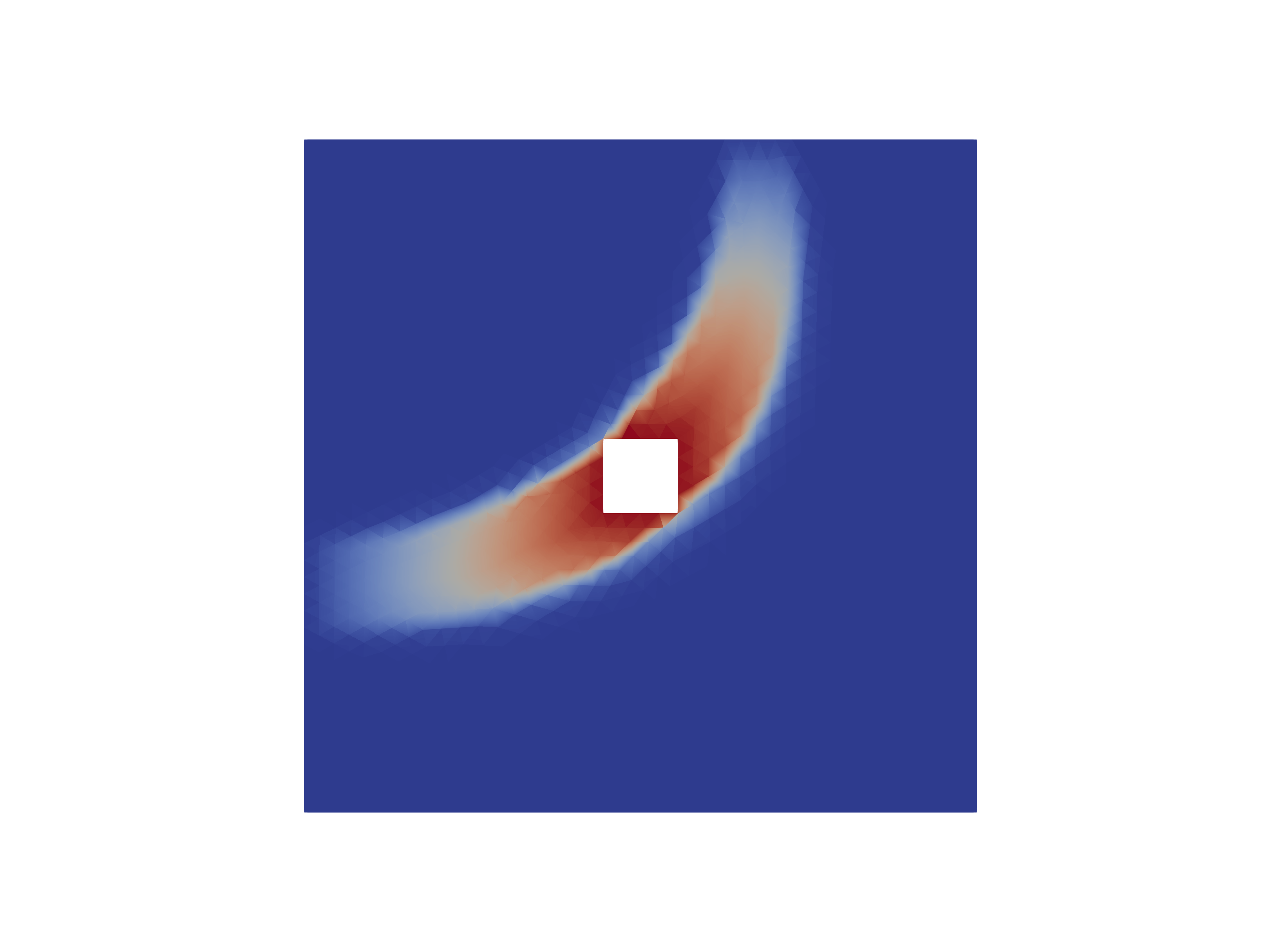}
 \end{minipage}
 \end{center}
   \begin{center}
 \hspace{-2em}
     \begin{minipage}{0.3\textwidth}\centering
 \begin{overpic}[scale=0.05]{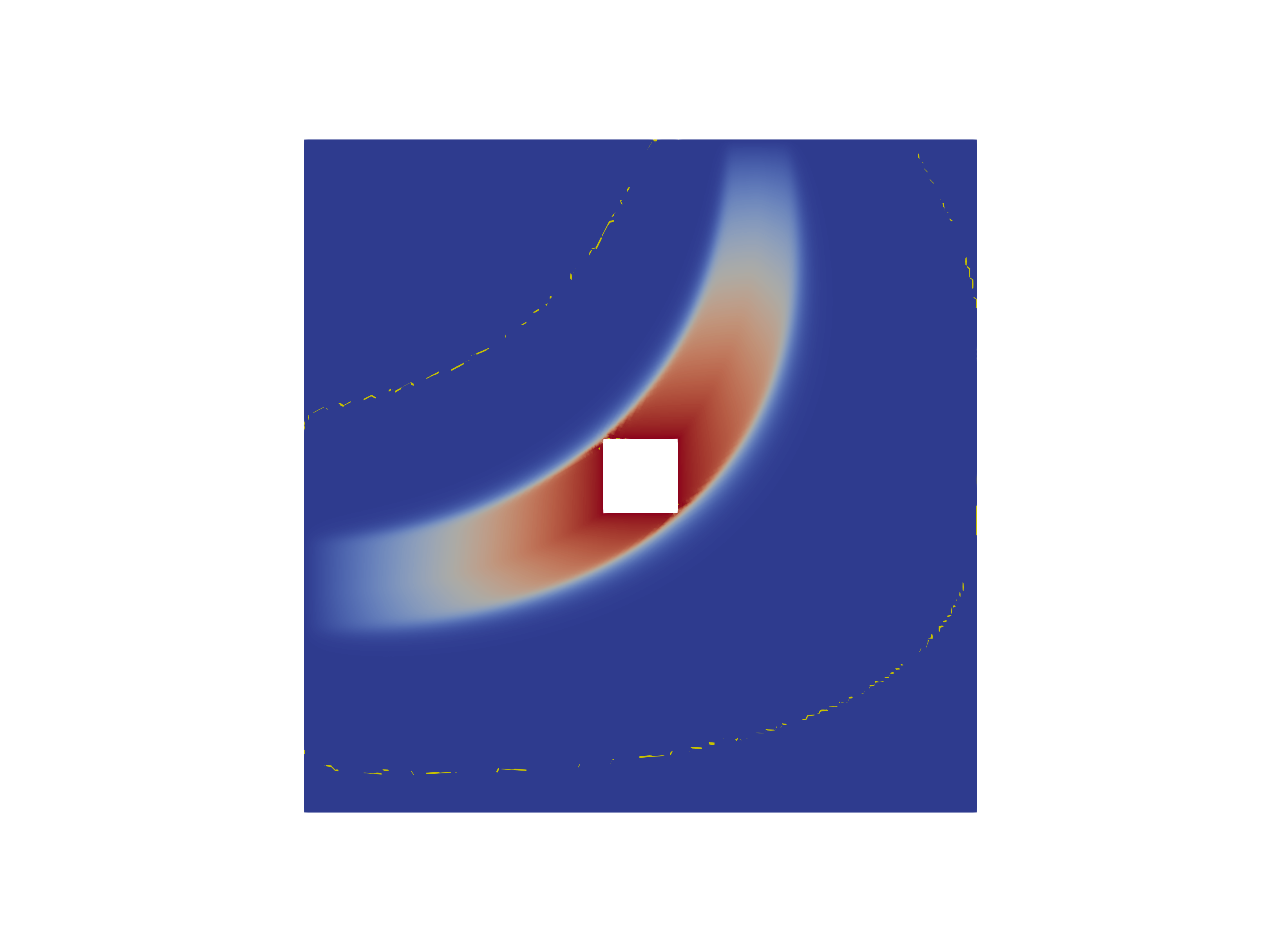}
  \put(-4.5,40){$h \approx 0.007$} 
\put(-4.5,30){$p=2$}
 \end{overpic}
 \end{minipage}
  \begin{minipage}{0.3\textwidth}\centering
     \includegraphics[scale=0.05]{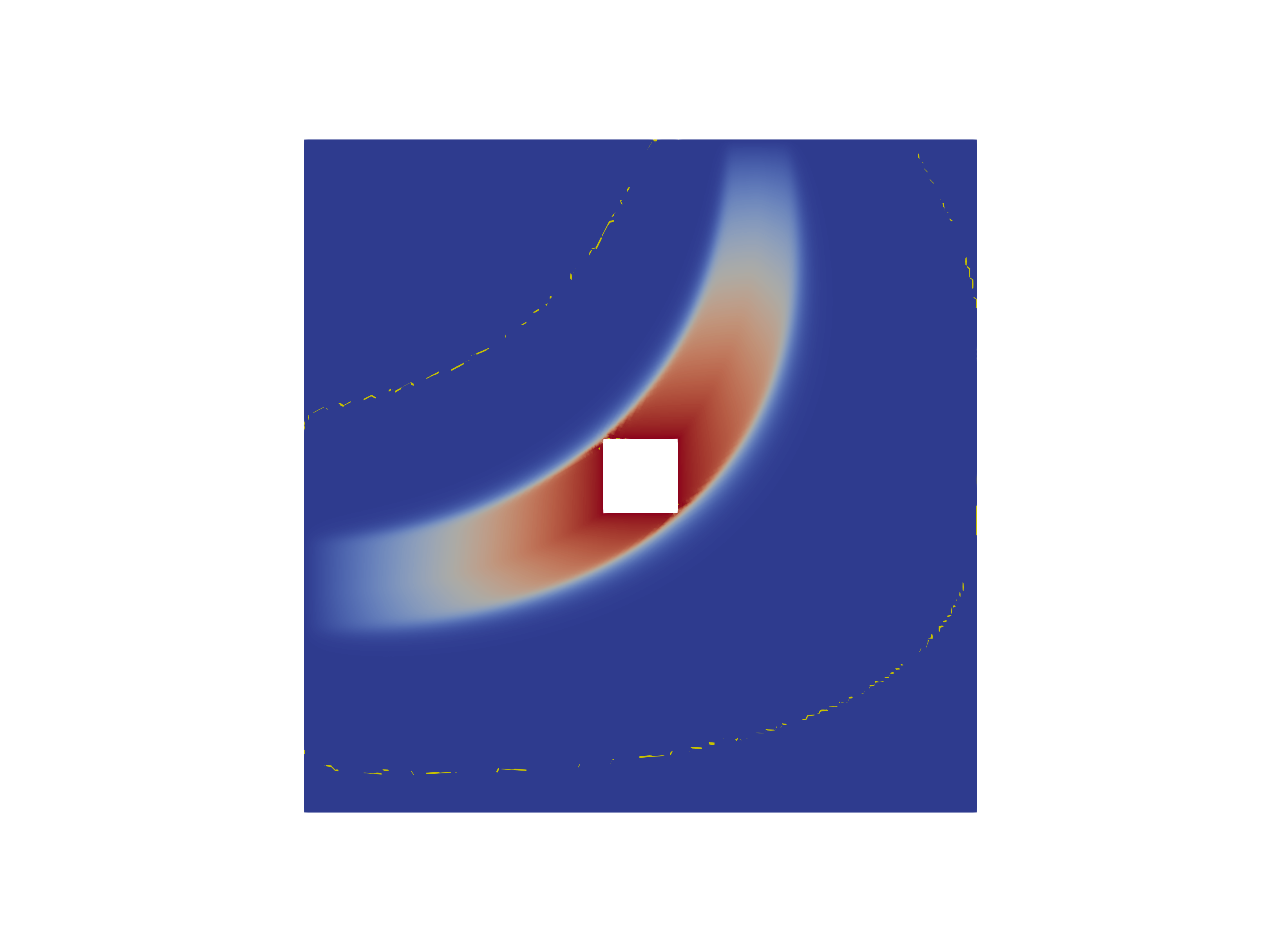}
 \end{minipage}
  \begin{minipage}{0.3\textwidth}\centering
     \includegraphics[scale=0.05]{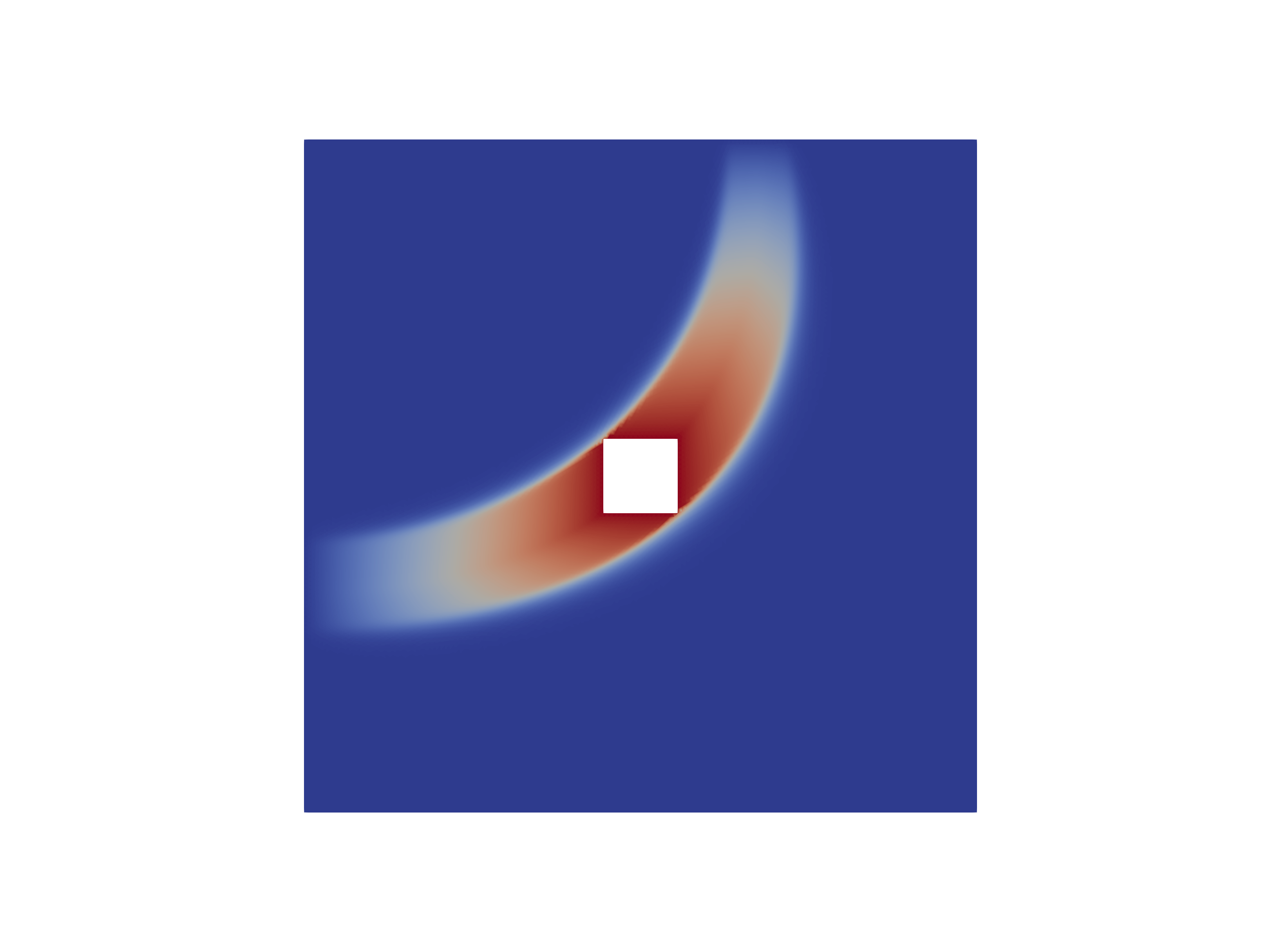}
 \end{minipage}
 \end{center}
\caption{ Comparison between the standard hybrid mixed solution $u_h^{\mathrm{mixed}}$ and the FOSPG solutions $u_h$ and $\mathcal{U}(\psi_h)$ of \Cref{alg:main_alg_disc} with $\mathcal{U}$ given in \eqref{eq:Z_2}.
We use \eqref{eq:extra_stablization} and  set $\epsilon_1 = \epsilon_2 =0.0$ for $p=0$ and $\epsilon_1 = \epsilon_2 = 0.1$ for $p=2$. The violation of DMP is indicated with yellow. 
} 
\label{fig:overshoots}
\end{figure}

\begin{table}[t]
    \centering
    \footnotesize{
    \begin{tabular}{|c|c c c c|c c c c|}
    \hline 
  & \multicolumn{4}{c|}{$h = 0.03$} & \multicolumn{4}{c|}{$h=0.007$} \\ 
  \hline  \hline
  $p$ & $\max u_h^{\mathrm{mixed}}$ & $\min u_h^{\mathrm{mixed}}$ & $\max \mathcal{U}(\psi_h)$  & $\min \mathcal{U}(\psi_h)$ & $\max u_h^{\mathrm{mixed}}$ &$\min u_h^{\mathrm{mixed}}$ & $\max \mathcal{U}(\psi_h)$  & $ \min \mathcal{U}(\psi_h)$  \\  
  \hline 
  $0$ & $1.014$ & $-0.620$ & $0.996$& $2.78e-06$ &$1.004$ &$-0.191$&$0.997$&$1.85e-06 $ \\  \hline
$1$&  $1.144$ & $-0.161$ & $1.000$ & $6.91e-08$  & $1.085$ & $-1.95e-03$ & $1.000$ & $4.89e-06$  \\  
\hline
$2$&  $1.157$ & $-0.012$ & $0.994$ & $3.10e-07$ &$1.081$&$-8.48e-11$ & $0.999$ & $3.29e-06$  \\
 \hline 
\end{tabular}
}
    \caption{ Minimum and maximum values (at quadrature points) of the standard mixed method solution $u_h^{\mathrm{mixed}}$ and the solution $\mathcal{U}(\psi_h)$ obtained by \Cref{alg:main_alg_disc} (hybridized FOSPG) on \Cref{example:punctured}.
    Both methods return a numerical flux that is locally mass conserving up to almost double precision-accuracy $\mathcal{O}(10^{-12})$.
    The solution $\mathcal{U}(\psi_h)$, with $\mathcal{U}$ given by \eqref{eq:Z_2}, is bound preserving everywhere in the domain by construction. 
    However, the mixed method solution $u_h^{\mathrm{mixed}}$ violates the DMP.
    Here, we use the stabilization term \eqref{eq:extra_stablization} with $\epsilon_1 = \epsilon_2 = 0.1$ for each polynomial degree $p$. 
    }
    \label{table:min_max}
\end{table}

\begin{remark}[Local mass conservation for anisotropic diffusion]
\label{rem:mass_conservation}
\Cref{cor:mass_conservation} guarantees mass-conservation at the elements where $u_h^*$ is strictly within the bounds. Nonetheless, we numerically observe local mass conservation everywhere. \Cref{example:punctured}  presented further challenges and required $\epsilon_1, \epsilon_2 \neq 0$ since the magnitude of the converged latent variable $\psi_h^k$ is of the order $\max |\psi_h^k| \approx 10^{10}$ otherwise. 
Having such a large solution variable can introduce round-off errors that dominate the discretization error in the computed solution.
Including the term \eqref{eq:extra_stablization} with $\epsilon_1=\epsilon_2 = 0.1$ limited the magnitude of $\psi_h^k$, delivering local mass conservation up to almost double-precision accuracy; i.e, $\max_{T \in \mathcal{T}_h} |(\nabla \cdot \bm q_h - f, 1)_T | = \mathcal{O}(10^{-12})$.
\end{remark}

\begin{remark}[Post processing $u_h$]\label{remark:limiter}
Recall that the local average of the solution $u_h$ is bound preserving by \Cref{remark:bd_preserv_avg}. Therefore, 
we can apply the classical linear scaling limiter \cite{zhang2010maximum} to obtain a bound-preserving polynomial approximation over the whole domain. 
On each element $T\in \mathcal{T}_h$, let $\overline{u}_h$ be the cell average of $u_h$, 
$M = \max_{x\in T} u_h(x)$ and $m = \min_{x\in T} u_h(x)$.
Define $\tilde u_h$ as
$$\tilde u_h(x) = \overline{u}_h+\theta (u_h(x)-\overline{u}_h), \quad \theta = \min\left\{
\left|\frac{\overline{u}-\overline{u}_h}{M-\overline{u}_h}\right|,
\left|\frac{\underline{u}-\overline{u}_h}{m-\overline{u}_h}\right|, 1
\right\}, \quad x\in T.$$ 
Then $\underline{u}\le \tilde u_h(x)\le \overline{u}$ for all $x\in T$.
\Cref{fig:limiter} shows the solutions $u_h$ and the limited solution $\tilde{u}_h$.
It is clear that the limited solution is bound-preserving.
\end{remark}
\begin{figure}[t]
   \begin{center}
  \hspace{-6em}
  \begin{minipage}{0.45\textwidth}
  \centering 
    \begin{overpic}[scale=0.075]{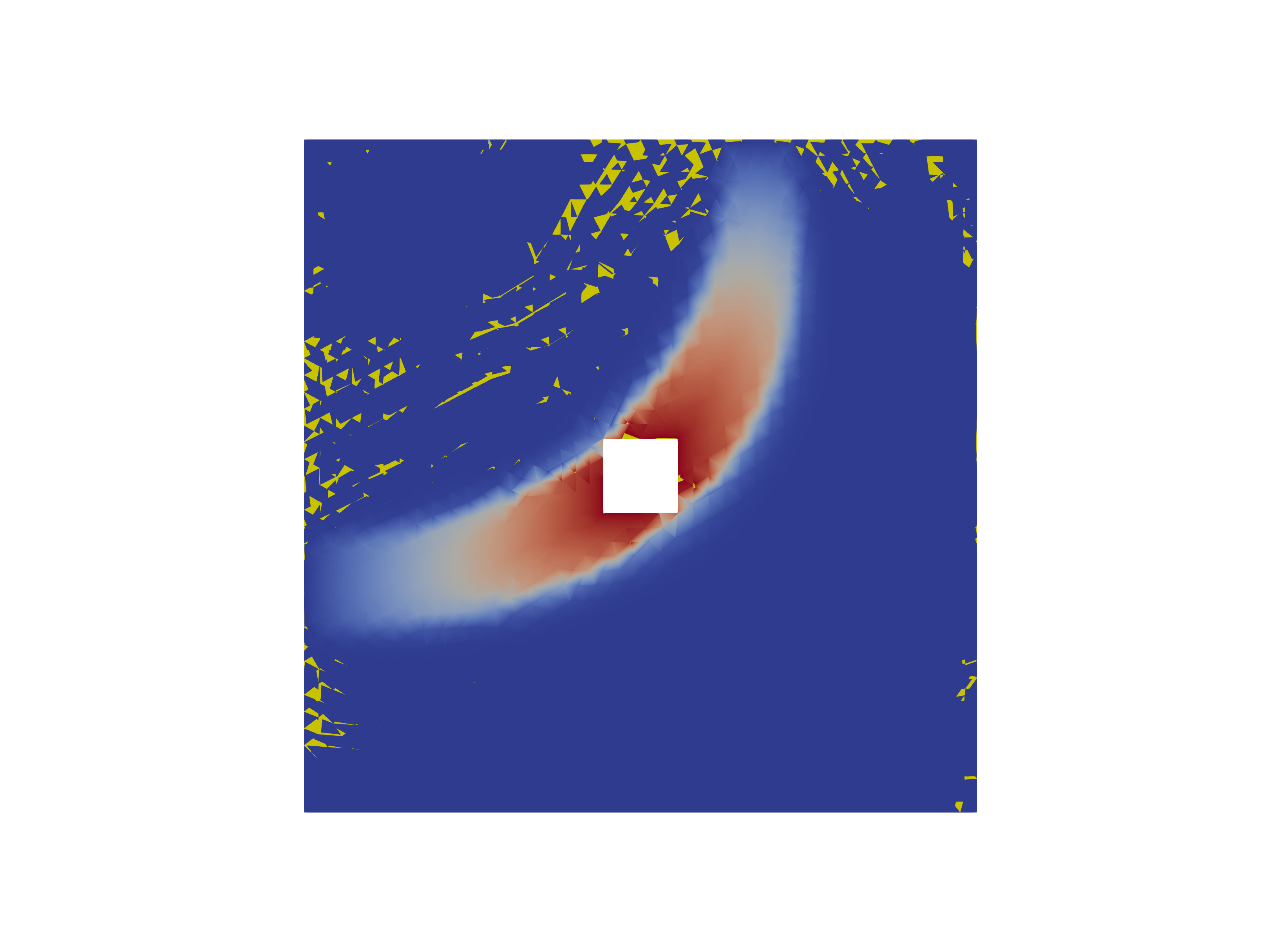}
    \put(48,67){$u_h$}
    \end{overpic}
 \end{minipage}
   \begin{minipage}{0.45\textwidth}
  \centering 
    \begin{overpic}[scale=0.075]{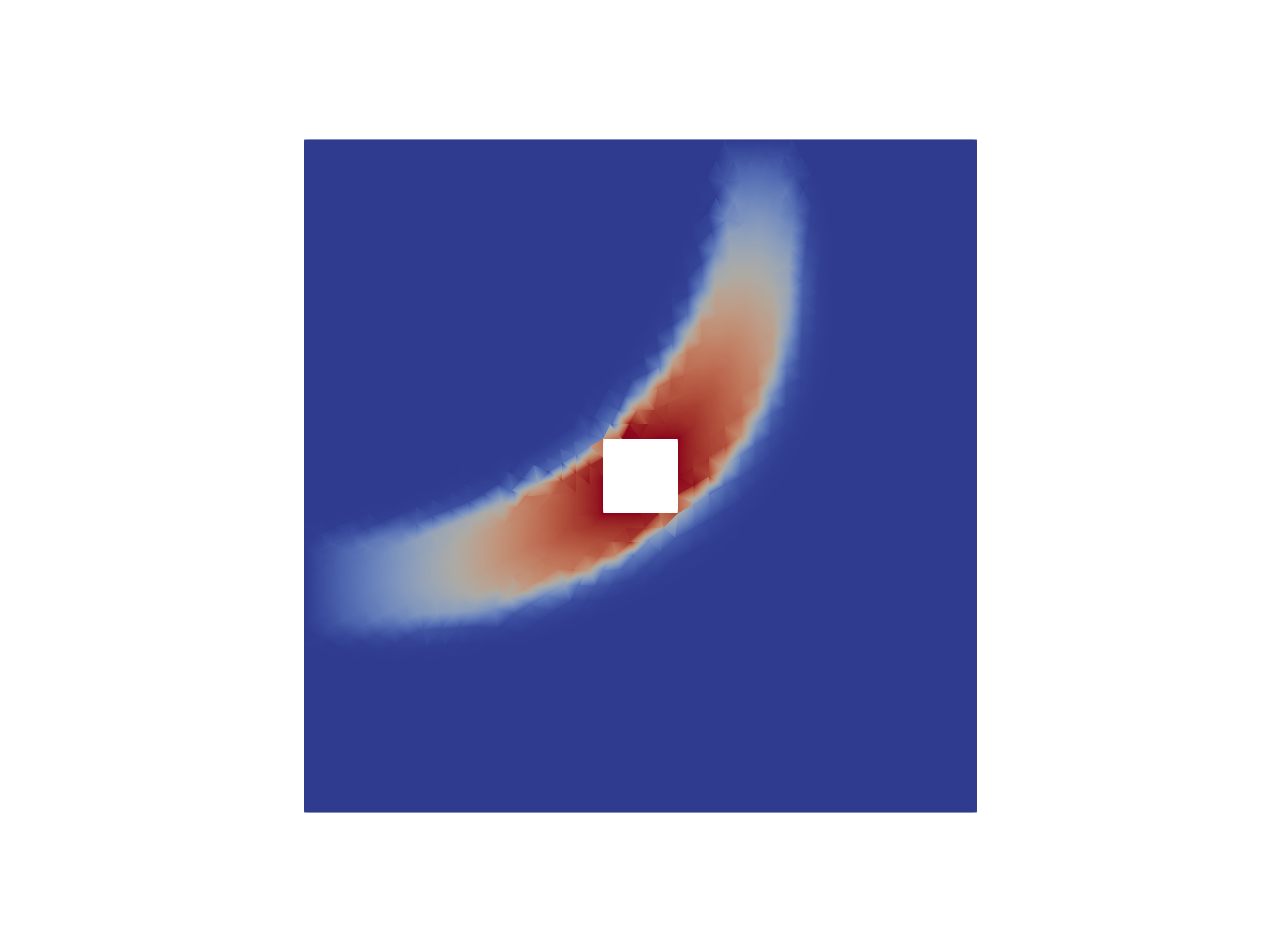}
    \put(48,67){$\tilde{u}_h$}
    \end{overpic}
 \end{minipage}
 \end{center}
\caption{(Post processing the solution $u_h$ to \Cref{example:punctured}). We set $p =2$ and $h\approx 0.03$ , $ \epsilon_1 = 0.0$ and $\epsilon_2 = 0.01$ . Left: Computed solution $u_h$ that has bound preserving local averages. Right: Limited solution $\tilde u_h$ that is bound preserving. } 
\label{fig:limiter}
\end{figure}

\subsection{Obstacle problems}  We consider two examples taken from \cite{keith2023proximal}. The first one is a smooth biactive solution,  see \cite[Subsection 4.8.1]{keith2023proximal} for a discussion on \textit{biactivity} and the challenges this property presents to active set methods. The second example is a nonsmooth spherical obstacle problem.
\begin{example}[Biactive solution]\label{example:biactive}
\normalfont 
In \eqref{eq:VI}, we set $\Omega = [-1,1]\times [-1,1]$, $\underline u = 0$,  and $\overline{u} = \infty$. We consider the smooth manufactured solution $u$ with source term $f$ given by:  
\begin{equation} \nonumber
u(x,y) = \begin{cases}
    0 & \mathrm{if} \quad x <0, \\ 
    x^4 & \mathrm{otherwise},\end{cases}  \quad f(x,y) = \begin{cases}
    0 & \mathrm{if} \quad x <0, \\ 
    -12x^2 & \mathrm{otherwise}.  
\end{cases}
\end{equation}
We set $g = u \vert_{\partial \Omega}$. \Cref{fig:rates_ex1} demonstrates that the approximation resulting from \Cref{alg:main_alg_disc} yields the expected error rates for each $p \in \{ 0,1,2,3\}$ for this example.  
\begin{figure}[t]
\begin{center}
\hspace{-3em}
 \begin{minipage}{0.31\textwidth}
         \resizebox{\textwidth}{!}{\pgfplotsset{width=7cm,compat=1.8}
\definecolor{color0}{rgb}{0.7843, 0.7843, 0.7843}
\definecolor{color1}{rgb}{0, 0.4470, 0.7410}
\definecolor{color2}{rgb}{0.8500, 0.3250, 0.0980}
\definecolor{color3}{rgb}{0.9290, 0.6940, 0.1250}
\definecolor{color4}{rgb}{0.7060, 0.3840, 0.7650}
\definecolor{color5}{rgb}{0.4660, 0.6740, 0.1880}
\definecolor{color6}{rgb}{0.3010, 0.7450, 0.9330}
\definecolor{color7}{rgb}{0.6350, 0.0780, 0.1840}
\definecolor{color8}{rgb}{0.0, 0.4078, 0.3412}

\newcommand{\logLogSlopeTriangle}[6]
{
							
	\pgfplotsextra
	{
		\pgfkeysgetvalue{/pgfplots/xmin}{\xmin}
		\pgfkeysgetvalue{/pgfplots/xmax}{\xmax}
		\pgfkeysgetvalue{/pgfplots/ymin}{\ymin}
		\pgfkeysgetvalue{/pgfplots/ymax}{\ymax}
		
				\pgfmathsetmacro{\xArel}{#1-#2}
		\pgfmathsetmacro{\yArel}{#3}
		\pgfmathsetmacro{\xBrel}{#1}
		\pgfmathsetmacro{\yBrel}{\yArel}
		\pgfmathsetmacro{\xCrel}{\xArel}
				
		\pgfmathsetmacro{\lnxB}{\xmin*(1-(#1-#2))+\xmax*(#1-#2)} 		\pgfmathsetmacro{\lnxA}{\xmin*(1-#1)+\xmax*#1} 		\pgfmathsetmacro{\lnyA}{\ymin*(1-#3)+\ymax*#3} 		\pgfmathsetmacro{\lnyC}{\lnyA+#5/#4*(\lnxA-\lnxB)}
		\pgfmathsetmacro{\yCrel}{\lnyC-\ymin)/(\ymax-\ymin)} 		
				\coordinate (A) at (rel axis cs:\xArel,\yArel);
		\coordinate (B) at (rel axis cs:\xBrel,\yBrel);
		\coordinate (C) at (rel axis cs:\xCrel,\yCrel);
		
				\draw[gray!65!black, line width=0.8pt, #6]   (A)-- node[pos=0.5,anchor=south] {\small #4}
		(B)-- 
		(C)-- node[pos=0.5,anchor=east] {\small #5}
		cycle;
	}
}

\begin{tikzpicture}[
  font=\large
  ]
	\begin{loglogaxis}[
			  xscale=1.2,
		xlabel={Mesh size $h$},
		ylabel={$\|u-u_h\|_{L^2(\Omega)}$},
  				xmajorgrids,
		ymajorgrids,
		clip mode=individual,
		legend style={fill=white, fill opacity=0.6, draw opacity=1, text opacity=1, draw = none, at={(1.06,0.45)}},
		mark repeat=1,
			ytick = {1e-8, 1e-6, 1e-4,1e-2},
            ymin = 1e-9,
            ymax = 1,
	xtick = {0.33607, 0.16803, 0.08401, 0.0420},
		 xticklabels = {$h$,$h/2$,$h/4$,$h/8$}
		]
  	\addplot[color7, mark=diamond*, very thick] table [x index={0}, y index={1}, col sep=comma] {experiment1/primal_error_0.csv};
		\addlegendentry{$p=0$};
	
		\addplot[color1, mark=diamond*, very thick] table [x index={0}, y index={1}, col sep=comma] {experiment1/primal_error_1.csv};
		\addlegendentry{$p=1$};
     
		\addplot[color2, mark=diamond*, very thick] table [x index={0}, y index={1}, col sep=comma] {experiment1/primal_error_2.csv};
		\addlegendentry{$p=2$};

		\addplot[color4, mark=diamond*, very thick] table [x index={0}, y index={1}, col sep=comma] {experiment1/primal_error_3.csv};
		\addlegendentry{$p=3$};

		\logLogSlopeTriangle{0.215}{0.1}{0.65}{1}{-2.0}{fill=color1};
		\logLogSlopeTriangle{0.215}{0.1}{0.47}{1}{-3.0}{fill=color2};
		\logLogSlopeTriangle{0.215}{0.1}{0.85}{1}{-1.0}{fill=color7};
		\logLogSlopeTriangle{0.215}{0.1}{0.15}{1}{-4.0}{fill=color4};

			\end{loglogaxis}
\end{tikzpicture}
}
 \end{minipage}
  \begin{minipage}{0.31\textwidth}
         \resizebox{\textwidth}{!}{\pgfplotsset{width=7cm,compat=1.8}
\definecolor{color0}{rgb}{0.7843, 0.7843, 0.7843}
\definecolor{color1}{rgb}{0, 0.4470, 0.7410}
\definecolor{color2}{rgb}{0.8500, 0.3250, 0.0980}
\definecolor{color3}{rgb}{0.9290, 0.6940, 0.1250}
\definecolor{color4}{rgb}{0.7060, 0.3840, 0.7650}
\definecolor{color5}{rgb}{0.4660, 0.6740, 0.1880}
\definecolor{color6}{rgb}{0.3010, 0.7450, 0.9330}
\definecolor{color7}{rgb}{0.6350, 0.0780, 0.1840}
\definecolor{color8}{rgb}{0.0, 0.4078, 0.3412}

\newcommand{\logLogSlopeTriangle}[6]
{
							
	\pgfplotsextra
	{
		\pgfkeysgetvalue{/pgfplots/xmin}{\xmin}
		\pgfkeysgetvalue{/pgfplots/xmax}{\xmax}
		\pgfkeysgetvalue{/pgfplots/ymin}{\ymin}
		\pgfkeysgetvalue{/pgfplots/ymax}{\ymax}
		
				\pgfmathsetmacro{\xArel}{#1-#2}
		\pgfmathsetmacro{\yArel}{#3}
		\pgfmathsetmacro{\xBrel}{#1}
		\pgfmathsetmacro{\yBrel}{\yArel}
		\pgfmathsetmacro{\xCrel}{\xArel}
				
		\pgfmathsetmacro{\lnxB}{\xmin*(1-(#1-#2))+\xmax*(#1-#2)} 		\pgfmathsetmacro{\lnxA}{\xmin*(1-#1)+\xmax*#1} 		\pgfmathsetmacro{\lnyA}{\ymin*(1-#3)+\ymax*#3} 		\pgfmathsetmacro{\lnyC}{\lnyA+#5/#4*(\lnxA-\lnxB)}
		\pgfmathsetmacro{\yCrel}{\lnyC-\ymin)/(\ymax-\ymin)} 		
				\coordinate (A) at (rel axis cs:\xArel,\yArel);
		\coordinate (B) at (rel axis cs:\xBrel,\yBrel);
		\coordinate (C) at (rel axis cs:\xCrel,\yCrel);
		
				\draw[gray!65!black, line width=0.8pt, #6]   (A)-- node[pos=0.5,anchor=south] {\small #4}
		(B)-- 
		(C)-- node[pos=0.5,anchor=east] {\small #5}
		cycle;
	}
}

\begin{tikzpicture}[
  font=\large
  ]
	\begin{loglogaxis}[
			  xscale=1.2,
		xlabel={Mesh size $h$},
		ylabel={$\|u-\mathcal{U}(\psi_h)\|_{L^2(\Omega)}$},
  				xmajorgrids,
		ymajorgrids,
		clip mode=individual,
		legend style={fill=white, fill opacity=0.6, draw opacity=1, text opacity=1, draw = none, at={(1.2,0.45)}},
		mark repeat=1,
			ytick = {1e-8,1e-6, 1e-4,1e-2},
            ymin = 1e-9,
            ymax = 1,
	xtick = {0.33607, 0.16803, 0.08401, 0.0420},
		 xticklabels = {$h$,$h/2$,$h/4$,$h/8$}
		]
  
		\addplot[color7, mark=diamond*, very thick] table [x index={0}, y index={1}, col sep=comma] {experiment1/latent_error_0.csv};
		\addlegendentry{$p=0$};
	
		\addplot[color1, mark=diamond*, very thick] table [x index={0}, y index={1}, col sep=comma] {experiment1/latent_error_1.csv};
		\addlegendentry{$p=1$};
     
		\addplot[color2, mark=diamond*, very thick] table [x index={0}, y index={1}, col sep=comma] {experiment1/latent_error_2.csv};
		\addlegendentry{$p=2$};

		\addplot[color4, mark=diamond*, very thick] table [x index={0}, y index={1}, col sep=comma] {experiment1/latent_error_3.csv};
		\addlegendentry{$p=3$};

	      \logLogSlopeTriangle{0.215}{0.1}{0.65}{1}{-2.0}{fill=color1};
		\logLogSlopeTriangle{0.215}{0.1}{0.47}{1}{-3.0}{fill=color2};
		\logLogSlopeTriangle{0.215}{0.1}{0.85}{1}{-1.0}{fill=color7};
		\logLogSlopeTriangle{0.215}{0.1}{0.15}{1}{-4.0}{fill=color4};

		\legend{}; 	\end{loglogaxis}
\end{tikzpicture}
}
 \end{minipage}
 \begin{minipage}{0.31\textwidth}
         \resizebox{\textwidth}{!}{
\pgfplotsset{width=7cm,compat=1.8}
\definecolor{color0}{rgb}{0.7843, 0.7843, 0.7843}
\definecolor{color1}{rgb}{0, 0.4470, 0.7410}
\definecolor{color2}{rgb}{0.8500, 0.3250, 0.0980}
\definecolor{color3}{rgb}{0.9290, 0.6940, 0.1250}
\definecolor{color4}{rgb}{0.7060, 0.3840, 0.7650}
\definecolor{color5}{rgb}{0.4660, 0.6740, 0.1880}
\definecolor{color6}{rgb}{0.3010, 0.7450, 0.9330}
\definecolor{color7}{rgb}{0.6350, 0.0780, 0.1840}
\definecolor{color8}{rgb}{0.0, 0.4078, 0.3412}

\newcommand{\logLogSlopeTriangle}[6]
{
							
	\pgfplotsextra
	{
		\pgfkeysgetvalue{/pgfplots/xmin}{\xmin}
		\pgfkeysgetvalue{/pgfplots/xmax}{\xmax}
		\pgfkeysgetvalue{/pgfplots/ymin}{\ymin}
		\pgfkeysgetvalue{/pgfplots/ymax}{\ymax}
		
				\pgfmathsetmacro{\xArel}{#1-#2}
		\pgfmathsetmacro{\yArel}{#3}
		\pgfmathsetmacro{\xBrel}{#1}
		\pgfmathsetmacro{\yBrel}{\yArel}
		\pgfmathsetmacro{\xCrel}{\xArel}
				
		\pgfmathsetmacro{\lnxB}{\xmin*(1-(#1-#2))+\xmax*(#1-#2)} 		\pgfmathsetmacro{\lnxA}{\xmin*(1-#1)+\xmax*#1} 		\pgfmathsetmacro{\lnyA}{\ymin*(1-#3)+\ymax*#3} 		\pgfmathsetmacro{\lnyC}{\lnyA+#5/#4*(\lnxA-\lnxB)}
		\pgfmathsetmacro{\yCrel}{\lnyC-\ymin)/(\ymax-\ymin)} 		
				\coordinate (A) at (rel axis cs:\xArel,\yArel);
		\coordinate (B) at (rel axis cs:\xBrel,\yBrel);
		\coordinate (C) at (rel axis cs:\xCrel,\yCrel);
		
				\draw[gray!65!black, line width=0.8pt, #6]   (A)-- node[pos=0.5,anchor=south] {\small #4}
		(B)-- 
		(C)-- node[pos=0.5,anchor=east] {\small #5}
		cycle;
	}
}

\begin{tikzpicture}[
  font=\large
  ]
	\begin{loglogaxis}[
			  xscale=1.2,
		xlabel={Mesh size $h$},
		ylabel={$\|\bm q-\bm q_h\|_{L^2(\Omega)}$},
  				xmajorgrids,
		ymajorgrids,
		clip mode=individual,
		legend style={fill=white, fill opacity=0.6, draw opacity=1, text opacity=1, draw = none},
		legend pos=south east,
		mark repeat=1,
            ymin = 1e-9,
            ymax = 1,   
		ytick = {1e-8, 1e-6, 1e-4, 1e-2}, 
			xtick = {0.33607, 0.16803, 0.08401, 0.0420},
		 xticklabels = {$h$,$h/2$,$h/4$,$h/8$}
		]
  
		\addplot[color7, mark=diamond*, very thick] table [x index={0}, y index={1}, col sep=comma] {experiment1/flux_error_0.csv};
		\addlegendentry{$p=0$};
	
		\addplot[color1, mark=diamond*, very thick] table [x index={0}, y index={1}, col sep=comma] {experiment1/flux_error_1.csv};
		\addlegendentry{$p=1$};
     
		\addplot[color2, mark=diamond*, very thick] table [x index={0}, y index={1}, col sep=comma] {experiment1/flux_error_2.csv};
		\addlegendentry{$p=2$};

		\addplot[color4, mark=diamond*, very thick] table [x index={0}, y index={1}, col sep=comma] {experiment1/flux_error_3.csv};
		\addlegendentry{$p=3$};

\logLogSlopeTriangle{0.215}{0.1}{0.72}{1}{-2.0}{fill=color1};
		\logLogSlopeTriangle{0.215}{0.1}{0.47}{1}{-3.0}{fill=color2};
		\logLogSlopeTriangle{0.215}{0.1}{0.93}{1}{-1.0}{fill=color7};
		\logLogSlopeTriangle{0.215}{0.1}{0.27}{1}{-4.0}{fill=color4};

		 \legend{}; 	\end{loglogaxis}
\end{tikzpicture}
}
 \end{minipage}
 \end{center}
 \caption{ (\Cref{example:biactive}). Computed $L^2$ errors and rates of the approximations $(u_h, \mathcal{U}(\psi_h), \bm q_h)$ of \Cref{alg:main_alg_disc} with  $\mathcal{U}(\psi_h) = \exp(\psi_h)$.   The algorithm is terminated when $\|u_h^{k} -u_h^{k-1}\|_{L^2(\Omega)} < 10^{-12}$ and the tolerance for the Newton solver is set to $10^{-10}$.  The coarsest mesh size $h\approx 0.33$. We set $\epsilon_1 = 0$ in \eqref{eq:extra_stablization} for all $p$.  For $p=0$ and $p=1$, we set $\epsilon_2 = 0$. For $p=2$, we set $\epsilon_2 = 1e-05$ and for $p=3$, we set $\epsilon_2 = 1e-07$. In this example, we set $\alpha^k = 1.5^k$. }
\label{fig:rates_ex1}
 \end{figure}
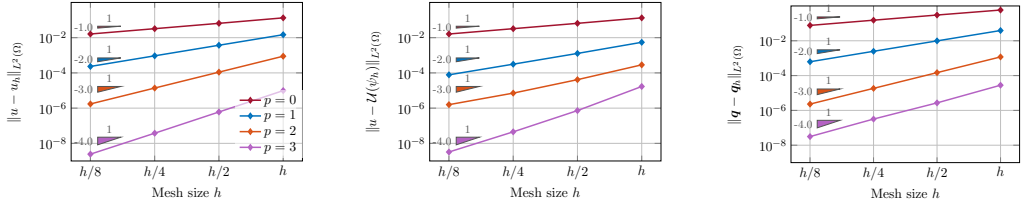

\end{example}
\begin{example}[Spherical obstacle] \normalfont \label{example:spherical}
We consider the example from \cite[Section 4.8.4]{keith2023proximal}. Accordingly, we set $\Omega$ to be the circle centered at $(0,0)$ of radius 1, $\overline{u} = \infty$, $A$ to be the identity, $g = f = 0$, and 
\[ 
    \underline u = \begin{cases}
        \sqrt{ 1/4 - r^2} & \mathrm{if } \quad r \leq 9/20, \\ 
                \varphi (r) & \mathrm{otherwise},
    \end{cases}
    \qquad
    \text{where }
    r = \sqrt{x^2 + y^2}
    \,.
\]
In the above, $\varphi(r)$ is the unique $C^1$ linear extension of $r \mapsto \sqrt{1/4 - r^2}$ for $r > 9/20$. The exact solution,
\[
    u = \begin{cases}
    Q  \ln \sqrt{x^2+y^2} &\mathrm{if } \quad \sqrt{x^2 +y^2} >  a, \\ 
    \underline u  & \mathrm{otherwise},
    \end{cases}
\]
where $a= \exp(W_{-1}(-1/(2e^2))/2 + 1)\approx 0.34898$ where $W_{-1}(\cdot)$ is the $-1$-branch of the Lambert W-function, and $Q = \sqrt{1/4-a^2}/\ln a$, belongs to $H^{5/2-\epsilon}(\Omega)$.
Hence, one can only expect at most an order of $5/2$ (resp.\ $3/2$) for $\|u-u_h\|_{L^2(\Omega)}$ (resp.\ $\|\bm q - \bm q_h\|_{L^2(\Omega)} $). 

\Cref{fig:spherical_obs} shows the solution computed with $p=2$ along with 
a local mass conservation indicator $\xi_h\in V_h^0$ defined as $\xi_h|_{T} = |(\nabla \cdot \bm q_h , 1)_T|$ for all $T\in \mathcal{T}_h$.
This demonstrates the result stated in \Cref{cor:mass_conservation}; i.e., that the computed solution is locally mass conservative away from the contact zone. 
\Cref{fig:rates_ex2} reports the rates of decay of the discretization error. For $p=2$, the optimal rate of $1.5$ is observed for the flux. For error in the primal variable, we observe a rate of $2.0$ for $p= 1$ and $p=2$. 
\begin{figure}
  \begin{center}
  \hspace{-5cm}
 \begin{minipage}{0.45\textwidth}
 \centering
\includegraphics[scale=0.15]{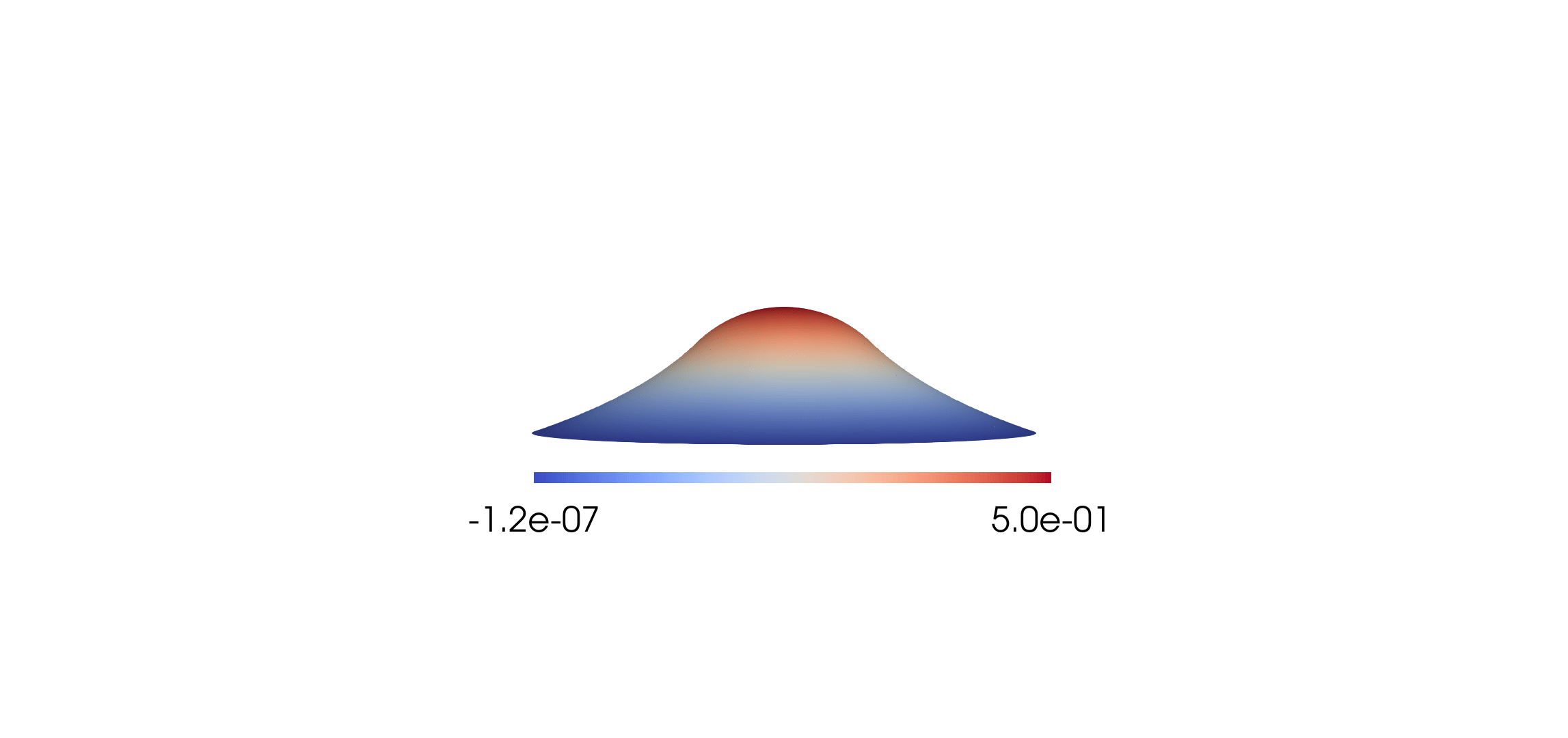}
 \end{minipage}
  \begin{minipage}{0.45\textwidth} 
  \centering
\includegraphics[scale=0.15]{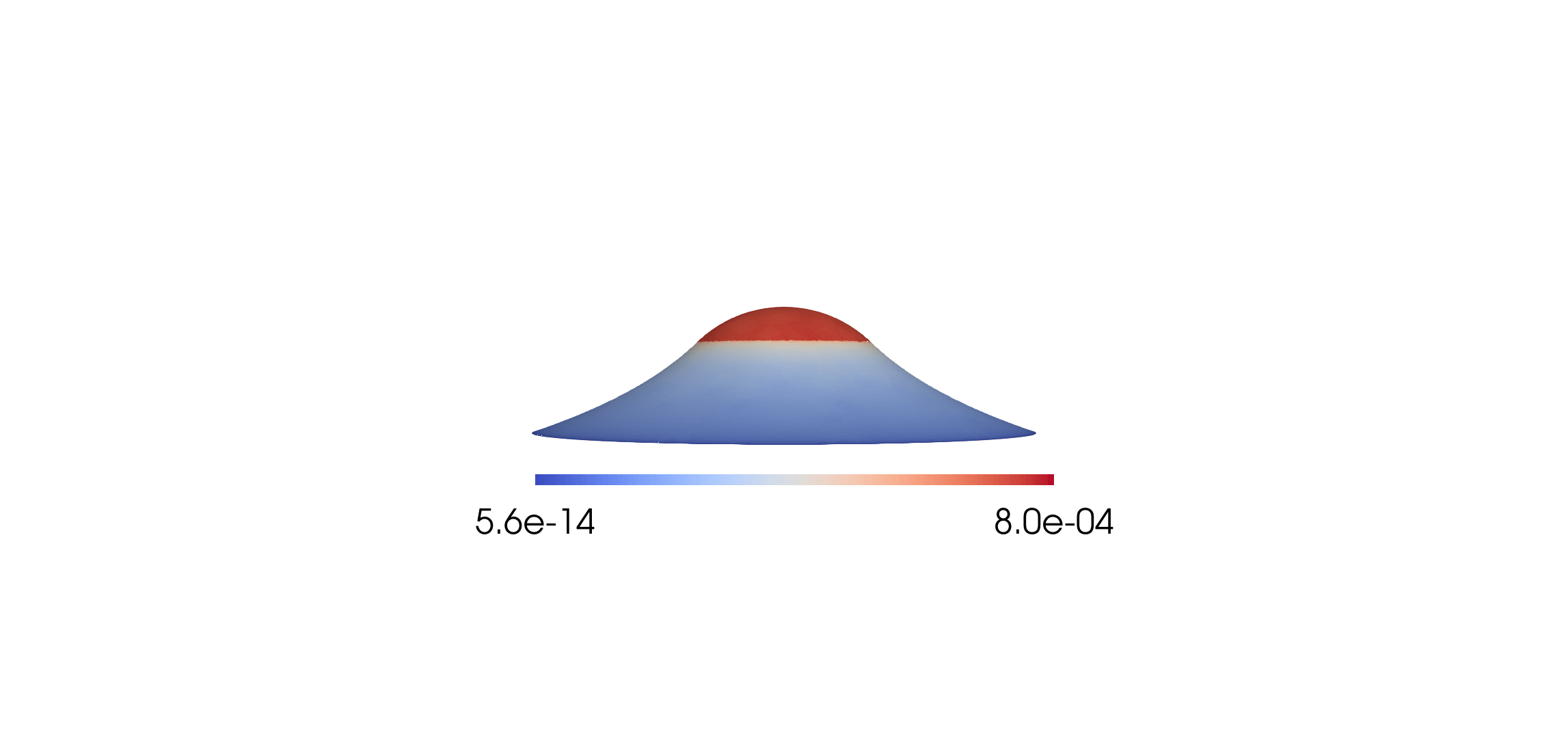}
 \end{minipage}
 \end{center}
\caption{Left: Computed solution $u_h$ for $p =2$ and $h\approx 0.015$ with $\epsilon_1 = 0$ and $\epsilon_2 = 1e-04$ and $\mathcal{U}(\psi) = \underline u + \exp(\psi)$. Right: Local mass conservation indicator 
$\xi_h$, plotted with log scale.
Observe that we obtain local mass conservation with almost double precision accuracy on elements not intersecting the obstacle. } 
\label{fig:spherical_obs}
\end{figure}
\begin{figure}[htb]
\begin{center}
\hspace{-3em}
 \begin{minipage}{0.32\textwidth}
         \resizebox{\textwidth}{!}{\pgfplotsset{width=7cm,compat=1.8}
\definecolor{color0}{rgb}{0.7843, 0.7843, 0.7843}
\definecolor{color1}{rgb}{0, 0.4470, 0.7410}
\definecolor{color2}{rgb}{0.8500, 0.3250, 0.0980}
\definecolor{color3}{rgb}{0.9290, 0.6940, 0.1250}
\definecolor{color4}{rgb}{0.7060, 0.3840, 0.7650}
\definecolor{color5}{rgb}{0.4660, 0.6740, 0.1880}
\definecolor{color6}{rgb}{0.3010, 0.7450, 0.9330}
\definecolor{color7}{rgb}{0.6350, 0.0780, 0.1840}
\definecolor{color8}{rgb}{0.0, 0.4078, 0.3412}

\newcommand{\logLogSlopeTriangle}[6]
{
							
	\pgfplotsextra
	{
		\pgfkeysgetvalue{/pgfplots/xmin}{\xmin}
		\pgfkeysgetvalue{/pgfplots/xmax}{\xmax}
		\pgfkeysgetvalue{/pgfplots/ymin}{\ymin}
		\pgfkeysgetvalue{/pgfplots/ymax}{\ymax}
		
				\pgfmathsetmacro{\xArel}{#1-#2}
		\pgfmathsetmacro{\yArel}{#3}
		\pgfmathsetmacro{\xBrel}{#1}
		\pgfmathsetmacro{\yBrel}{\yArel}
		\pgfmathsetmacro{\xCrel}{\xArel}
				
		\pgfmathsetmacro{\lnxB}{\xmin*(1-(#1-#2))+\xmax*(#1-#2)} 		\pgfmathsetmacro{\lnxA}{\xmin*(1-#1)+\xmax*#1} 		\pgfmathsetmacro{\lnyA}{\ymin*(1-#3)+\ymax*#3} 		\pgfmathsetmacro{\lnyC}{\lnyA+#5/#4*(\lnxA-\lnxB)}
		\pgfmathsetmacro{\yCrel}{\lnyC-\ymin)/(\ymax-\ymin)} 		
				\coordinate (A) at (rel axis cs:\xArel,\yArel);
		\coordinate (B) at (rel axis cs:\xBrel,\yBrel);
		\coordinate (C) at (rel axis cs:\xCrel,\yCrel);
		
				\draw[gray!65!black, line width=0.8pt, #6]   (A)-- node[pos=0.5,anchor=south] {\small #4}
		(B)-- 
		(C)-- node[pos=0.5,anchor=east] {\small #5}
		cycle;
	}
}

\begin{tikzpicture}[
  font=\large
  ]
	\begin{loglogaxis}[
			  xscale=1.2,
		xlabel={Mesh size $h$},
		ylabel={$\|u-u_h\|_{L^2(\Omega)}$},
  				xmajorgrids,
		ymajorgrids,
		clip mode=individual,
		legend style={fill=white, fill opacity=0.6, draw opacity=1, text opacity=1, draw = none, at={(1.15,0.36)}},
		mark repeat=1,
			ytick = {1e-6, 1e-4,1e-2, 1},
            ymin = 1e-7,
            ymax = 1e-1,
		xtick = {0.0625,0.03125,0.015625,0.0078125},
		 xticklabels = {$h$,$h/2$,$h/4$,$h/8$}
		]
  
		\addplot[color7, mark=diamond*, very thick] table [x index={0}, y index={1}, col sep=comma] {circle/primal_error_0_0.0.csv};
		\addlegendentry{$p=0$};
	
		\addplot[color1, mark=diamond*, very thick] table [x index={0}, y index={1}, col sep=comma] {circle/primal_error_1_1e-05.csv};
		\addlegendentry{$p=1$};
        	  
		\addplot[color2, mark=diamond*, very thick] table [x index={0}, y index={1}, col sep=comma] {circle/primal_error_2_0.0002.csv};
		\addlegendentry{$p=2$};

		\logLogSlopeTriangle{0.2}{0.1}{0.35}{1}{-2.0}{fill=color1};
				\logLogSlopeTriangle{0.215}{0.1}{0.75}{1}{-1.0}{fill=color7};
		
			\end{loglogaxis}
\end{tikzpicture}
}
 \end{minipage}
  \begin{minipage}{0.32\textwidth}
         \resizebox{\textwidth}{!}{\pgfplotsset{width=7cm,compat=1.8}
\definecolor{color0}{rgb}{0.7843, 0.7843, 0.7843}
\definecolor{color1}{rgb}{0, 0.4470, 0.7410}
\definecolor{color2}{rgb}{0.8500, 0.3250, 0.0980}
\definecolor{color3}{rgb}{0.9290, 0.6940, 0.1250}
\definecolor{color4}{rgb}{0.7060, 0.3840, 0.7650}
\definecolor{color5}{rgb}{0.4660, 0.6740, 0.1880}
\definecolor{color6}{rgb}{0.3010, 0.7450, 0.9330}
\definecolor{color7}{rgb}{0.6350, 0.0780, 0.1840}
\definecolor{color8}{rgb}{0.0, 0.4078, 0.3412}

\newcommand{\logLogSlopeTriangle}[6]
{
							
	\pgfplotsextra
	{
		\pgfkeysgetvalue{/pgfplots/xmin}{\xmin}
		\pgfkeysgetvalue{/pgfplots/xmax}{\xmax}
		\pgfkeysgetvalue{/pgfplots/ymin}{\ymin}
		\pgfkeysgetvalue{/pgfplots/ymax}{\ymax}
		
				\pgfmathsetmacro{\xArel}{#1-#2}
		\pgfmathsetmacro{\yArel}{#3}
		\pgfmathsetmacro{\xBrel}{#1}
		\pgfmathsetmacro{\yBrel}{\yArel}
		\pgfmathsetmacro{\xCrel}{\xArel}
				
		\pgfmathsetmacro{\lnxB}{\xmin*(1-(#1-#2))+\xmax*(#1-#2)} 		\pgfmathsetmacro{\lnxA}{\xmin*(1-#1)+\xmax*#1} 		\pgfmathsetmacro{\lnyA}{\ymin*(1-#3)+\ymax*#3} 		\pgfmathsetmacro{\lnyC}{\lnyA+#5/#4*(\lnxA-\lnxB)}
		\pgfmathsetmacro{\yCrel}{\lnyC-\ymin)/(\ymax-\ymin)} 		
				\coordinate (A) at (rel axis cs:\xArel,\yArel);
		\coordinate (B) at (rel axis cs:\xBrel,\yBrel);
		\coordinate (C) at (rel axis cs:\xCrel,\yCrel);
		
				\draw[gray!65!black, line width=0.8pt, #6]   (A)-- node[pos=0.5,anchor=south] {\small #4}
		(B)-- 
		(C)-- node[pos=0.5,anchor=east] {\small #5}
		cycle;
	}
}

\begin{tikzpicture}[
  font=\large
  ]
	\begin{loglogaxis}[
				  xscale=1.2,
		xlabel={Mesh size $h$},
		ylabel={$\|u-\mathcal{U}(\psi_h)\|_{L^2(\Omega)}$},
  				xmajorgrids,
		ymajorgrids,
		clip mode=individual,
		legend style={fill=white, fill opacity=0.6, draw opacity=1, text opacity=1},
		legend pos=south east,
		mark repeat=1,
		ytick = {1e-6, 1e-4,1e-2},
            ymin = 1e-7,
            ymax = 1e-1,
		xtick = {0.0625,0.03125,0.015625,0.0078125},
		 xticklabels = {$h$,$h/2$,$h/4$,$h/8$}
		]
         \addplot[color7, mark=diamond*, very thick] table [x index={0}, y index={1}, col sep=comma] {circle/latent_error_0_0.0.csv};
		\addlegendentry{$p=0$};
				\addplot[color1, mark=diamond*, very thick] table [x index={0}, y index={1}, col sep=comma] {circle/latent_error_1_0.0.csv};
		\addlegendentry{$p=1$};

		\addplot[color2, mark=diamond*, very thick] table [x index={0}, y index={1}, col sep=comma] {circle/latent_error_2_0.0002.csv};
		\addlegendentry{$p=2$};

		\logLogSlopeTriangle{0.2}{0.1}{0.45}{1}{-2.0}{fill=color1};
				\logLogSlopeTriangle{0.215}{0.1}{0.85}{1}{-1.0}{fill=color7};
				
				  \legend{}; 	\end{loglogaxis}
\end{tikzpicture}
}
 \end{minipage}
 \begin{minipage}{0.32\textwidth}
         \resizebox{\textwidth}{!}{
\pgfplotsset{width=7cm,compat=1.8}
\definecolor{color0}{rgb}{0.7843, 0.7843, 0.7843}
\definecolor{color1}{rgb}{0, 0.4470, 0.7410}
\definecolor{color2}{rgb}{0.8500, 0.3250, 0.0980}
\definecolor{color3}{rgb}{0.9290, 0.6940, 0.1250}
\definecolor{color4}{rgb}{0.7060, 0.3840, 0.7650}
\definecolor{color5}{rgb}{0.4660, 0.6740, 0.1880}
\definecolor{color6}{rgb}{0.3010, 0.7450, 0.9330}
\definecolor{color7}{rgb}{0.6350, 0.0780, 0.1840}
\definecolor{color8}{rgb}{0.0, 0.4078, 0.3412}

\newcommand{\logLogSlopeTriangle}[6]
{
							
	\pgfplotsextra
	{
		\pgfkeysgetvalue{/pgfplots/xmin}{\xmin}
		\pgfkeysgetvalue{/pgfplots/xmax}{\xmax}
		\pgfkeysgetvalue{/pgfplots/ymin}{\ymin}
		\pgfkeysgetvalue{/pgfplots/ymax}{\ymax}
		
				\pgfmathsetmacro{\xArel}{#1-#2}
		\pgfmathsetmacro{\yArel}{#3}
		\pgfmathsetmacro{\xBrel}{#1}
		\pgfmathsetmacro{\yBrel}{\yArel}
		\pgfmathsetmacro{\xCrel}{\xArel}
				
		\pgfmathsetmacro{\lnxB}{\xmin*(1-(#1-#2))+\xmax*(#1-#2)} 		\pgfmathsetmacro{\lnxA}{\xmin*(1-#1)+\xmax*#1} 		\pgfmathsetmacro{\lnyA}{\ymin*(1-#3)+\ymax*#3} 		\pgfmathsetmacro{\lnyC}{\lnyA+#5/#4*(\lnxA-\lnxB)}
		\pgfmathsetmacro{\yCrel}{\lnyC-\ymin)/(\ymax-\ymin)} 		
				\coordinate (A) at (rel axis cs:\xArel,\yArel);
		\coordinate (B) at (rel axis cs:\xBrel,\yBrel);
		\coordinate (C) at (rel axis cs:\xCrel,\yCrel);
		
				\draw[gray!65!black, line width=0.8pt, #6]   (A)-- node[pos=0.5,anchor=south] {\small #4}
		(B)-- 
		(C)-- node[pos=0.5,anchor=east] {\small #5}
		cycle;
	}
}

\begin{tikzpicture}[
  font=\large
  ]
	\begin{loglogaxis}[
			  xscale=1.2,
		xlabel={Mesh size $h$},
		ylabel={$\|\bm q-\bm q_h\|_{L^2(\Omega)}$},
  				xmajorgrids,
		ymajorgrids,
		clip mode=individual,
		legend style={fill=white, fill opacity=0.6, draw opacity=1, text opacity=1, draw = none},
		legend pos=south east,
		mark repeat=1,
            ymin = 1e-4,
            ymax = 1e-1,   
		ytick = {1e-4, 1e-2, 1}, 
			xtick = {0.0625,0.03125,0.015625,0.0078125},
		 xticklabels = {$h$,$h/2$,$h/4$,$h/8$}
		]
  
            \addplot[color7, mark=diamond*, very thick] table [x index={0}, y index={1}, col sep=comma] {circle/flux_error_0_0.0.csv};
            \addlegendentry{$p=0$};

        \addplot[color1, mark=diamond*, very thick] table [x index={0}, y index={1}, col sep=comma] {circle/flux_error_1_0.0.csv};
		\addlegendentry{$p=1$} ;
  
		\addplot[color2, mark=diamond*, very thick] table [x index={0}, y index={1}, col sep=comma] {circle/flux_error_2_0.0002.csv};
		\addlegendentry{$p=2$};

		\logLogSlopeTriangle{0.215}{0.1}{0.35}{1}{-1.5}{fill=color1};
		\logLogSlopeTriangle{0.215}{0.1}{0.64}{1}{-1.0}{fill=color7};
  						
		 \legend{}; 	\end{loglogaxis}
\end{tikzpicture}
}
 \end{minipage}
 \end{center}
 \begin{center}
 \hspace{-3em}
     \begin{minipage}{0.32\textwidth}
         \resizebox{\textwidth}{!}{\pgfplotsset{width=7cm,compat=1.8}
\definecolor{color0}{rgb}{0.7843, 0.7843, 0.7843}
\definecolor{color1}{rgb}{0, 0.4470, 0.7410}
\definecolor{color2}{rgb}{0.8500, 0.3250, 0.0980}
\definecolor{color3}{rgb}{0.9290, 0.6940, 0.1250}
\definecolor{color4}{rgb}{0.7060, 0.3840, 0.7650}
\definecolor{color5}{rgb}{0.4660, 0.6740, 0.1880}
\definecolor{color6}{rgb}{0.3010, 0.7450, 0.9330}
\definecolor{color7}{rgb}{0.6350, 0.0780, 0.1840}
\definecolor{color8}{rgb}{0.0, 0.4078, 0.3412}

\newcommand{\logLogSlopeTriangle}[6]
{
							
	\pgfplotsextra
	{
		\pgfkeysgetvalue{/pgfplots/xmin}{\xmin}
		\pgfkeysgetvalue{/pgfplots/xmax}{\xmax}
		\pgfkeysgetvalue{/pgfplots/ymin}{\ymin}
		\pgfkeysgetvalue{/pgfplots/ymax}{\ymax}
		
				\pgfmathsetmacro{\xArel}{#1-#2}
		\pgfmathsetmacro{\yArel}{#3}
		\pgfmathsetmacro{\xBrel}{#1}
		\pgfmathsetmacro{\yBrel}{\yArel}
		\pgfmathsetmacro{\xCrel}{\xArel}
				
		\pgfmathsetmacro{\lnxB}{\xmin*(1-(#1-#2))+\xmax*(#1-#2)} 		\pgfmathsetmacro{\lnxA}{\xmin*(1-#1)+\xmax*#1} 		\pgfmathsetmacro{\lnyA}{\ymin*(1-#3)+\ymax*#3} 		\pgfmathsetmacro{\lnyC}{\lnyA+#5/#4*(\lnxA-\lnxB)}
		\pgfmathsetmacro{\yCrel}{\lnyC-\ymin)/(\ymax-\ymin)} 		
				\coordinate (A) at (rel axis cs:\xArel,\yArel);
		\coordinate (B) at (rel axis cs:\xBrel,\yBrel);
		\coordinate (C) at (rel axis cs:\xCrel,\yCrel);
		
				\draw[gray!65!black, line width=0.8pt, #6]   (A)-- node[pos=0.5,anchor=south] {\small #4}
		(B)-- 
		(C)-- node[pos=0.5,anchor=east] {\small #5}
		cycle;
	}
}

\begin{tikzpicture}[
  font=\large
  ]
	\begin{loglogaxis}[
			  xscale=1.2,
		xlabel={Mesh size $h$},
		ylabel={$\|u-u_h\|_{L^2(\Omega)}$},
  				xmajorgrids,
		ymajorgrids,
		clip mode=individual,
		legend style={fill=white, fill opacity=0.6, draw opacity=1, text opacity=1, draw = none, at={(1.15,0.36)}},
		mark repeat=1,
			ytick = { 1e-6,1e-4,1e-2, 1},
            ymin = 1e-7,
            ymax = 1e-1,
		xtick = {0.0625,0.03125,0.015625,0.0078125},
		 xticklabels = {$h$,$h/2$,$h/4$,$h/8$}
		]
  
		\addplot[color7, mark=diamond*, very thick] table [x index={0}, y index={1}, col sep=comma] {circle/primal_error_0_Z2_0.0.csv};
		\addlegendentry{$p=0$};
	 
				        \addplot[color1, mark=diamond*, very thick] table [x index={0}, y index={1}, col sep=comma] {circle/primal_error_1_Z2_0.0.csv};
		\addlegendentry{$p=1$};
  
		\addplot[color2, mark=diamond*, very thick] table [x index={0}, y index={1}, col sep=comma] {circle/primal_error_2_Z2_0.0002.csv};
		\addlegendentry{$p=2$};

		\logLogSlopeTriangle{0.2}{0.1}{0.35}{1}{-2.0}{fill=color1};
				\logLogSlopeTriangle{0.215}{0.1}{0.75}{1}{-1.0}{fill=color7};
		
			\end{loglogaxis}
\end{tikzpicture}
}
 \end{minipage}
  \begin{minipage}{0.32\textwidth}
         \resizebox{\textwidth}{!}{\pgfplotsset{width=7cm,compat=1.8}
\definecolor{color0}{rgb}{0.7843, 0.7843, 0.7843}
\definecolor{color1}{rgb}{0, 0.4470, 0.7410}
\definecolor{color2}{rgb}{0.8500, 0.3250, 0.0980}
\definecolor{color3}{rgb}{0.9290, 0.6940, 0.1250}
\definecolor{color4}{rgb}{0.7060, 0.3840, 0.7650}
\definecolor{color5}{rgb}{0.4660, 0.6740, 0.1880}
\definecolor{color6}{rgb}{0.3010, 0.7450, 0.9330}
\definecolor{color7}{rgb}{0.6350, 0.0780, 0.1840}
\definecolor{color8}{rgb}{0.0, 0.4078, 0.3412}

\newcommand{\logLogSlopeTriangle}[6]
{
							
	\pgfplotsextra
	{
		\pgfkeysgetvalue{/pgfplots/xmin}{\xmin}
		\pgfkeysgetvalue{/pgfplots/xmax}{\xmax}
		\pgfkeysgetvalue{/pgfplots/ymin}{\ymin}
		\pgfkeysgetvalue{/pgfplots/ymax}{\ymax}
		
				\pgfmathsetmacro{\xArel}{#1-#2}
		\pgfmathsetmacro{\yArel}{#3}
		\pgfmathsetmacro{\xBrel}{#1}
		\pgfmathsetmacro{\yBrel}{\yArel}
		\pgfmathsetmacro{\xCrel}{\xArel}
				
		\pgfmathsetmacro{\lnxB}{\xmin*(1-(#1-#2))+\xmax*(#1-#2)} 		\pgfmathsetmacro{\lnxA}{\xmin*(1-#1)+\xmax*#1} 		\pgfmathsetmacro{\lnyA}{\ymin*(1-#3)+\ymax*#3} 		\pgfmathsetmacro{\lnyC}{\lnyA+#5/#4*(\lnxA-\lnxB)}
		\pgfmathsetmacro{\yCrel}{\lnyC-\ymin)/(\ymax-\ymin)} 		
				\coordinate (A) at (rel axis cs:\xArel,\yArel);
		\coordinate (B) at (rel axis cs:\xBrel,\yBrel);
		\coordinate (C) at (rel axis cs:\xCrel,\yCrel);
		
				\draw[gray!65!black, line width=0.8pt, #6]   (A)-- node[pos=0.5,anchor=south] {\small #4}
		(B)-- 
		(C)-- node[pos=0.5,anchor=east] {\small #5}
		cycle;
	}
}

\begin{tikzpicture}[
  font=\large
  ]
	\begin{loglogaxis}[
			  xscale=1.2,
		xlabel={Mesh size $h$},
		ylabel={$\|u-\mathcal{U}(\psi_h)\|_{L^2(\Omega)}$},
  				xmajorgrids,
		ymajorgrids,
		clip mode=individual,
		legend style={fill=white, fill opacity=0.6, draw opacity=1, text opacity=1, draw = none},
		legend pos=south east,
		mark repeat=1,
			ytick = {1e-6, 1e-4,1e-2, 1},
            ymin = 1e-7,
            ymax = 1e-1,
		xtick = {0.0625,0.03125,0.015625,0.0078125},
		 xticklabels = {$h$,$h/2$,$h/4$,$h/8$}
		]
  
		\addplot[color7, mark=diamond*, very thick] table [x index={0}, y index={1}, col sep=comma] {circle/latent_error_0_Z2_0.0.csv};
		\addlegendentry{$p=0$};
	 
				        \addplot[color1, mark=diamond*, very thick] table [x index={0}, y index={1}, col sep=comma] {circle/latent_error_1_Z2_0.0.csv};
		\addlegendentry{$p=1$};
  
		\addplot[color2, mark=diamond*, very thick] table [x index={0}, y index={1}, col sep=comma] {circle/latent_error_2_Z2_0.0002.csv};
		\addlegendentry{$p=2$};

		\logLogSlopeTriangle{0.2}{0.1}{0.45}{1}{-2.0}{fill=color1};
				\logLogSlopeTriangle{0.215}{0.1}{0.85}{1}{-1.0}{fill=color7};

		\legend{}; 
    \end{loglogaxis}
\end{tikzpicture}
}
 \end{minipage}
 \begin{minipage}{0.32\textwidth}
         \resizebox{\textwidth}{!}{\pgfplotsset{width=7cm,compat=1.8}
\definecolor{color0}{rgb}{0.7843, 0.7843, 0.7843}
\definecolor{color1}{rgb}{0, 0.4470, 0.7410}
\definecolor{color2}{rgb}{0.8500, 0.3250, 0.0980}
\definecolor{color3}{rgb}{0.9290, 0.6940, 0.1250}
\definecolor{color4}{rgb}{0.7060, 0.3840, 0.7650}
\definecolor{color5}{rgb}{0.4660, 0.6740, 0.1880}
\definecolor{color6}{rgb}{0.3010, 0.7450, 0.9330}
\definecolor{color7}{rgb}{0.6350, 0.0780, 0.1840}
\definecolor{color8}{rgb}{0.0, 0.4078, 0.3412}

\newcommand{\logLogSlopeTriangle}[6]
{
							
	\pgfplotsextra
	{
		\pgfkeysgetvalue{/pgfplots/xmin}{\xmin}
		\pgfkeysgetvalue{/pgfplots/xmax}{\xmax}
		\pgfkeysgetvalue{/pgfplots/ymin}{\ymin}
		\pgfkeysgetvalue{/pgfplots/ymax}{\ymax}
		
				\pgfmathsetmacro{\xArel}{#1-#2}
		\pgfmathsetmacro{\yArel}{#3}
		\pgfmathsetmacro{\xBrel}{#1}
		\pgfmathsetmacro{\yBrel}{\yArel}
		\pgfmathsetmacro{\xCrel}{\xArel}
				
		\pgfmathsetmacro{\lnxB}{\xmin*(1-(#1-#2))+\xmax*(#1-#2)} 		\pgfmathsetmacro{\lnxA}{\xmin*(1-#1)+\xmax*#1} 		\pgfmathsetmacro{\lnyA}{\ymin*(1-#3)+\ymax*#3} 		\pgfmathsetmacro{\lnyC}{\lnyA+#5/#4*(\lnxA-\lnxB)}
		\pgfmathsetmacro{\yCrel}{\lnyC-\ymin)/(\ymax-\ymin)} 		
				\coordinate (A) at (rel axis cs:\xArel,\yArel);
		\coordinate (B) at (rel axis cs:\xBrel,\yBrel);
		\coordinate (C) at (rel axis cs:\xCrel,\yCrel);
		
				\draw[gray!65!black, line width=0.8pt, #6]   (A)-- node[pos=0.5,anchor=south] {\small #4}
		(B)-- 
		(C)-- node[pos=0.5,anchor=east] {\small #5}
		cycle;
	}
}

\begin{tikzpicture}[
  font=\large
  ]
	\begin{loglogaxis}[
			  xscale=1.2,
		xlabel={Mesh size $h$},
		ylabel={$\|\bm q - \bm q_h\|_{L^2(\Omega)}$},
  				xmajorgrids,
		ymajorgrids,
		clip mode=individual,
		legend style={fill=white, fill opacity=0.6, draw opacity=1, text opacity=1, draw = none},
		legend pos=south east,
		mark repeat=1,
			ytick = { 1e-6, 1e-4,1e-2, 1},
            ymin = 1e-4,
            ymax = 1e-1,
		xtick = {0.0625,0.03125,0.015625,0.0078125},
		 xticklabels = {$h$,$h/2$,$h/4$,$h/8$}
		]
  
		\addplot[color7, mark=diamond*, very thick] table [x index={0}, y index={1}, col sep=comma] {circle/flux_error_0_Z2_0.0.csv};
		\addlegendentry{$p=0$};
	 
				        \addplot[color1, mark=diamond*, very thick] table [x index={0}, y index={1}, col sep=comma] {circle/flux_error_1_Z2_0.0.csv};
		\addlegendentry{$p=1$};
  
		\addplot[color2, mark=diamond*, very thick] table [x index={0}, y index={1}, col sep=comma] {circle/flux_error_2_Z2_0.0002.csv};
		\addlegendentry{$p=2$};

	\logLogSlopeTriangle{0.215}{0.1}{0.35}{1}{-1.5}{fill=color1};
		\logLogSlopeTriangle{0.215}{0.1}{0.65}{1}{-1.0}{fill=color7};

	 \legend{}; 	\end{loglogaxis}
\end{tikzpicture}
}
 \end{minipage}
  \end{center}
\caption{ (\Cref{example:spherical}). Computed $L^2$ errors and rates of the approximations $(u_h, \mathcal{U}(\psi_h), \bm q_h)$ of \Cref{alg:main_alg_disc} with different choices for $\mathcal{U}$.  (Top row: $\mathcal U(\psi) = \underline u + \exp(\psi)$. Bottom row: $ \mathcal U(\psi) = \underline u + \ln(1+\exp(\psi))$.)  The algorithm is terminated once $\|u_h^{k} -u_h^{k-1}\|_{L^2(\Omega)} < 10^{-6}$ and the tolerance for the Newton solver is set to $10^{-10}$.  The coarsest mesh size $h\approx 0.058$. For $p=0$ and $p=1$, we set $\epsilon_1 = \epsilon_2 = 0$. For $p=2$, we set  $\epsilon_1 =0 , \epsilon_2 = 2e-04$.  In this example, we set $\alpha^k = 1$. }
\label{fig:rates_ex2}
\end{figure}
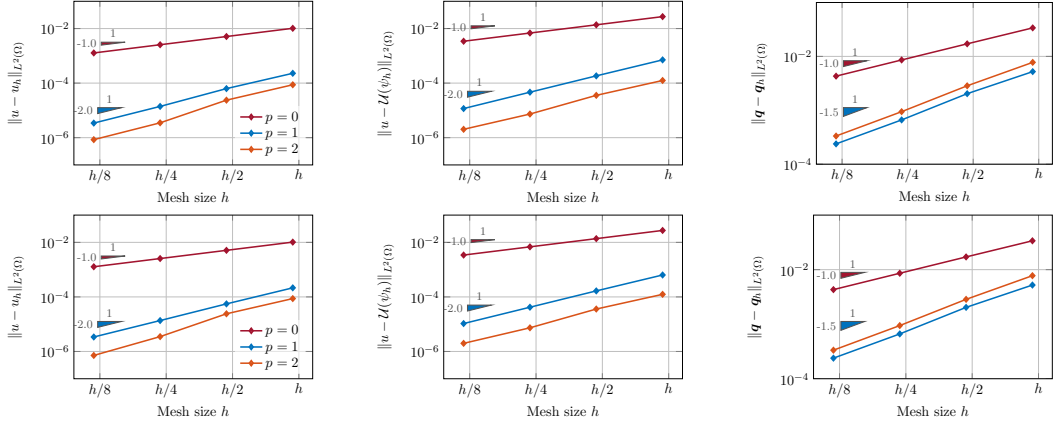

We now investigate the number of linear solves required per proximal iteration $k$ for different stopping criteria for Newton's method.  Here, we choose $\alpha^k = 1$, $p=1$, and $\mathcal{U}(\psi) = \underline u + \exp(\psi)$. We use a quadrature rule with the element vertices and we set $\epsilon_1= \epsilon_2 = 0$ in \eqref{eq:extra_stablization}. The algorithm is stopped when $\|u_h^k-u_h^{k-1}\|_{L^2(\Omega)} < 10^{-6}$. 

In \Cref{table:linearized}, we use only one linear solve per proximal step and see that the required number of subproblems does not change as the mesh is refined. However, the approximation $\mathcal{U}(\psi_h)$ stops converging for the last mesh refinement in this case. Decreasing the tolerance on the successive difference of iterates to $\|u_h^k-u_h^{k-1}\|_{L^2(\Omega)} < 10^{-8}$ fixes this issue but requires more proximal steps.  In \Cref{table:accurate_Newton} and \Cref{table:in_between},  we allow for multiple linear solves per nonlinear problem, stopping Newton's method when the square root of the inner product between the residual and the linear update reaches a certain tolerance  \cite{schoberl2014c++}. In \Cref{table:accurate_Newton}, we set this tolerance to $10^{-10}$ and observe that the number of linear solves does not change with mesh refinement, and $\mathcal{U}(\psi_h)$ converges with $h$. Finally, in \Cref{table:in_between}, we adaptively update the tolerance for Newton's method based on the successive difference between the two proximal iterates; i.e, we set the Newton tolerance to be $\min(0.1,\|u_h^k - u_h^{k-1}\|)$.  Here, fewer linear solvers are needed, and $\mathcal{U}(\psi_h)$ still converges with $h$. 
This leads us to conclude that future research is required to optimize the nonlinear solvers used in proximal Galerkin methods.
\begin{table}
\footnotesize{
\begin{tabular}{||c|c|c|c|c|c||}
\hline
\multicolumn{6}{|| c || }{$\|A^{-1/2} (\bm q-\bm q^k_h)\|_{L^2(\Omega)}$ } \\
\hline
$k$ & Linear solves &  $h$ & $h/2$ & $h/4$  & $h/8$ \\
\hline  
1  &  1 & 5.163e-01  &5.163e-01 &5.163e-01& 5.163e-01 \\ 
\hline 
2 &1  &4.543e-01 &4.543e-01&  4.543e-01 & 4.543e-01 \\
\hline 
3 & 1   & 2.567e-01  &2.567e-01&  2.568e-01& 2.568e-01 \\  
\hline
\vdots  & \vdots  & \vdots   & \vdots   &\vdots  &\vdots   \\ 
\hline 
\multicolumn{2}{||c|}{Total iterations $k$} & 14 & 13 & 13 & 13  \\ 
\hline 
\multicolumn{2}{||c|}{Total linear solves} & 14 & 13 & 13 & 13 \\ 
\hline
\multicolumn{2}{||c|}{Final error: $\|A^{-1/2}(\bm q - \bm q_h^k)\|_{L^2(\Omega)}$} &5.240e-03 & 2.054e-03& 6.852e-04 & 2.774e-04 \\ 
\hline 
\multicolumn{2}{||c|}{Final error: $\|u  - \mathcal{U}(\psi_h^k)\|_{L^2(\Omega)}$} &7.114e-04 & 1.847e-04& 4.660e-05 & 3.062e-03 \\ 
\hline
\end{tabular}
}
\caption{ (One Newton step per proximal step). Number of linear solves  needed per proximal Galerkin iteration $k$ and the corresponding error $\|A^{-1/2} (\bm q- \bm q_h^k)\|_{L^2(\Omega)}$.  }
\label{table:linearized}
\end{table}
\begin{table}
\footnotesize{
\begin{tabular}{||c|c|c|c|c|c||}
\hline
\multicolumn{6}{|| c || }{$\|A^{-1/2} (\bm q-\bm q^k_h)\|_{L^2(\Omega)}$ } \\
\hline
$k$ & Linear solves &  $h$ & $h/2$ & $h/4$  & $h/8$ \\
\hline  
1  &  7 & 3.408e-01  & 3.408e-01 &3.408e-01 & 3.408e-01 \\ 
\hline 
2 &6-7 &8.277e-02 &8.272e-02&   8.271e-02 & 8.271e-02\\
\hline 
3 & 7   & 2.298e-02  &2.247e-02&  2.242e-02& 2.242e-02 \\  
\hline
\vdots  & \vdots  & \vdots   & \vdots   &\vdots  &\vdots   \\ 
\hline 
\multicolumn{2}{||c|}{Total iterations $k$} & 12 & 12 & 12 & 11  \\ 
\hline 
\multicolumn{2}{||c|}{Total linear solves} & 59 & 62 & 63 & 63 \\ 
\hline
\multicolumn{2}{||c|}{Final error: $\|A^{-1/2}(\bm q - \bm q_h^k)\|_{L^2(\Omega)}$} &5.219e-03 & 2.034e-03& 6.633e-04 & 2.381e-04 \\ 
\hline 
\multicolumn{2}{||c|}{Final error: $\|u  - \mathcal{U}(\psi_h^k)\|_{L^2(\Omega)}$} &7.104e-04 & 1.842e-04& 4.640e-05 & 1.166e-05 \\ 
\hline
\end{tabular}
}
\caption{(Newton tolerance $= 10^{-10}$). Number of linear solves  needed per proximal Galerkin iteration $k$ and the corresponding error $\|A^{-1/2} (\bm q- \bm q_h^k)\|_{L^2(\Omega)}$.  }
\label{table:accurate_Newton}
\end{table}

\begin{table}
\footnotesize{
\begin{tabular}{||c|c|c|c|c|c||}
\hline
\multicolumn{6}{|| c || }{$\|A^{-1/2} (\bm q-\bm q^k_h)\|_{L^2(\Omega)}$ } \\
\hline
$k$ & Linear solves &  $h$ & $h/2$ & $h/4$  & $h/8$ \\
\hline  
1  &  3 & 3.389e-01   & 3.389e-01 &3.389e-01 & 3.389e-01 \\ 
\hline 
2 & 3  &8.059e-02 &8.055e-02&   8.054e-02 &  8.054e-02\\
\hline 
3 & 3    & 2.265e-02  &2.213e-02&  2.208e-02& 2.208e-02 \\  
\hline 
\vdots  & \vdots  & \vdots   & \vdots   &\vdots  &\vdots   \\ 
\hline 
\multicolumn{2}{||c|}{Total iterations $k$} & 12 & 12 & 12 & 11  \\ 
\hline 
\multicolumn{2}{||c|}{Total linear solves} & 27 & 27 & 28 & 27 \\ 
\hline
\multicolumn{2}{||c|}{Final error: $\|A^{-1/2}(\bm q - \bm q_h^k)\|_{L^2(\Omega)}$} &5.219e-03 & 2.034e-03& 6.633e-04 & 2.381e-04 \\ 
\hline 
\multicolumn{2}{||c|}{Final error: $\|u  - \mathcal{U}(\psi_h^k)\|_{L^2(\Omega)}$} &7.104e-04 & 1.842e-04& 4.641e-05 & 1.166e-05 \\ 
\hline
\end{tabular}
}
\caption{ (Newton tolerance = $\min(0.1, \|u_h^{k} - u_h^{k-1}\|_{L^2(\Omega)}))$. Number of linear solves needed per proximal Galerkin iteration $k$ and the corresponding error $\|A^{-1/2} (\bm q- \bm q_h^k)\|_{L^2(\Omega)}$.  }
\label{table:in_between}
\end{table}
\end{example}

\bibliographystyle{plain}
\bibliography{references}

\appendix    
\section{Proof of Lemma \ref{lemma:inf_sup}} \label{sec:proof_Ah}
\begin{proof}
First,  observe that 
\begin{equation}
\mathcal{A}_h ((\bm q_h, u_h, \hat u_h), (\bm q_h, u_h, \hat u_h))
=
\|A^{-1/2} \bm q_h\|^2_{L^2(\Omega)}  \label{eq:inf_sup_0}
\end{equation}
From \cite[Lemma 3.1]{egger2010hybrid}, see also \cite[lemma 3.1]{gao2018error} for a detailed proof, there exists $\bm{\tau}_h \in \bm \Sigma_h^p$ such that for any $T \in \mathcal{T}_h$ and for all $ \bm p \in \mathbb{P}^{p-1}(T)^d$ and $ q \in \mathbb{P}^{p}(\partial T), $
\begin{align}
\label{eq:def_test_infsup}
(\bm \tau_h, \bm p)_T + \langle \bm \tau_h \cdot \bm n , q \rangle_{\partial T} = (A^{-1} \nabla_h u_h, \bm p)_{T}
+  
\langle h_T^{-1} (\hat u_h -  u_h) , q \rangle_{\partial T},  
\end{align}
and 
\begin{equation}
\|\bm \tau_h \|^2_{\mathcal{T}_h } \leq c_I ( \|A^{-1} \nabla_h u_h\|^2_{\mathcal{T}_h} +  \| h_T^{-1/2} ( u_h - \hat u_h ) \|^2_{\partial \mathcal{T}_h} ).    \label{eq:norm_tau}
\end{equation}
Then, for $\gamma > 0$, we have that 
\begin{equation}
\mathcal{B}_h( \gamma \bm \tau_h, (u_h , \hat u_h)) = \gamma ( \|A^{-1/2} \nabla_h u_h\|^2_{\mathcal{T}_h} +  \|h_T^{-1/2} ( u_h - \hat u_h ) \|^2_{\partial \mathcal{T}_h} )
\end{equation}
It then follows that 
\begin{equation}
\begin{aligned}
\mathcal{A}_h ((\bm q_h, u_h, \hat u_h), (\gamma \bm \tau_h, 0,0 )) & - (A^{-1} \bm q_h, \gamma \bm \tau_h) \\ & =\gamma ( \|A^{-1/2} \nabla_h u_h\|^2_{\mathcal{T}_h} + \| h_T^{-1/2} ( u_h - \hat u_h ) \|^2_{\partial \mathcal{T}_h}).  \label{eq:inf_sup_1}
\end{aligned}
\end{equation}
Using Cauchy-Schwarz and Young's inequalities, we obtain that
 \begin{align*}
    | (A^{-1} \bm q_h, \gamma \bm \tau_h) |  & \leq  \gamma \|A^{-1/2}\|_{L^{\infty}(\Omega)} \|A^{-1/2}\bm q_h\|_{\mathcal T_h} \| \bm \tau_h \|_{\mathcal{T}_h} \\ & \leq \frac{1}{2} \|A^{-1/2}\bm q_h\|_{\mathcal T_h}^2    +\frac{\gamma^2 c_I c_A}{2} (\|A^{-1/2} \nabla_h u_h\|^2_{\mathcal{T}_h} + \|h_T^{-1/2} ( u_h - \hat u_h ) \|^2_{\partial \mathcal{T}_h})
 \end{align*}
 Thus, choosing $\gamma = 1/(c_Ic_A)$  and adding \eqref{eq:inf_sup_0} with \eqref{eq:inf_sup_1} yield 
 \begin{equation}
  \mathcal{A}_h ((\bm q_h, u_h, \hat u_h), (\gamma \bm \tau_h + \bm q_h , u_h,\hat u_h )) \gtrsim   \vvvert(\bm q_h, u_h, \hat u_h)\vvvert^2. \label{eq:inf_supAh} 
 \end{equation}
 In addition, observe that with \eqref{eq:norm_tau} and triangle inequality, we easily derive that 
 \begin{equation} 
 \vvvert (\gamma \bm \tau_h + \bm q_h , u_h,\hat u_h ) \vvvert \lesssim \vvvert (\bm q_h, u_h, \hat u_h) \vvvert. \label{eq:inf_sup_stab_bound}
 \end{equation}
 
      The proof of \eqref{eq:continuity_A} follows from applications of discrete trace inequalities, see \cite[Proposition 3.4]{egger2010hybrid}. We omit the details.  
 \end{proof} 
 
\section{Proof of \Cref{cor:err_estimate}}\label{appendix:error_estimate}

\begin{proof} (\Cref{cor:err_estimate})
Select $\bm r_h = \bm \Pi_h \bm \sigma (u)$ where $\bm \Pi_h$ is the local $L^2$-projection onto the space $\mathbb{P}^0(T)^d$. Note that $\bm r_h \in \bm \Sigma_h^p$ and $\nabla \cdot \bm r_h = 0$ \rami{in $T$ since $A$ is assumed to be piecewise constant}. By definition of the norm $\|\cdot\|_{\bm S^d}$ \eqref{eq:def_norm_Sd} and stability of the $L^2$-projection (averaging operator) in $L^{\rho}(T)$, we obtain that 
\begin{align}
\|\bm \sigma(u) - \bm r_h \|_{\bm S^d} &\leq h^s \|A^{-1/2} ( \bm \sigma(u) - \bm r_h) \|_{L^{\rho}(\Omega)} + h \|A^{-1/2} \nabla \cdot \bm \sigma(u)\|_{L^2(\Omega)} 
\\ & \lesssim h^{s} (|u|_{H^{1+s}(\Omega)} + h^{1-s}\|\nabla \cdot (\bm \sigma(u)) \|_{L^2(\Omega)}). \nonumber
\end{align}
In the above, we also used the embedding $H^{s}(\Omega)^d \hookrightarrow L^{\rho}(\Omega)^d$. Further, according to \cite[Lemma 7.1]{ern2017finite}, we also have a Poincar\'e inequality in $H^s(T)$: 
\begin{equation}
\|\bm \sigma(u) - \bm r_h \|_{L^2(T)} \lesssim h_T^s |\bm \sigma(u)|_{H^s(T)} \lesssim h_T^s \|u\|_{H^{1+s}(T)}. 
\end{equation}
For $v_h$, select $v_h = \mathcal{I}_h u $ a quasi-interpolant of $u$ \cite[Section 6]{ern2017finite} which satisfies 
\begin{equation}
  \|u - \mathcal{I}_h u \|_{L^2(T)} + h \|\nabla (u - \mathcal{I}_h u)\|_{L^2(T) } \lesssim h^{1+s} |u|_{H^{1+s}(\Delta_T)}, 
\end{equation}
where $\Delta_T$ is a macro-element. We may also select $\hat v_h = \mathcal{I}_h u \in M_{h,0}^p$ since it is single valued on interior facets and $\mathcal{I}_h u \vert_{\partial \Omega} = 0$ .  This implies that the third term in $\vvvert (\bm \sigma (u) - \bm r_h , u-v_h, u - \hat v_h)  \vvvert$ is zero. Collecting the above observations and summing over elements yield 
\begin{align}
\mathcal{E}(u) \lesssim  h^{s} (|u|_{H^{1+s}(\Omega)} + h^{1-s}\|\nabla \cdot (\bm \sigma(u)) \|_{L^2(\Omega)}). 
\end{align}
Proceeding, we bound $\mathcal{E}(\delta)$. Selecting $\varphi_h = \Pi_h \delta$, the $L^2$- projection onto $V_h^p$, we derive 
\begin{align}
\|\delta - \Pi_h \delta\|_{H^1(\mathcal{T}_h)^*}  = \sup_{w \in H^1(\mathcal{T}_h) } \frac{(\delta - \Pi_h \delta, w - \Pi_h w ) }{\|w\|_{\mathrm{DG}}} 
\end{align}
With the approximation properties of the $L^2$ projection, we  estimate 
\begin{equation}
    \|w -  \Pi_h w \|_{L^2(\Omega)}
    \leq 
    h \|\nabla_h  w\|_{\mathcal{T}_h} .    
\end{equation}
From the $L^2$ stability of $\Pi_h$, it follows that 
\[\|\delta - \Pi_h \delta\|_{H^1(\mathcal{T}_h)^*} \lesssim h \|\delta\|_{L^2(\Omega)}. \]
The second term in $\mathcal{E}(\delta)$ is bounded by approximation properties of the $L^2$-projection.
\end{proof}

\end{document}